\documentclass[14pt]{amsart}
\usepackage{amscd}
\usepackage{amsmath}
\usepackage{amsfonts}
\usepackage{bm}
\usepackage{amsfonts,amsmath,amssymb,amscd,bbm,amsthm,mathrsfs,dsfont}
\usepackage{mathrsfs}
\usepackage{pb-diagram}
\usepackage{color}
\usepackage{amssymb}
\usepackage{xypic}
\usepackage[dvips]{graphicx}

\usepackage[all]{xy}

\newtheorem{Thm}{Theorem}[section]
\newtheorem{Fac}[Thm]{Fact}
\newtheorem{Lem}[Thm]{Lemma}
\newtheorem{Def}[Thm]{Definition}
\newtheorem{Cor}[Thm]{Corollary}
\newtheorem{Prop}[Thm]{Proposition}
\newtheorem{Ex}[Thm]{Example}
\newtheorem{Rem}[Thm]{Remark}

\title[On Structure of cluster algebras of geometric type I]
{ On Structure of cluster algebras of geometric type I:\\ In view of sub-seeds and seed homomorphisms}

\author{Min Huang $\;\;\;\;\;\;$ Fang Li $\;\;\;\;\;\;$ Yichao Yang}
\address{Min huang
\newline Department
of Mathematics, Zhejiang University (Yuquan Campus), Hangzhou, Zhejiang
310027,  P.R.China}
\email{minhuang1989@hotmail.com}

\address{Fang Li
\newline Department
of Mathematics, Zhejiang University (Yuquan Campus), Hangzhou, Zhejiang
310027, P.R.China}
\email{fangli@zju.edu.cn}

\address{Yichao Yang
\newline D\'{e}partement de Math\'{e}matiques, Universit\'{e} de Sherbrooke, Sherbrooke, Qu\'{e}bec, Canada}
\email{yichao.yang@usherbrooke.ca}

\date{version of \today}

\newcommand{\lra}{\longrightarrow}

\newcommand{\ra}{\rightarrow}
\newcommand{\sdp}{\times\kern-.2em\vrule height1.1ex depth-.05ex}
\newcommand{\epi}{\lra \kern-.8em\ra}

 \normalbaselines
\begin{document}

\renewcommand{\thefootnote}{\alph{footnote}}
\setcounter{footnote}{-1} \footnote{\emph{ Mathematics Subject
Classification(2010)}:  13F60, 05E15, 20M10.}
\renewcommand{\thefootnote}{\alph{footnote}}
\setcounter{footnote}{-1} \footnote{ \emph{Keywords}: seed homomorphism, mixing-type sub-seed, rooted cluster morphism, sub-rooted cluster algebra, rooted cluster quotient algebra.}

\begin{abstract} Our motivation is to build a systematic method in order to investigate the structure of cluster algebras of geometric type.
  The method is given through the notion of mixing-type sub-seeds, the theory of seed homomorphisms and the view-point of gluing of seeds. As an application, for (rooted) cluster algebras,
   we completely classify rooted cluster subalgebras and characterize rooted cluster quotient algebras in detail. Also, we build the relationship between the categorification of a rooted cluster algebra and that of its rooted cluster subalgebras.

  Note that cluster algebras of geometric type studied here are of the sign-skew-symmetric case.
\end{abstract}

\maketitle
\bigskip
\tableofcontents
\section{Introduction and preliminaries}

Cluster algebras are commutative algebras that were introduced
by Fomin and Zelevinsky \cite{fz1} in order to give a
combinatorial characterization of total positivity and canonical
bases in algebraic groups.  The theory of cluster
algebras is related to numerous other fields. Since its introduction, the study on cluster algebras mainly involves intersection with Lie theory, representation theory of algebras, its combinatorial method (e.g. the periodicity issue) and categorification and the sub-class constructed from Riemannian surfaces and its topological setting, including the Teichm$\ddot{u}$ller theory.

The algebraic structure and properties of cluster algebras were originally studied in a series of articles \cite{fz1,fz2,[1],fz4} involving bases and the positivity conjecture. The positive conjecture has been proved by Lee and Schiffler in \cite{LS} in the skew-symmetric case and moreover, was claimed to be true by Kontsevich etc. in \cite{GHKK} in the skew-symmetrizable case.

In this work, we characterize cluster algebras through their internal structure. The categorical framework for cluster algebras is provided in \cite{ADS}. More
precisely, the so-called {\em rooted cluster
morphism} is introduced to characterize the relations among {\em rooted cluster algebras} with ``rooted" meaning fixed initial seeds. In particular, injective morphisms and surjective morphisms and isomorphisms are investigated for rooted cluster algebras in some special cases, including those from Riemanian surfaces in
\cite{ADS}.

In \cite{ADS}, the structure of a cluster algebras is discussed through its rooted cluster subalgebras and quotient algebras. However, as shown in the sequel, we find that its structure is determined in general by its sub-seeds, correspondingly by the so-called {\em sub-rooted cluster algebras}.  Thus, our view is different from that in \cite{ADS}.

In this article, we mainly focus to study the structure of rooted cluster algebras, including all rooted
cluster subalgebras and rooted cluster quotient algebras, via sub-seeds and seed homomorphisms. It partly was studied in
\cite{ADS, CZ} in some special cases. For this aim,  we propose a
systematic method to characterize rooted cluster subalgebras and rooted cluster quotient algebras.
In addition to the methods of sub-seeds and seed homomorphisms, the method of gluing of seeds for rooted cluster quotient algebras is also an important topic in our discussion. As an incidental result,  the partial answer of one problem on the rooted cluster morphism $\sigma_{x,1}$ in \cite{ADS} is given in our way.

As a new idea in this article, we introduce the so-called {\em (mixing-type) sub-seeds} and {\em seed homomorphisms}.

The concept of mixing-type subseeds is basic for us to discuss the structure of cluster algebras. Our motivation is to unify two extremes: freeze exchange variables or delete exchange/frozen variables, into a concept. It supports us a possibility to characterize the structure of a rooted cluster algebra. In fact, we have proved in (page 18, Thm 4.4) that any rooted cluster algebras can be expressed as such form.

In \cite{ADS} and \cite{Sa}, seed (anti-)isomorphisms and $\sigma$-similarity of seeds are defined, that are indeed consistent with isomorphisms of seeds, which are special cases of seed homomorphisms. We set up the corresponding structure of a sub-seed in a cluster algebra, which is called a {\em sub-rooted cluster algebra}.
The interesting fact is that seed homomorphisms are compatible with graph
homomorphisms in graph theory, which gives a possibility to establish a connection
between the cluster algebras theory and the graph theory, because a seed can be presented as a cluster quiver when the exchange matrix is skew-symmetric.

Our original motivation for introducing the concept of seed homomorphisms is to understand the structure of totally sign-skew-symmetric cluster algebras. In the sign-skew-symmetric case, many problems, e.g. positivity conjecture and $F$-polynomials, will become very difficult. So, such research in our paper is necessary. This new concept will be used in our further work \cite{HL2} on the positivity and $F$-polynomials for sign-skew-symmetric cluster algebras,
building on \cite{LS} and this present work.

As a preliminary application of our conclusions, we give a relation between the finite type/finite mutation
type of a rooted cluster algebra and that of its sub-rooted cluster algebras and  establish a
connection between rooted cluster sub-algebras and their monoidal categorification.

In summary, we list our main contributions in this article as follows:
\\
$\bullet$\; Build the theory of seed homomorphisms (Definition \ref{seedhom}) for its importance as a tool in this work.
\\
$\bullet$\; Introduce the notion of mixing-type sub-rooted cluster algebras via mixing-type sub-seeds and using it as the main tool, for a given rooted cluster algebra, we give the characterization  of rooted cluster subalgebras (Theorem \ref{rooted cluster subalgebra})
 and that of pure sub-cluster algebras as a class of cluster quotient algebras in the acyclic case obtained via specialisation (Theorem \ref{cluster quo.}).
 \\
$\bullet$\;  The method of gluing frozen variables is used effectively to characterize general surjective rooted cluster morphisms. Concretely, any surjective rooted cluster morphism determines uniquely a noncontractible surjective rooted cluster morphism, which can be written as a composition of a rooted cluster isomorphism and a series of surjective canonical morphisms via gluing pairs of frozen variables step-by-step (Proposition \ref{inducedhom}, Theorem \ref{maindecomp}).
  \\
$\bullet$\;  As another application of the method of mixing-type sub-seeds, we build the relationship between the categorification of a rooted cluster algebra and that of its rooted cluster subalgebra (Theorem \ref{subcat}).

The organization of this article still contains the following further contents.

In the next part of this section, we explain the notions and notations about cluster algebras.

It is proved in Section 3 that a rooted cluster isomorphism is equivalent to an initial seed isomorphism (Proposition \ref{basiclem}).

In Section 4, the characterization  of rooted cluster subalgebras (Theorem \ref{rooted cluster subalgebra}) means that all rooted cluster subalgebras of a rooted cluster algebra is a sub-class of its mixing-type sub-rooted cluster algebras. It is worth to mention that this result has been independently found in \cite{CZ} in the other form of characterization.
In this section, we also show that a pure cluster subalgebra is always a rooted cluster subalgebra (Proposition \ref{Lem rooted}) and then give the  characterization of  proper rooted cluster subalgebras (Corollary \ref{0 sub matrix}).

As an application of Section 4, in Section 5, we calculate the number of non-trivial proper rooted cluster subalgebras in a rooted cluster algebra.

In Section 6, as a corollary of Theorem \ref{cluster quo.}, it is proved that a sub-rooted cluster algebra of a rooted cluster algebra of finite type (respectively, finite mutation type) is also of finite type (respectively, finite mutation type) (Corollary \ref{cor6.8}).

\bigskip

The successive works \cite{HL} and \cite{HL2} are based on this paper.

\bigskip

In this paper, we always consider totally sign-skew-symmetric cluster algebras of geometric type introduced in \cite{fz1}\cite{fz2}, as mentioned as follows.
\bigskip

The original definition of cluster algebra given in \cite{fz1} is in terms of exchange pattern. We recall the equivalent definition in terms of seed mutation in \cite{fz2}; for more details, refer to \cite{GSV, fz1, fz2}.

An $n\times n$ integer matrix $A=(a_{ij})$ is called {\bf
sign-skew-symmetric} if either $a_{ij}=a_{ji}=0$ or $a_{ij}a_{ji}<0$ for any $1\leq i,j\leq
n$.

An $n\times n$ integer matrix $A=(a_{ij})$ is called {\bf
skew-symmetric} if $a_{ij}=-a_{ji}$ for all $1\leq i,j\leq n$.

An $n\times n$ integer matrix $A=(a_{ij})$ is called {\bf
$D$-skew-symmetrizable} if $d_ia_{ij}=-d_ja_{ji}$ for all $1\leq
i,j\leq n$, where $D$=diag$(d_i)$ is a diagonal matrix with all
$d_i\in \mathbb{Z}_{\geq 1}$.

Let $\widetilde{A}$ be an $n\times (n+m)$ integer matrix whose
principal part, denoted as $A$, is the $n\times n$ submatrix formed
by the first $n$-rows and the first $n$-columns. The entries of
$\widetilde A$ are written by $a_{xy}$, $x\in X$ and $y\in
\widetilde{X}$. We say $\widetilde A$ to be {\bf
sign-skew-symmetric} (respectively, {\bf skew-symmetric}, $D$-{\bf
skew-symmetrizable}) whenever $A$ possesses this property.

For two $n\times (n+m)$ integer
 matrices $A=(a_{ij})$ and
$A'=(a'_{ij})$, we say that $A'$ is
obtained from $A$ by a {\bf matrix mutation $\mu_i$} in direction $i, 1\leq i\leq n$, represented as $A'=\mu_i(A)$, or say that $A$ and $A'$ are {\bf mutation
equivalent}, represented as $A\simeq A'$, if
\begin{equation}\label{matrixmutation}
a'_{jk}=\left\{\begin{array}{lll} -a_{jk},& \text{if}
~~j=i \ \text{or}\;
k=i;\\a_{jk}+\frac{|a_{ji}|a_{ik}+a_{ji}|a_{ik}|}{2},&
\text{otherwise}.\end{array}\right.\end{equation}

It is easy to verify that $\mu_i(\mu_i(A))=A$. The
skew-symmetric/symmetrizable property of matrices is invariant under
mutations. However, the sign-skew-symmetric property is not so. For
this reason, a sign-skew-symmetric matrix $A$ is called {\bf totally
sign-skew-symmetric} if any matrix, that is mutation equivalent to
$A$, is sign-skew-symmetric.

 Give a field $\mathds F$ as an extension of the rational number field $\mathds Q$, assume that $u_1,\cdots,u_n,x_{n+1},\cdots,x_{n+m}\in \mathds F$ are $n+m$ algebraically independent over $\mathds Q$ for a positive integer $n$ and a non-negative integer $m$ such that $\mathds F=\mathds{Q}(u_1,\cdots,u_n,x_{n+1},\cdots,x_{n+m})$, the field of rational functions in the  set $\widetilde X=\{u_1,\cdots,u_n,x_{n+1},\cdots,x_{n+m}\}$ with coefficients in $\mathds{Q}$.

A {\bf seed} in $\mathds F$ is a triple $\Sigma=(X, X_{fr},\widetilde B)$, where\\
(a)~$X=\{x_1,\cdots x_n\}$ is a transcendence basis of $\mathds F$ over the fraction field of $\mathds Z[x_{n+1},\cdots,x_{n+m}]$, which is called a {\bf cluster}, whose each $x\in X$ is called a {\bf cluster variable} (see \cite{fz2});\\
(b)~$X_{fr}=\{x_{n+1},\cdots,x_{n+m}\}$ are called the {\bf frozen cluster} or, say, the {\bf frozen part} of $\Sigma$ in $\mathds F$, where all $x\in X_{fr}$ are called  {\bf stable (cluster) variables} or {\bf frozen (cluster) variables};\\
(c)~$\widetilde{X}=X\cup X_{fr}$ is called a {\bf extended cluster};\\
(d)~$\widetilde{B}=(b_{xy})_{x\in X,y\in \widetilde{X}}=(B\;B_1)$ is a $n\times (n+m)$ matrix over $\mathds{Z}$ with rows and columns indexed
by $X$ and $\widetilde{X}$, which is totally sign-skew-symmetric. The $n\times n$ matrix $B$ is called the {\bf exchange matrix} and
$\widetilde B$ the {\bf extended exchange matrix} corresponding to the seed $\Sigma$.

 In a seed $\Sigma=(X,X_{fr},\widetilde B)$, if $X=\emptyset$, that is,
$\widetilde{X}=X_{fr}$, we call the seed a {\bf trivial seed}.

Given a seed $\Sigma=(X,X_{fr},\widetilde{B})$ and $x,y\in
\widetilde{X}$, we say $(x,y)$ is a {\bf connected pair} if
$b_{xy}\neq 0$ or $b_{yx}\neq 0$ with $\{x,y\}\cap X\neq \emptyset$. A seed $\Sigma$ is defined to be {\bf connected} if for any $x,y\in \widetilde{X}$,
there exists a sequence of variables
$(x=z_0,z_1,\cdots,z_s=y)\subseteq \widetilde{X}$ such that
$(z_i,z_{i+1})$ are connected pairs for all $0\leq i\leq s-1$.

Let $\Sigma=(X,X_{fr},\widetilde B)$ be a seed in $\mathds F$ with $x\in X$, the {\bf mutation} $\mu_x$ of $\Sigma$ at $x$ is defined satisfying $\mu_x(\Sigma)=(\mu_x(X), X_{fr}, \mu_{x}(\widetilde B))$ such that

(a)~ The {\bf adjacent cluster} $\mu_{x}(X)=\{\mu_{x}(y)\mid y\in X\} $, where  $\mu_{x}(y)$ is given by the {\bf exchange relation}
\begin{equation}\label{exchangerelation}
\mu_{x}(y)=\left\{\begin{array}{lll} \frac{\prod\limits_{t\in \widetilde{X}, b_{xt}>0}t^{b_{xt}}+\prod\limits_{t\in \widetilde{X}, b_{xt}<0}t^{-b_{xt}}}{x},& \text{if}
~~y=x;\\y,& \text{if}~~ y\neq x. \end{array}\right.\end{equation}
This new variable $\mu_{x}(x)$ is also called a {\em new} {\bf cluster
variable}.

(b)~ $\mu_{x}(\widetilde B)$ is obtained from B by applying the matrix mutation in direction
$x$ and then relabeling one row and one column by replacing $x$ with $\mu_x(x)$.

It is easy to see that the mutation $\mu_x$ is an involution, i.e., $\mu_{\mu_x(x)}(\mu_x(\Sigma))=\Sigma$.

Two seeds $\Sigma'$ and $\Sigma''$ in $\mathds F$ are called {\bf mutation equivalent} if there exists a sequence of mutations $\mu_{y_1},\cdots,\mu_{y_s}$ such that $\Sigma''=\mu_{y_s}\cdots\mu_{y_1}(\Sigma')$.
 Trivially, the mutation equivalence gives an equivalence relation on the set of seeds in $\mathds F$.

Let $\Sigma$ be a seed in $\mathds F$. Denote by $\mathcal S$ the set of all seeds mutation equivalent to $\Sigma$. In particular, $\Sigma\in \mathcal S$. For any $\bar \Sigma\in\mathcal S$, we have $\bar \Sigma=(\bar X,X_{fr},\widetilde {\bar B})$. Denote $\mathcal X=\cup_{\bar\Sigma\in\mathcal S}\bar X$.

\begin{Def}\label{clusteralgebra}
 Let $\Sigma$ be a seed in $\mathds F$. The {\bf cluster algebra} $\mathcal A=\mathcal A(\Sigma)$, associated with $\Sigma$, is defined to be the $\mathds Z[x_{n+1},\cdots,x_{n+m}]$-subalgebra of $\mathds F$ generated by $\mathcal X$. $\Sigma$ is called the {\bf initial seed} of $\mathcal A$.
\end{Def}

This notion of cluster algebra was given in \cite{fz1}\cite{fz2}, where it is called the {\bf cluster algebra of geometric type} as a special case of general cluster algebras.

Such kind of cluster algebras is considered as the most important one, with respect to additive categorification, etc., in the theory of cluster algebras based on the views of references \cite{fz2} \cite{CK} \cite{GSV} \cite{GLS}.
It accords with many important examples, such as those given in \cite{fz1}\cite{GSV}\cite{K1}\cite{Sc}, including those constructed from the coordinate rings of many varieties, e.g. from Grassmannians \cite{GSV} and algebraic groups \cite{fz1}. They supply with a close connection between the theory of cluster algebras and representation theory, see \cite{BMRRT}\cite{K1}\cite{GLS}, etc.

\section{Seed homomorphisms and some elementary properties}

By definition, cluster algebras are determined by original seeds and their mutations. We find in the next section that the relations between cluster algebras can be restricted to their seeds so as to obtain the relations between them. Motivated by this fact and our discussion in the sequel, we now introduce the so-called {\em seed
 homomorphism}.

 For the initial seed $\Sigma=(X,X_{fr},\widetilde{B})$ of a cluster algebra $\mathcal A$ and two pairs $(x,y)$ and $(z,w)$ with $x,z\in X$ and
 $y,w\in\widetilde{X}$, we say that $(x,y)$ and $(z,w)$ are {\bf adjacent
pairs} if $b_{xz}\neq 0$ or $x=z$.

\begin{Def}\label{seedhom}
Let $\Sigma=(X,X_{fr},\widetilde{B})$ and $\Sigma'=(X',X'_{fr},\widetilde{B'})$ be two
seeds.

(i)~ A map $f$ from $\widetilde{X}$ to $\widetilde{X'}$ is called a {\bf seed homomorphism} from the seed $\Sigma$ to the seed $\Sigma'$ if
 (a) ~ $f(X)\subseteq X'$ and
 (b) ~ for any adjacent pairs $(x,y)$ and $(z,w)$ with $x,z\in X$,
 $y,w\in\widetilde{X}$, it holds that
 \begin{equation}
(b'_{f(x)f(y)}b_{xy})(b'_{f(z)f(w)}b_{zw})\geq 0 \;\;\;\;\;\;\;\text{and}\;\;\;\;\;\;\; |b'_{f(x)f(y)}|\geq |b_{xy}|.
\end{equation}

(ii)~ A seed homomorphism
$f:\Sigma\rightarrow\Sigma'$  is called a {\bf positive seed
homomorphism} if $b'_{f(x)f(y)}b_{xy}\geq 0$ for all $x\in X$ and $y\in
\widetilde{X}$. In contrast, a seed homomorphism $f$ is called a {\bf negative seed
homomorphism} if $b'_{f(x)f(y)}b_{xy}\leq 0$ for all $x\in X$ and $y\in
\widetilde{X}$.
\end{Def}

For seed homomorphisms $f:\Sigma\rightarrow\Sigma'$ and
$g:\Sigma'\rightarrow\Sigma''$, define their composition $gf:\Sigma\rightarrow\Sigma''$
satisfying that $gf(x)=g(f(x))$ for all $x\in \widetilde{X}$. Then we can define the {\bf seed category}, denoted as
\textbf{Seed}, to be the category whose objects are all seeds and
whose morphisms are all seed homomorphisms with composition defined as
above.

A quiver $\Gamma$ is called a {\bf cluster quiver} if it is a finite quiver  with no loops and no cycles of lengths 2 (see \cite{fosth}\cite{LLY1}\cite{LLY2}).
The meaning of this class of quivers follows the fact that cluster quivers can be corresponding one-to-one with the skew-symmetric integer square matrices, generating cluster algebras without frozen variables.

In fact, our idea of seed homomorphisms is original from homomorphisms of direct graphs.

As given in \cite{HN}, recall that for two quivers (said as digraphs
in graph theory) $Q$ and $P$ with the vertex sets $Q_0$ and $P_0$, a {\bf
quiver homomorphism} $f$ from $Q$ to $P$, written as $f: Q\rightarrow P$,  is
a mapping $f: Q_{0}\rightarrow P_{0}$ such that there is an arrow
from $f(u)$ to $f(v)$ in $P$ whenever there is an arrow from $u$ to
$v$ in $Q$ for $u,v\in Q_0$.

Following this quiver homomorphism, a homomorphism $f$ of quivers
from $Q$ to $P$ is called a {\bf cluster quiver homomorphism} if
$f(Q_{0,ex})\subseteq P_{0,ex}.$

In the case for {\em skew-symmetric} seeds, we know the one-to-one
correspondence between seeds and cluster quivers. For two connected
cluster quivers $Q$ and $Q'$ and their seeds $\Sigma=\Sigma(Q)$ and
$\Sigma'=\Sigma(Q')$,  a positive (respectively, negative) seed
homomorphism $f_s: \Sigma\rightarrow \Sigma'$ corresponds to a
cluster quiver homomorphism (respectively, anti-homomorphism) $f_c:
Q\rightarrow Q'$.

 In fact, the condition (b) of the definition of seed homomorphism
means that $b'_{f_s(x)f_s(y)}b_{xy}\geq 0$ for all $x\in X$ and $y\in
\widetilde{X}$ or $b'_{f_s(x)f_s(y)}b_{xy}\leq 0$ for all $x\in X$ and
$y\in \widetilde{X}$. Also, from (b),
$|b'_{f_s(x)f_s(y)}|\geq|b_{xy}|$, which follows that we can define
a cluster quiver homomorphism or anti-homomorphism $f_c$ from the
corresponding cluster quiver $Q$ to the other one $Q'$ with
$f_c(x)=f_s(x)$ for any vertex $x$ in $Q$ such that for any vertices
$x, y$ of $Q$ if there is an arrow from $x$ to $y$, then there is an
arrow from $f_c(x)$ to $f_c(y)$ for all $x\in X$ and $y\in \widetilde{X}$
or from $f_c(y)$ to $f_c(x)$ for all $x\in X$ and $y\in \widetilde{X}$.
It implies that $f_c$ is a cluster quiver homomorphism or
anti-homomorphism.

According to the fact above, we can see seed homomorphisms as an improvement of quiver homomorphisms in the theory of graphs, in order to be useful for
  skew-symmetrizable seeds or more generally, totally sign-skew-symmetrizable seeds.
  In this view, the homomorphism method in graph theory will have a
natural influence on our study in this article and more further work.

 \begin{Def}\label{mixingseed}
 Let $\Sigma=(X,X_{fr},\widetilde{B})$ be a seed with $\widetilde{B}$
an $n\times (n+m)$ totally sign-skew-symmetric  integer matrix. Assume $I_{0}$
is a subset of $X$, $I_{1}$ a subset of $\widetilde{X}$ with
$I_{0}\cap I_{1}=\emptyset$ and $I_{1}=I_{1}'\cup I_{1}''$ for
$I_{1}'=X\cap I_{1}$ and $I_{1}''=X_{fr}\cap I_{1}$. Denoting
$X'=X\backslash (I_{0}\cup I_{1}')$,
$\widetilde{X'}=\widetilde{X}\backslash I_{1}$ and $\widetilde{B'}$
as a $\sharp X' \times \sharp \widetilde{X'}$-matrix with
$b_{xy}'=b_{xy}$ for any $x\in X'$ and $y\in \widetilde{X'}$, one can
define the new seed $\Sigma_{I_{0},I_{1}}=(X',(X_{fr}\cup I_0)\setminus I_1,\widetilde{B'})$, which
is called a {\bf mixing-type sub-seed} or, say, {\bf $(I_{0},I_{1})$-type sub-seed}, of the seed
$\Sigma=(X,X_{fr},\widetilde{B})$.
\end{Def}

\begin{Fac} An $(I_{0}',I'_{1})$-type
 sub-seed of an $(I_{0},I_{1})$-type sub-seed
of a seed $\Sigma$ is a mixing-type sub-seed of
$\Sigma$. That is, for a seed $\Sigma$,
$(\Sigma_{I_0,I_1})_{I'_0,I'_1}=\Sigma_{I'_0\cup (I_0\setminus
I'_1),I_1\cap I'_1}$.
\end{Fac}

Furthermore, we have $X_{fr}'=I_{0}\cup(X_{fr}\backslash I_{1}'')$.

To obtain an $(I_{0},I_{1})$-sub-seed from the initial seed is equivalent to saying freeze the cluster variables in $I_{0}$ and delete the cluster variables in $I_{1}$.

First, we discuss two special cases of mixing-type sub-seeds of a seed $\Sigma$.

{\bf Case $1$: $I_{1}=\emptyset$.} \; That is, we only freeze the variables in $I_{0}$ that are original mutable and do not delete any variables.

In this case, we have the sub-seed $\Sigma_{I_{0},\emptyset}=(X',X_{fr}\cup I_0,\widetilde{B_0})$ with cluster $X'=X\backslash I_0$ consisting of mutable variables and extended cluster
$\widetilde{X'}=\widetilde{X}$.
Since $\widetilde{B}$ is sign-skew-symmetric, it follows that $\widetilde{B_0}$ is also sign-skew-symmetric.
The frozen variables of the sub-seed $\Sigma_{I_{0},\emptyset}$ form the set
$X_{fr}'=\widetilde{X}\backslash X'=(X\backslash X')\cup X_{fr}=I_{0}\cup X_{fr}$.

We call this sub-seed $\Sigma_{I_{0},\emptyset}=(X',X_{fr}\cup I_0,\widetilde{B_0})$  a {\bf pure sub-seed} of the seed $\Sigma=(X,X_{fr},\widetilde{B})$.

{\bf Case $2$: $I_{0}=\emptyset$.}\; That is, we only delete the variables in $I_1$ while the remaining variables remain unchanged and do not freeze any exchangeable variables.

In this case, we have the sub-seed $\Sigma_{\emptyset,I_{1}}=(X'',X''_{fr},\widetilde{B_1})$ with $X''=X\backslash I_1$, $X''_{fr}=X_{fr}\backslash I_1$ and $\widetilde{X''}=X''\cup X''_{fr}$.
$\widetilde{B_1}$ is sign-skew-symmetric since $\widetilde{B}$ is so.

We call this sub-seed $\Sigma_{\emptyset,I_{1}}=(X'',X''_{fr},\widetilde{B_1})$ a {\bf partial sub-seed} of the seed $\Sigma=(X,X_{fr},\widetilde{B})$.

For two seeds $\Sigma_{1}=(X_{1},(X_1)_{fr},\widetilde{B^{1}})$ and
$\Sigma_{2}=(X_{2},(X_2)_{fr},\widetilde{B^{2}})$, if there exists $($possibly
empty$)$ $\Delta_{1}\subseteq (X_{1})_{fr}$ and $\Delta_{2}\subseteq
(X_{2})_{fr}$ such that $|\Delta_{1}|= |\Delta_{2}|$, then
$\Sigma_{1}$ and $\Sigma_{2}$ are said to be {\bf glueable} along
$\Delta_{1}$ and $\Delta_{2}$. Let $\Delta$ be a family of
undeterminates in bijection with $\Delta_{1}$ and $\Delta_{2}$.

Recall in \cite{ADS} that the {\bf amalgamated sum} of $\Sigma_1$
and $\Sigma_2$ along $\Delta_{1}$ and $\Delta_{2}$ is defined as
$\Sigma=(X,X_{fr},\widetilde{B})$, where
$\widetilde{X}=(\widetilde{X_{1}}\backslash \Delta_{1})\cup
(\widetilde{X_{2}}\backslash \Delta_{2})\cup \Delta$, $X=X_{1}\cup
X_{2}$ and the matrix $\widetilde{B}$ is defined as:
\begin{equation}\label{union} \widetilde{B}=\left(\begin{array}{ccccc}
B_{11}^{1} & 0 & B_{12}^{1} & 0 & B_{13}^{1} \\
0 & B_{11}^{2} & 0 & B_{12}^{2} & B_{13}^{2}
\end{array}\right)~~~~~~~~~\text{with}\; \widetilde{B^{i}}=\left(\begin{array}{ccccc}
B_{11}^{i} &  B_{12}^{i} & B_{13}^{i}
\end{array}\right), i=1,2.
\end{equation}

We denote the amalgamated sum as the notations
$\Sigma=\Sigma_{1}\amalg
_{\Delta_{1},\Delta_{2}} \Sigma_{2}$.

In particular, when $\Delta_1$ and $\Delta_2$ are empty sets, we call this amalgamated sum $\Sigma=\Sigma_{1}\amalg
_{\Delta_{1},\Delta_{2}} \Sigma_{2}$ the {\bf union seed} of $\Sigma_1$ and $\Sigma_2$, denoted as $\Sigma=\Sigma_1\sqcup\Sigma_2$.

For a given seed $\Sigma$, if $\Sigma_1$ and $\Sigma_2$ are partial
sub-seeds of type $(\emptyset,I_1)$ of $\Sigma$ such that $X_1\cap X_2=\emptyset$, we replace $\Sigma_1\amalg_{(X_{1})_{fr}\cap (X_{2})_{fr},(X_{1})_{fr}\cap
(X_{2})_{fr}}\Sigma_2$ by the
notation $\Sigma_1\amalg\Sigma_2$. For subseeds $\Sigma_1,\Sigma_2,$ and $\Sigma_3$ of type $(\emptyset,I_1)$  of
$\Sigma$, we have $(\Sigma_1\amalg
\Sigma_2)\amalg\Sigma_3=\Sigma_1\amalg(\Sigma_2\amalg\Sigma_3)$, that is, the associative law holds for the amalgamated sum.

In fact, $(\Sigma_1\amalg
\Sigma_2)\amalg\Sigma_3$ has the set of exchangeable cluster variables $(X_1\cup X_2)\cup X_3$ and
the set of frozen cluster variables $((X_{1})_f\cup (X_{2})_f)\cup (X_{3})_f$ and
$\Sigma_1\amalg(\Sigma_2\amalg\Sigma_3)$ has the set of exchangeable cluster variables $X_1\cup (X_2\cup
X_3)$ and the set of frozen cluster variables $(X_{1})_f\cup ((X_{2})_f\cup
(X_{3})_f)$.
Following the associative law of the sets, we get the associative law of the amalgamated sum, since subseeds
are uniquely determined by their cluster variables.

\begin{Ex}
Let $Q:\xymatrix{x_1\ar[r]^{}&*+[F]{x_2}\ar[r]^{}&x_3\ar[r] &x_4}$, $Q_1:\xymatrix{x_1\ar[r]^{}&*+[F]{x_2}}$ and $Q_2:\xymatrix{*+[F]{x_2}\ar[r]^{}&x_3}$ be quivers with $x_1,x_3,x_4$ exchangeable variables, $x_2$ frozen. Since $\{x_1\}\cap \{x_3\}=\emptyset$, we get $\Sigma(Q_1)\amalg\Sigma(Q_2)=\Sigma(Q')$, where $Q':\xymatrix{x_1\ar[r]^{}&*+[F]{x_2}\ar[r]^{}&x_3}$.
\end{Ex}

\begin{Def}
A seed $\Sigma=(X,X_{fr},\widetilde{B})$ is called {\bf indecomposable} if it is connected and for any decomposition $\Sigma_1\amalg
\Sigma_2$, either $\Sigma_1$ or $\Sigma_2$ is a trivial seed, equivalently, either $\Sigma=\Sigma_2$ or $\Sigma=\Sigma_1$.
\end{Def}

Note that if $X_{fr}=\emptyset$, then the meaning of {\em indecomposable} and {\em connected} coincide.

\begin{Ex}
Let $Q:\xymatrix{x_1\ar[r]^{}&*+[F]{x_2}\ar[r]^{}&x_3}$, $Q_1:\xymatrix{x_1\ar[r]^{}&*+[F]{x_2}}$ and $Q_2:\xymatrix{*+[F]{x_2}\ar[r]^{}&x_3}$ be quivers with $x_1,x_3$ exchangeable variables, $x_3$ frozen. Then according to the definition, $\Sigma(Q)$ is connected, but it is not indecomposable, since $\Sigma(Q)=\Sigma(Q_1)\amalg\Sigma(Q_2)$.
\end{Ex}

\begin{Rem}
In case $\Sigma$ is skew-symmetrizable, the
indecomposability of $\Sigma$ is defined in \cite{CZ} in terms of
valued ice quiver.
\end{Rem}
Due to the definitions of indecomposability and positive seed
homomorphism, we have the following lemmas, which are easy to see.

\begin{Lem}\label{indecom}
A seed $\Sigma=(X,X_{fr},\widetilde{B})$ is indecomposable if and only if
for any $x, y\in \widetilde{X}$ with $x\not=y$,  there exists a sequence of
exchangeable cluster variables $(x_1,\cdots,x_s)$ in $X$ such that
$b_{x_1x}\neq 0$, $b_{x_ix_{i+1}}\neq 0$ and $b_{x_sy}\neq 0$ for
$i=1,\cdots,s$.
\end{Lem}

\begin{proof}
``Only if'':\; Otherwise,  there exist $x,y\in \widetilde{X}$, $x\not=y$,  which
do not satisfy the condition.
Set
 $$I_1=\{x\}\cup\{z\in \widetilde{X}|\;\exists\;
(x_1,\cdots,x_s)\subseteq X\;\text{such that}\;b_{x_1x}\neq 0,
b_{x_ix_{i+1}}\neq 0\;\text{and}\;b_{x_sz}\neq
0\;\text{for}\;i=1,\cdots,s\}$$ and
$I'_1=\widetilde{X}\setminus I_1$. Then $x\in I_1$ and $y\in I'_1$, that is, $\Sigma_{\emptyset,
I_1}$ and $\Sigma_{\emptyset,I'_1}$ are non-trivial.

Since $\Sigma$ is connected, we have $\Sigma=\Sigma_{\emptyset,
I_1}\amalg\Sigma_{\emptyset,I'_1}$,
which contradicts to the indecomposability of $\Sigma$.

``If'':\; Clearly, $\Sigma$ is connected. If $Q$ is
decomposable, then we have $\Sigma=\Sigma_1\amalg\Sigma_2$ with
 non-trivial $\Sigma_{1}=(X_{1},(X_1)_{fr},\widetilde{B^{1}})$ and $\Sigma_{2}=(X_{2},(X_2)_{fr},\widetilde{B^{2}})$. Then we can find $x\in X_1$, $y\in X_2$ such that there exists a sequence of
exchangeable cluster variables $(x_1,\cdots,x_s)$ in $X=X_1\cup X_2$ satisfying that
$b_{x_1x}\neq 0$, $b_{x_ix_{i+1}}\neq 0$ and $b_{x_sy}\neq 0$ for
$i=1,\cdots,s$, which is impossible according to (\ref{union}), the form of $\widetilde{B}$.
\end{proof}

\begin{Lem}
If a non-trivial seed $\Sigma=(X,X_{fr},\widetilde{B})$ is indecomposable,
then any seed homomorphism $f:\Sigma\rightarrow\Sigma'$ is either
positive or negative.
\end{Lem}

\begin{proof}
Assume there exist $z\in X$ and $w\in\widetilde{X}$ such that
$b'_{f(z)f(w)}b_{zw}>0$. For any $x\in X$ and $y\in \widetilde{X}$ with $b_{xy}\neq0$, by Lemma \ref{indecom}, there exists a sequence
$(z=z_0,z_1\cdots,z_{s-1},z_s=x)$ in $X$ such that $b_{z_kz_{k+1}}\neq 0$ for
$0\leq k\leq s-1$. Set $z_{s+1}=y$. Since $f$ is a seed homomorphism, we have\\
\centerline{$(b'_{f(z)f(w)}b_{zw})(b'_{f(z)f(z_1)}b_{zz_1})\geq 0,
\;\;\;\text{and}\;\;\;|b'_{f(z)f(w)}|\geq|b_{zw}|>0,
|b'_{f(z)f(z_1)}|\geq|b_{zz_1}|>0,$} Thus,
$b'_{f(z_0)f(z_1)}b_{z_0z_1}>0$. Similarly, for $0\leq i\leq s-1$, we
have\\
\centerline{$(b'_{f(z_i)f(z_{i+1})}b_{z_iz_{i+1}})(b'_{f(z_{i+1})f(z_{i+2})}b_{z_iz_{i+1}})\geq
0$}\\
\;\;and\;\;\;\centerline{$|b'_{f(z_i)f(z_{i+1})}|\geq|b_{z_iz_{i+1}}|>0,
|b'_{f(z_{i+1})f(z_{i+2})}|\geq|b_{z_{i+1}z_{i+2}}|>0$.} Therefore, using induction,
we have $b'_{f(z_i)f(z_{i+1})}b_{z_iz_{i+1}}>0$ for $0\leq i\leq s-1$. In particular, for $i=s$, it follows that for any $x\in X$ and $y\in \widetilde{X}$,
$$b'_{f(x)f(y)}b_{xy}=b'_{f(z_s)f(z_{s+1})}b_{z_sz_{s+1}}>0,$$ which means that $f$ is a positive seed homomorphism.

Similarly, if there exist $z\in X$ and $w\in\widetilde{X}$ with
$b'_{f(z)f(w)}b_{zw}<0$, then $f$ is negative.
\end{proof}

\begin{Lem}\label{decom}
Let $\Sigma=(X,X_{fr},\widetilde{B})$ be a connected seed of a cluster
algebra $\mathcal A$, then there uniquely exist indecomposable
subseeds
$\Sigma_1=(X_1,(X_1)_{fr},\widetilde{B_1}),\cdots,\Sigma_t=(X_t,(X_t)_{fr},\widetilde{B_t})$
 of type $(\emptyset, I_1)$ of $\Sigma$ for some integer $t$ such that

(a)~ $X_i\cap X_j=\emptyset$ if $i\neq j$ and

(b)~ $\Sigma=\coprod\limits_{i=1}^{t}\Sigma_i$ is the decomposition of the
amalgamated sum.

\end{Lem}

\begin{proof}
In the $n\times m$ totally sign-skew-symmetric matrix
$\widetilde{B}$, by an appropriate permutation of the $n$ row
indices and the first $n$ column indices of $\widetilde{B}$
simultaneously, the principal part $B$ of $\widetilde{B}$ can be
decomposed into a block diagonal matrix diag$(B_1\; B_2\; \cdots\;
B_t)$. Then
$\widetilde{B}$ can be written through an appropriate permutation of row indices and column indices as follows:
\begin{equation}\label{matrixdecom}
\left(\begin{array}{ccccc}
B_1 & 0 & \cdots & 0 & B'_1 \\
0 & B_2 & \cdots & 0 & B'_2 \\
  & \cdots &    & \cdots &  \\
0 & 0 & \cdots & B_t & B'_{t}
\end{array}\right),
\end{equation}
satisfying that (i)~ all $B_i$ are indecomposable (i.e. $B_i$ can not be decomposed as block diagonal matrices with smaller ranks via the above operation), (ii)~ all $B_i$ are totally sign-skew-symmetric.

For $1\leq i\leq t$, let $X_i$ be the subset of the exchangeable
variables of $X$ corresponding to the row indexes of $B_i$. By the
uniqueness (up to the permutation of $\{B_i\}_i$) of the block diagonal decomposition of the principal part
$B$, we have $X_i\cap X_j=\emptyset$ for $i\not=j$.

The set $(X_i)_{fr}$ of frozen variables adjacent to $X_i$ is just the
subset of $X_{fr}$ corresponding to the column indexes of $B'_i$ which are non-zero. Then $X_{fr}=\cup_{i=1}^{t}(X_i)_{fr}$.

Let $I_i=\widetilde{X}\setminus(X_i\cup (X_i)_{fr})$ and
$\Sigma_i=\Sigma_{\emptyset,I_i}$ for $1\leq i\leq t$. By the
definition of decomposition of amalgamated sum and comparing with
the form of matrix in (\ref{matrixdecom}), we have
$\Sigma=\coprod\limits_{i=1}^{t}\Sigma_i$.
\end{proof}

\begin{Def}\label{seediso}
Let $\Sigma$ and $\Sigma'$ be two seeds and $f:\Sigma\rightarrow\Sigma'$
be a seed homomorphism.

 (a)~ $f$ is called a {\bf seed isomorphism} if $f$
induces bijections $X\rightarrow X'$ and $\widetilde{X}\rightarrow
\widetilde{X'}$ and $|b_{xy}|=|b'_{f(x)f(y)}|$ for all $x\in X$ and
$y\in \widetilde{X}$;

 (b)~ A seed isomorphism $f$ is called {\bf positive } (respectively, {\bf negative})
if $f$ is positive (respectively, negative) as a seed homomorphism.
\end{Def}

Trivially, we have the following lemmas by the definitions of
(positive) seed homomorphisms and seed isomorphisms:

\begin{Lem}\label{isomorphism}
A seed homomorphism $f:\Sigma\rightarrow\Sigma'$ is an isomorphism if
and only if there exists a unique seed homomorphism
$f^{-1}:\Sigma'\rightarrow \Sigma$ such that $f^{-1}f=id_{\Sigma}$
and $ff^{-1}=id_{\Sigma'}$.
\end{Lem}

\begin{Lem}
 A seed homomorphism $f$ is a positive (respectively, negative) seed isomorphism if and only if $f$ is isomorphic and $b_{xy}=b'_{f(x)f(y)}$ (respectively,  $b_{xy}=-b'_{f(x)f(y)}$) for all $x\in X$ and $y\in \widetilde{X}$.
\end{Lem}

\begin{proof}
``Only If'': The condition (b) in Definition \ref{seedhom}
means that $b'_{f(x)f(y)}b_{xy}\geq 0$ for all $x\in X$ and $y\in
\widetilde{X}$ or $b'_{f(x)f(y)}b_{xy}\leq 0$ for all $x\in X$ and
$y\in \widetilde{X}$. It always holds that
$|b'_{f(x)f(y)}|\geq|b_{xy}|$. Since $f$ is positive, we have $b'_{f(x)f(y)}b_{xy}\leq 0$. Then it follows $b_{xy}=b'_{f(x)f(y)}$ for all $x\in X$ and $y\in \widetilde{X}$.

``If'': It follows immediately by Definition \ref{seediso}.
\end{proof}

\begin{Rem}\label{isoantiiso}
In \cite{ADS}, two seeds $\Sigma=(X,X_{fr},\widetilde{B})$ and
$\Sigma'=(X',X'_{fr},\widetilde{B'})$ are called {\bf isomorphic} (respectively, {\bf anti-isomorphic}) if there
is a bijection $\varphi: \widetilde{X}\rightarrow \widetilde{X'}$,
including a bijection $\varphi: X\rightarrow X'$, such that
$b'_{\varphi(x)\varphi(y)}=b_{xy}$ (respectively, $b'_{\varphi(x)\varphi(y)}=-b_{xy}$) for $x\in X$ and $y\in \widetilde{X}$.
Obviously, their isomorphism (respectively, anti-isomorphism) given in \cite{ADS} is just our positive (respectively, negative) seed isomorphism defined here, which are both only the special cases of seed isomorphisms given by us. For example, let $Q:\xymatrix{x_1\ar[r]^{}&*+[F]{x_2}\ar[r]^{}&x_3}$ and $Q':\xymatrix{x_1\ar[r]^{}&*+[F]{x_2}&x_3\ar[l]^{}}$. Then we have a seed isomorphism $f:\Sigma(Q)\rightarrow \Sigma(Q'), x_i\mapsto x_i$, which is nighter positive nor negative.
\end{Rem}

\begin{Prop} \label{isosubseed}  Applying the notations in Definition \ref{mixingseed}, for a seed $\Sigma$ and its mixing-type sub-seeds $\Sigma_{I_0,I_1}$ and
$\Sigma_{J_0,J_1}$, if $\Sigma_{I_0,I_1}\cong \Sigma_{J_0,J_1}$ in
{\bf Seed}, then it holds that $\#I_1=\#J_1$ and  $\#(I_0\cup
I'_1)=\#(J_0\cup J'_1)$.
\end{Prop}

\begin{proof}
Since $\Sigma_{I_0,I_1}\cong \Sigma_{J_0,J_1}$,  we have
$\#(\widetilde{X}\setminus I_1)=\#(\widetilde{X}\setminus J_1)$ and
$\#(X\setminus (I_0\cup I'_1))=\#(X\setminus (J_0\cup J'_1))$. As
$I_1,J_1\subseteq \widetilde{X}$, $I_0\cup I'_1\subseteq X$ and $J_0\cup J'_1\subseteq X$, therefore, we get $\# I_1=\#J_1$ and
$\#(I_0\cup I'_1)=\#(J_0\cup J'_1)$.
\end{proof}

\begin{Cor} Following Proposition \ref{isosubseed}, it holds that

 (i)~ if $\Sigma_{I_0,I_1}\cong \Sigma_{J_0,J_1}$ in {\bf Seed} and $I'_1=J'_1=\emptyset$, then
$\#I_1=\#J_1$ and $\#I_0=\#J_0$;

(ii)~ if $\Sigma_{I_0,\emptyset}\cong
\Sigma_{J_0,\emptyset}$ in {\bf Seed}, then $\#I_0=\#J_0$;

(iii)~ if $\Sigma_{\emptyset,I_1}\cong \Sigma_{\emptyset,J_1}$ in {\bf Seed}, then
$\#I_1=\#J_1$.
\end{Cor}

Note that

 (1)~  The converse of the above proposition is
not true. For example, for a cluster quiver
$Q:\xymatrix{x_1\ar@{=>}[r]^{}&*+[F]{x_2}\ar[r]^{}&x_3}$ with the
exchangeable variables $1$ and $3$ and frozen variable $x_2$, it is
clear that
$\Sigma(Q)_{\emptyset,\{1\}}\not\cong\Sigma(Q)_{\emptyset,\{3\}}$.

(2)~ In general, $\#I_0=\#J_0$ does not hold even if
$\Sigma_{I_0,I_1}\cong \Sigma_{J_0,J_1}$ in {\bf Seed} and
$\#I_1=\#J_1$. For example, for the same quiver as in (1),
$\Sigma(Q)_{\{1\},\{2,3\}}\cong\Sigma(Q)_{\emptyset,\{1,3\}}$.

\begin{Prop}\label{submu}
Let $\Sigma=(X,X_{fr},\widetilde{B})$ be a seed and $\Sigma_{I_0,I_1}$
be a mixing-type subseed of $\Sigma$. Then there is a positive seed
isomorphism $\mu_{x}(\Sigma_{I_0,I_1})\cong
(\mu_{x}(\Sigma))_{I_0,I_1}$ for any $x\in X\setminus(I_0\cup
I_1)$.
\end{Prop}

\begin{proof}
Denote $\mu_{x}(\Sigma_{I_0,I_1})$ by $(X',X'_{fr},\widetilde{B'})$ and
$(\mu_{x}(\Sigma))_{I_0,I_1}$ by $(X'',X''_{fr},\widetilde{B''})$.
By definition, we have $$X'=X\setminus(I_0\cup I_1
\cup\{x\})\cup\{\mu_{x}^{\Sigma_{I_0,I_1}}(x)\},\;\;\;
\widetilde{X'}=\widetilde{X}\setminus (I_1\cup
\{x\})\cup\{\mu_{x}^{\Sigma_{I_0,I_1}}(x)\} ,$$
$$X''=X\setminus(I_0\cup I_1
\cup\{x\})\cup\{\mu_{x}^{\Sigma}(x)\},\;\;\;
\widetilde{X''}=\widetilde{X}\setminus (I_1\cup
\{x\})\cup\{\mu_{x}^{\Sigma}(x)\} ,$$

To compare the set $\widetilde{X'}$ with $\widetilde{X''}$,
their elements can be correspondent one-to-one with the identity map
but the correspondence between $\mu_{x}^{\Sigma_{I_0,I_1}}(x)$ and
$\mu_{x}^{\Sigma}(x)$.

According to the definition of mutation of matrices, for all $y\in X'$ and
$z\in\widetilde{X'}$, we have
\begin{eqnarray*}
b'_{yz}= \left\{\begin{array}{lll} b_{yz}+\frac{|b_{yx}|b_{xz}+b_{yx}|b_{xz}|}{2}, & \;\text{if}\; y\neq \mu_{x}^{\Sigma_{I_0,I_1}}(x), z\neq \mu_{x}^{\Sigma_{I_0,I_1}}(x); \\
    -b_{yz},& \;\text{otherwise,}\\
\end{array}\right.
\end{eqnarray*}
and for all $y\in X''$ and $z\in\widetilde{X''}$, we have
\begin{eqnarray*}
b''_{yz}= \left\{\begin{array}{lll} b_{yz}+\frac{|b_{yx}|b_{xz}+b_{yx}|b_{xz}|}{2}, & \;\text{if}\; y\neq \mu_{x}^{\Sigma}(x), z\neq \mu_{x}^{\Sigma}(x); \\
    -b_{yz},& \;\text{otherwise}.\\
\end{array}\right.
\end{eqnarray*}
After comparing $b'_{yz}$ with $b''_{yz}$, it is obvious that
$\widetilde{B'}=\widetilde{B''}$. Therefore, the result holds.
\end{proof}

\begin{Def} \label{imageseed}(\cite{ADS})
 Let $f:\Sigma\rightarrow\Sigma'$ be a seed homomorphism. The {\bf image seed} of $\Sigma$ under $f$ is defined to be $f(\Sigma)=(f(X),f(\widetilde X)\setminus f(X),B'')$, where
$B''$ is a $\#(f(X)) \times \#(f(\widetilde{X}))$-matrix with $b''_{xy}=b'_{xy}$ for any $x\in
 f(X)$ and $y\in f(\widetilde{X})$.
\end{Def}

It is easy to see that for $I'_1=\widetilde{X'}\backslash (\widetilde{X'}\cap
f(\widetilde{X}))$ and $I'_0=X'\backslash (f(X)\cup I'_1)$, we have
\begin{equation}\label{twoimagesubseed}
f(\Sigma)=\Sigma'_{I_0',I'_1}.
\end{equation}

Using the image seed $f(\Sigma)$ and by (\ref{twoimagesubseed}), we can introduce the
notions of injective/surjective seed homomorphisms as follows.
\begin{Def}
(i)~ A seed homomorphism $f:\Sigma\rightarrow\Sigma'$ is called {\bf
injective} if $\Sigma\overset{f}{\cong}f(\Sigma)$ in {\bf
Seed}.
(ii)~ A seed homomorphism $f:\Sigma\rightarrow\Sigma'$ is called
{\bf surjective} if $f(\Sigma)=\Sigma'$.
\end{Def}

Following this definition and that of seed isomorphism, trivially, we have the following.
\begin{Prop}
A seed homomorphism $f:\Sigma\rightarrow\Sigma'$ is isomorphic if and only if $f$ is injective and surjective.
\end{Prop}

 \section{Rooted cluster morphisms and the relationship with seed homomorphisms}

In \cite{ADS}, a {\bf rooted cluster algebra} is defined as a
cluster algebra $\mathcal{A}$ together with its initial seed $\Sigma$, denoted by
$\mathcal{A}(\Sigma)$. Moreover,  given a rooted cluster algebra
$\mathcal{A}(\Sigma)$, a sequence of cluster variables
$(y_{1},y_{2},\cdots,y_{l})$ is called {\bf $\Sigma$-admissible} if
$y_{1}$ is exchangeable in $\Sigma$ and $y_{i}$ is exchangeable in
$\mu_{y_{i-1}}\cdots\mu_{y_{1}}(\Sigma)$ for every $i\geq 2$. Let
$\mathcal{A}(\Sigma')$ be another rooted cluster algebra and $f:
\mathds F(\Sigma)\rightarrow \mathds F(\Sigma')$ a map as sets. A sequence of cluster variables
$\{y_{1},y_{2},\cdots,y_{l}\}\subseteq \mathcal{A}(\Sigma)$ is
called {\bf $(f,\Sigma,\Sigma')$-biadmissible} if it is
$\Sigma$-admissible and $(f(y_{1}),f(y_{2}),\cdots,f(y_{l}))$ is
$\Sigma'$-admissible.

\begin{Def}\label{rootmorph}(Definition 2.2, \cite{ADS})
A {\bf rooted cluster morphism}  $f$ from $\mathcal{A}(\Sigma)$
to $\mathcal{A}(\Sigma')$ is a ring morphism which sends
1 to 1 satisfying:\\
CM1. $f(\widetilde{X})\subseteq \widetilde{X'}\sqcup \mathbb{Z}$;\\
CM2. $f(X)\subseteq X'\sqcup \mathbb{Z}$;\\
CM3. For every $(f,\Sigma,\Sigma')$-biadmissible sequence
$(y_{1},y_{2},\cdots,y_{s})$ and for any $y\in\widetilde{X}$, we have
 $$f(\mu_{y_{s}}\cdots \mu_{y_{1}}(y))=\mu_{f(y_{s})}\cdots\mu_{f(y_{1})}(f(y)).$$
\end{Def}

 By Definition \ref{rootmorph}, a rooted cluster morphism is first a ring morphism. Following this, we say  a rooted cluster morphism $f$ to be {\bf surjective} (respectively, {\bf injective}) if $f$ is surjective (respectively, injective) as a ring morphism. A rooted cluster morphism  $f:\;\mathcal{A}(\Sigma)\rightarrow\mathcal{A}(\Sigma')$ in {\bf Clus} is called an {\bf isomorphism} in {\bf Clus} if it is both injective and also surjective and is written as $\mathcal{A}(\Sigma)\overset{f}{\cong}\mathcal{A}(\Sigma')$.

If there is an injective rooted cluster morphism $f$ from
$\mathcal{A}(\Sigma)$ to $\mathcal{A}(\Sigma')$, then the rooted
cluster algebra $\mathcal{A}(\Sigma)$ is called a {\bf rooted
cluster subalgebra} of $\mathcal{A}(\Sigma')$ (see \cite{CZ}).

Dually, if there is a surjective rooted cluster morphism $f$ from
$\mathcal{A}(\Sigma)$ to $\mathcal{A}(\Sigma')$, then the rooted
cluster algebra $\mathcal{A}(\Sigma')$ is called a {\bf rooted
cluster quotient algebra} of $\mathcal{A}(\Sigma)$. Note that, in the category {\bf Clus}, the surjective morphism and epimorphism are not coincide. See \cite{ADS}.

The {\bf category of rooted cluster algebras} is defined as the
category \textbf{Clus} whose objects are all rooted cluster algebras and whose morphisms between two rooted cluster algebras are all rooted
cluster morphisms.

It is easy to see that the conditions CM1 and CM2 in Definition \ref{rootmorph} give the restriction of the rooted cluster morphism $f$ on the seed $\Sigma$, which shows the relation between the seeds $\Sigma$ and $\Sigma'$. Motivated by this fact and our discussion, we have introduced the theory of seed homomorphisms in the last section, which will be used to unify and understand the research on rooted cluster algebras via the relations among seeds.

\begin{Def}\label{reducedseed}
Let $f:\mathcal{A}(\Sigma)\rightarrow \mathcal{A}(\Sigma')$ be a
rooted cluster morphism and
\begin{equation}\label{setI_1}
I_1=\{x\in
\widetilde{X}|f(x)\in\mathbb{Z}\}.
\end{equation}
From $f$,  define a
new seed $\Sigma^{(f)}=(X^{(f)},X^{(f)}_{fr},\widetilde{B^{(f)}})$ satisfying that:

 (I). $X^{(f)}=X\setminus I_1=\{x\in X|f(x)\notin \mathbb{Z}\}$;

(II). $\widetilde{X^{(f)}}=\widetilde{X}\setminus I_1=\{x\in
\widetilde{X}|f(x)\notin \mathbb{Z}\}$;

(III). $\widetilde{B^{(f)}}=(b^{(f)}_{xy})$ is a $\#(X^{(f)})\times
\#(\widetilde{X^{(f)}})$ matrix with
$$b^{(f)}_{xy}= \left\{\begin{array}{lll}b_{xy}, & \;\text{if}\; f(z)\neq 0 \;\forall z\in I_1 \;\text{adjacent to}\; x \;\text{or}\; y;\\
0,& \text{otherwise}.\\
\end{array}\right.$$

We call this seed $\Sigma^{(f)}=(X^{(f)},X^{(f)}_{fr},\widetilde{B^{(f)}})$ the {\bf contraction} of $\Sigma$ under $f$.
\end{Def}

 For a rooted cluster morphism $f:\mathcal{A}(\Sigma)\rightarrow \mathcal{A}(\Sigma')$, if $I_1=\{x\in
\widetilde{X}|f(x)\in\mathbb{Z}\}=\emptyset$, then $\Sigma^{(f)}=\Sigma$; in this case, we say $f$ to be a {\bf noncontractible morphism } and the seed $\Sigma$ to be a {\bf noncontractible seed} under $f$.

\begin{Rem}\label{c}
Using the definitions of $I_1$ in (\ref{setI_1}) and of the new seed $\Sigma^{(f)}$, we have
$\Sigma^{(f)}=\Sigma_{\emptyset,I_1}$ if $f(x)\not=0$ for any $x\in
\widetilde{X}$, since $b^{(f)}_{xy}=b_{xy}$ in (III) for any $x,y$ in this case.
\end{Rem}

\begin{Prop} \label{induceseed}
 A rooted cluster morphism $f: \mathcal
A(\Sigma)\rightarrow\mathcal A(\Sigma')$ determines uniquely a seed
homomorphism $(f^S, \Sigma^{(f)}, \Sigma')$ from $\Sigma^{(f)}$
to $\Sigma'$ via $f^S(x)=f(x)$ for $x\in \widetilde{X^{(f)}}$.
\end{Prop}

\begin{proof}
 By the definition of $\Sigma^{(f)}$, $\varphi$ satisfies the condition (a) of Definition \ref{seedhom}. For any two adjacent pairs $(x,y)$ and $(z,w)$ with $x,z\in X^{(f)}$ and $y,w\in \widetilde{X^{(f)}}$, if either $b^{(f)}_{xy}=
0$ or $b^{(f)}_{zw}=0$, then it always holds that
$(b^{(f)}_{xy}b'_{f(x)f(y)})(b^{(f)}_{zw}b'_{f(z)f(w)})=0$. So, now we assume $b^{(f)}_{xy}\neq
0$ and $b^{(f)}_{zw}\neq 0$.

Under this condition, by the definition of $\widetilde{B^{(f)}}$, we have $b^{(f)}_{xy}=b_{xy}$ and $b^{(f)}_{zw}=b_{zw}$. Then it follows that $f(u)\neq 0$ for all $u$
adjacent to $x$ or $z$.

The following discussion is only made in the case where $b_{xy},b_{zw}>0$. For the other cases, the same conclusion can be obtained in a similar way.

By CM3, we have $f(\mu_{x}(x))=\mu_{f(x)}(f(x))$, which means that
$$\frac{f(y^{b_{xy}}\prod\limits_{b_{xu}\geq 0,u\neq y}u^{b_{xu}}+\prod\limits_{b_{xu}\leq 0}u^{-b_{xu}})}{f(x)}=\frac{\prod\limits_{b'_{f(x)v}\geq 0,v\in\widetilde{X'}}v^{b'_{f(x)v}}+\prod\limits_{b'_{f(x)v}\leq 0,v\in\widetilde{X'}}v^{-b'_{f(x)v}}}{f(x)},$$
By the algebraic independence of variables in $\widetilde{X'}$, we have
\begin{equation}\label{eq:14}
f(y^{b_{xy}}\prod\limits_{b_{xu}\geq 0,u\neq
y}u^{b_{xu}})=\prod\limits_{b'_{f(x)v}\geq
0,v\in\widetilde{X'}}v^{b'_{f(x)v}}
\;\;\;\text{and}\;\;\;f(\prod\limits_{b_{xu}\leq
0}u^{-b_{xu}})=\prod\limits_{b'_{f(x)v}\leq
0,v\in\widetilde{X'}}v^{-b'_{f(x)v}},
\end{equation}
or
\begin{equation}\label{eq:15}
f(y^{b_{xy}}\prod\limits_{b_{xu}\geq 0,u\neq
y}u^{b_{xu}})=\prod\limits_{b'_{f(x)v}\leq
0,v\in\widetilde{X'}}v^{-b'_{f(x)v}}
\;\;\;\text{and}\;\;\;f(\prod\limits_{b_{xu}\leq
0}u^{-b_{xu}})=\prod\limits_{b'_{f(x)v}\geq
0,v\in\widetilde{X'}}v^{b'_{f(x)v}}.
\end{equation}

{\bf Case 1:}\; Assume $x=z$.

If (\ref{eq:14}) holds, then we get
$f(y)|\prod\limits_{b'_{f(x)v}\geq
0,v\in\widetilde{X'}}v^{b'_{f(x)v}}$. For the pair $(x,y)$, comparing the exponents of $f(y)$ in the two sides of the first expression in (\ref{eq:14}), we have
$b'_{f(x)f(y)}=\sum\limits_{f(u)=f(y),b_{xu}>0}b_{xu}>0$; similarly, for the pair $(z,w)$,
we have $b'_{f(z)f(w)}>0$. If $x=z$, then it follows
\begin{equation}\label{eq:16}
(b_{xy}b'_{f(x)f(y)})(b_{zw}b'_{f(z)f(w)})=(b_{xy}b'_{f(x)f(y)})(b_{xw}b'_{f(x)f(w)})>0.
\end{equation}

If (\ref{eq:15}) holds, then we get
$b'_{f(x)f(y)}=\sum\limits_{f(u)=f(y),b_{xu}>0}(-b_{xu})<0$,
$b'_{f(z)f(w)}<0$, and similarly, (\ref{eq:16}) also follows.

{\bf Case 2:}\; Assume $x\neq z$.

Applying the result in Case 1 on the adjacent pairs $(x,y)$ and
$(x,z)$, we have
\begin{equation}\label{eq:17}
(b_{xy}b'_{f(x)f(y)})(b_{xz}b'_{f(x)f(z)})>0.
\end{equation}
On the other hand, applying the result in case 1 on the adjacent pairs
$(z,x)$ and $(z,w)$, we get
\begin{equation}\label{eq:18}
(b_{zx}b'_{f(z)f(x)})(b_{zw}b'_{f(z)f(w)})>0.
\end{equation} Combining (\ref{eq:17}) and (\ref{eq:18}), therefore, we have
$(b_{xy}b'_{f(x)f(y)})(b_{zw}b'_{f(z)f(w)})>0.$

In summary, it follows
$(b^{(f)}_{xy}b'_{f(x)f(y)})(b^{(f)}_{zw}b'_{f(z)f(w)})=(b_{xy}b'_{f(x)f(y)})(b_{zw}b'_{f(z)f(w)})>0.$

 Moreover, no matter whether (\ref{eq:14}) or (\ref{eq:15}) holds,
we have $|b'_{f(x)f(y)}|=\sum\limits_{f(u)=f(y)}|b_{xy}|\geq
|b_{xy}|=|b^{(f)}_{xy}|$. Therefore, $f^S$ is a seed homomorphism
from $\Sigma^{(f)}$ to $\Sigma'$.
\end{proof}

Following this proposition, we call $f^S$ the {\bf restricted seed
homomorphism} of the rooted cluster morphism $f$.

Conversely, for any seed homomorphism $g: \Sigma\rightarrow\Sigma'$,
we can induce a ring homomorphism $G:\mathbb{Q}[X_{fr}][X^{\pm
1}]\rightarrow \mathbb{Q}[X'_{fr}][X'^{\pm 1}]$ by defining $G(x)=g(x)$
for $x\in\widetilde{X}$ and $G(x^{-1})=g(x)^{-1}$ for all $x\in X$.
The restriction $G|_{\mathcal{A}(\Sigma)}$ of $G$ is also a ring homomorphism from $\mathcal{A}(\Sigma)$ to $\mathbb{Q}[X'_{fr}][X'^{\pm 1}]$. If $Im(G|_{\mathcal{A}(\Sigma)})\subseteq \mathcal{A}(\Sigma')$ and $G|_{\mathcal{A}(\Sigma)}$ is a rooted cluster morphism from $\mathcal{A}(\Sigma)$ to
$\mathcal{A}(\Sigma')$, we call $G|_{\mathcal{A}(\Sigma)}$ the {\bf induced rooted cluster morphism} of $g$. In this case, denote $G|_{\mathcal{A}(\Sigma)}$ as $g^E$.

Motivated by the above discussion, the natural questions one has to consider are
that for a rooted cluster morphism $f: \mathcal{A}(\Sigma)\rightarrow\mathcal{A}(\Sigma')$ and a seed homomorphism $g:\Sigma\rightarrow\Sigma'$.

{\bf Question} (I).~ When $(f^S)^E$ and $g^E$ exist, under what conditions $f=(f^S)^E$ and  $g=(g^E)^S$ hold?

First, we have the following easy observation:
\begin{Lem}\label{equel}
Let $f,g:\mathcal{A}(\Sigma)\rightarrow\mathcal{A}(\Sigma')$
be rooted cluster morphisms. If $f(x)=g(x)\neq 0$ for all $x\in
\widetilde{X}$ of $\Sigma$, then $f=g$.
\end{Lem}

\begin{proof}
It suffices to prove $f(y)=g(y)$ for all cluster
variable $y\in \mathcal{A}(\Sigma)$. By Laurent phenomenon, $y\in
\mathbb{Q}[X_{fr}][X^{\pm 1}]$, so $y=\frac{h}{m}$ for a monomial
$m$ and a polynomial $h$. Thus,
$f(y)=\frac{f(h)}{f(m)}=\frac{g(h)}{g(m)}=g(y)$.
\end{proof}

In particular, this proposition is satisfied if
$f,g:\mathcal{A}(\Sigma)\rightarrow\mathcal{A}(\Sigma')$ are
noncontractible.

Now, we use this lemma to answer the question.

In fact, for a seed homomorphism $g: \Sigma\rightarrow\Sigma'$, if its induced rooted cluster morphism $g^E$ exists, then we always have $g=(g^E)^{S}:
\Sigma\rightarrow\Sigma'$, since $\Sigma^{(g^E)}=\Sigma$ and
for all $x\in \widetilde{X}$, $(g^E)^{S}(x)=g^E(x)=g(x)$.

Moreover, for a rooted cluster morphism $f:
\mathcal{A}(\Sigma)\rightarrow\mathcal{A}(\Sigma')$, if the induced rooted cluster morphism $(f^{S})^E: \mathcal{A}(\Sigma^{(f)})\rightarrow\mathcal{A}(\Sigma')$  exists,
then $f=(f^{S})^{E}$ if and only if $f$ is noncontractible. Indeed, $\Sigma^{(f)}=\Sigma$ if and only if $f$ is
noncontractible, and in this case, $(f^{S})^{E}(x)=f^{S}(x)=f(x)$ for all $x\in
\widetilde{X}$. Then by Lemma \ref{equel}, $f=(f^{S})^{E}$ if and only if $f$ is noncontractible.

In summary, we obtain the following result as the answer of Question (I).
\begin{Prop} \label{gobackresult}
For a rooted cluster morphism $f: \mathcal{A}(\Sigma)\rightarrow\mathcal{A}(\Sigma')$ and a seed homomorphism $g:\Sigma\rightarrow\Sigma'$, if
the induced rooted cluster morphisms $(f^{S})^E$ and $g^E$ exist, then (1)~ $g=(g^E)^{S}:
\Sigma\rightarrow\Sigma'$ and (2)~ $f=(f^{S})^{E}: \mathcal{A}(\Sigma^{(f)})\rightarrow\mathcal{A}(\Sigma')$ if and only if $f$ is noncontractible.  \end{Prop}

A further question is that: under what condition, do there exist $(f^S)^E$ and $g^E$? This question seems difficult for us now. Maybe we would study it in the future work.

\begin{Prop}\label{generate}
For a rooted cluster morphism
$g:\mathcal{A}(\Sigma)\rightarrow\mathcal{A}(\Sigma')$ and any
$(g,\Sigma,\Sigma')$-biadmissible sequence $(y_1,\cdots,y_t)$, the
morphism
\begin{equation}\label{morphinvariant}
g:\; \mathcal{A}(\mu_{y_t}\cdots\mu_{y_1}(\Sigma))\rightarrow\mathcal{A}(\mu_{g(y_t)}\cdots\mu_{g(y_1)}(\Sigma'))
\end{equation}
is still a rooted cluster morphism on the seed $\mu_{y_t}\cdots\mu_{y_1}(\Sigma)$.
\end{Prop}

\begin{proof}
 By induction, it suffices to prove the result for the case $t=1$. First, note that for any $x\in \widetilde{X}$ and $x\not=y_1$, we have
$g(x)\neq g(y_1)$ since $g(\mu_{y_1}(x))=g(x)=\mu_{g(y_1)}(g(x))$ by CM3 on $\mathcal{A}(\Sigma)$.
 Hence, $g(\widetilde{X}\setminus \{y_1\})\subseteq(\widetilde{X'}\setminus \{g(y_1)\})\cup \mathds{Z}$ and $g(X\setminus \{y_1\})\subseteq(X'\setminus \{g(y_1)\})\cup \mathds{Z}$. Following this, CM1 and CM2 for $g$ on $\mathcal{A}(\mu_{y_1}(\Sigma))$ are obtained directly.  For any
$(g,\mu_{y_1}(\Sigma),\mu_{g(y_1)}(\Sigma'))$-biadmissible sequence
$(z_2,\cdots,z_s)$, by definition, $(y_1,z_2,\cdots,z_s)$ is a
$(g,\Sigma,\Sigma')$-biadmissible sequence. Then by CM3 for $g$ on $\mathcal A(\Sigma)$, we have
$g(\mu_{z_s}\cdots\mu_{z_2}\mu_{y_1}(y))=\mu_{g(z_s)}\cdots\mu_{g(z_2)}\mu_{g(y_1)}(g(y))$ for $y\in \widetilde{X}$. Combining the fact that $\mu_{y_1}(\widetilde{X})$ is the extended cluster of $\mu_{y_1}(\Sigma)$, CM3 holds for $g$ on $\mathcal{A}(\mu_{y_1}(\Sigma))$.
\end{proof}

Note that as algebras, we have
$\mathcal{A}(\Sigma)=
\mathcal{A}(\mu_{y_t}\cdots\mu_{y_1}(\Sigma))$ and
$\mathcal{A}(\Sigma')=
\mathcal{A}(\mu_{g(y_t)}\cdots\mu_{g(y_1)}(\Sigma'))$. From the fact of (\ref{morphinvariant}), we know that the morphism $g$ is still of rooted
cluster on the seed $\mu_{y_t}\cdots\mu_{y_1}(\Sigma)$ for any $(g,\Sigma,\Sigma')$-biadmissible sequence $(y_1,\cdots,y_t)$. For this reason, we say $g$ to be  a {\bf rooted
cluster morphism generated by $(y_1,\cdots,y_t)$}.

 For a seed homomorphism $g_0: \Sigma\rightarrow\Sigma'$, assume that its induced rooted cluster morphism $g_0^E: \mathcal A(\Sigma)\rightarrow\mathcal A(\Sigma')$ exists. Then by  Proposition \ref{generate},  for any
$((g_0)^E,\Sigma,\Sigma')$-biadmissible sequence $(y_1,\cdots,y_t)$, the
morphism
$g_0^E: \mathcal{A}(\mu_{y_t}\cdots\mu_{y_1}(\Sigma))\rightarrow\mathcal{A}(\mu_{g(y_t)}\cdots\mu_{g(y_1)}(\Sigma'))$ is still a rooted cluster morphism on the initial seed $\mu_{y_t}\cdots\mu_{y_1}(\Sigma)$. By Proposition \ref{gobackresult}, for $\Sigma$, we have $(g_0^E)^S=g_0: \Sigma\rightarrow\Sigma'$; for this reason, for the seed $\mu_{y_t}\cdots\mu_{y_1}(\Sigma)$ and the rooted
cluster morphism $g_0^E$ generated by $(y_1,\cdots,y_t)$, we denote
\begin{equation}\label{invariant}
\mu_{y_t}\cdots\mu_{y_1}(g_0)=(g_0^E)^S: \mu_{y_t}\cdots\mu_{y_1}(\Sigma)\rightarrow \mu_{g_0(y_t)}\cdots\mu_{g_0(y_1)}(\Sigma').
   \end{equation}
   where we say $\mu_{y_t}\cdots\mu_{y_1}(g_0)$ to be obtained from $g_0$ by the {\bf $t$-mutations of seed homomorphisms} at the exchangeable variables $y_1,\cdots,y_t$ for any positive integer $t$.

   Of course, $g_0^E$ is noncontractible; by Proposition \ref{gobackresult}, $((g_0^E)^S)^E=g_0^E$. Using (\ref{invariant}), we obtain  $(\mu_{y_t}\cdots\mu_{y_1}(g_0))^E=g_0^E$ as algebra morphisms.
Therefore, Proposition \ref{generate} indeed tells us that

 {\em For a seed homomorphism $g_0$, the operation for giving the induced rooted cluster morphisms $g_0^E$ is invariant under mutations of seed homomorphisms. }

The following result illustrates the relation between seed isomorphism and rooted cluster isomorphism.

\begin{Prop}\label{basiclem}
$\mathcal{A}(\Sigma)\cong\mathcal{A}(\Sigma')$ in \textbf{Clus} if
and only if $\Sigma\cong\Sigma'$ in \textbf{Seed}.
\end{Prop}
\begin{proof}
``Only if'':\;\;Let
$f:\mathcal{A}(\Sigma)\rightarrow\mathcal{A}(\Sigma')$ be a rooted
cluster isomorphism with inverse $g$. According to Proposition
\ref{induceseed}, we have $f^S$ and $g^S$ as the restricted seed
homomorphisms of $f$ and $g$, respectively. As the restrictions of $f$ and $g$, it is clear that  $g^Sf^S|_{\widetilde{X}}=id_{\widetilde{X}}$ and $f^Sg^S|_{\widetilde{X'}}=id_{\widetilde{X'}}$. Note that $f$ and $g$ cannot map an exchangeable variable to a frozen variable. We also have $g^Sf^S|_{X}=id_{X}$ and $f^Sg^S|_{X'}=id_{X'}$.

Moreover, since $f^S$ is a seed homomorphism,  $|b_{xy}|\leq
|b'_{f^S(x)f^S(y)}|$ for all $x,y\in \widetilde{X}$. On the other hand, for $g^S$ as a seed
homomorphism, we have $|b'_{f^S(x)f^S(y)}|\leq
|b_{g^Sf^S(x)g^Sf^S(x)}|=|b_{xy}|$ for all $x,y\in \widetilde{X}$.
Therefore, $f^S:\Sigma\rightarrow\Sigma'$ is a seed
isomorphism.

``If'':\;\;For a seed isomorphism  $\Sigma\overset{F}{\cong}\Sigma'$, we have its prolongation $\bar f:
\mathbb{Q}[X_{fr}][X^{\pm 1}]\rightarrow \mathbb{Q}[X'_{fr}][X'^{\pm
1}]$ as an algebra isomorphism of Laurent polynomials, which satisfies that $\bar f(x)=F(x)$ for all
$x\in \widetilde{X}$ and $\bar f(x)^{-1}=(F(x))^{-1}$ for $x\in X$. Now
we prove that $\bar f$ can induce a rooted cluster isomorphism
$F^E:\mathcal{A}(\Sigma)\rightarrow\mathcal{A}(\Sigma')$ for $F^E=\bar f|_{\mathcal{A}(\Sigma)}$.

First, we prove that $F^E=\bar f|_{\mathcal{A}(\Sigma)}$ satisfies the conditions CM1,CM2 and CM3 such that $F^E(\mathcal{A}(\Sigma))\subseteq \mathcal{A}(\Sigma')$.

In fact, CM1 and CM2 for $F^E$ hold clearly since $F$ is a seed homomorphism. Now we prove that any $\Sigma$-admissible
sequence $(z_1,\cdots,z_s)$ is $(F^E,\Sigma,\Sigma')$-biadmissible and then that CM3 holds for $F^E$.

For $s=1$,  $(z_1)$ is
$(F^E,\Sigma,\Sigma')$-biadmissible trivially due to $F$ as a seed
isomorphism.

For $\widetilde{X}\ni x\not=z_1$, it is clear that $F^E(\mu_{z_1}(x))=\mu_{F^E(z_1)}(F^E(x))$ by the definition of mutation and the injection of $F^E$.

 Now consider the case for $x=z_1$.
    Since $F$ is a seed isomorphism, we have $b_{z_1y}=b'_{F(z_1)F(y)}=b'_{F^E(z_1)F^E(y)}$ for all $b_{z_1y}\neq 0$ or $b_{z_1y}=-b'_{F(z_1)F(y)}=-b'_{F^E(z_1)F^E(y)}$ for all $b_{z_1y}\neq 0$. In the both cases, it holds $$\prod\limits_{b_{z_1y}>0}F^E(y^{b_{z_1y}})+\prod\limits_{b_{z_1y}<0}F^E(y^{-b_{z_1y}})=\prod\limits_{b'_{F^E(z_1)F^E(y)}>0}F^E(y)^{b'_{F^E(z_1)F^E(y)}}+\prod\limits_{b'_{F^E(z_1)F^E(y)}<0 }F^E(y)^{-b'_{F^E(z_1)F^E(y)}}.$$
    Thus, $F^E(\mu_{z_1}(z_1))=\mu_{F^E(z_1)}(F^E(z_1))$, since $F^E|_{\widetilde{X}}=F|_{\widetilde{X}}$ is a bijection.

Assume that $(z_1,\cdots,z_s)$ is
$(F^E,\Sigma,\Sigma')$-biadmissible and $F^E$ satisfies CM3 for $s<t$.

Now consider the case for $s=t$. By the definition of seed isomorphisms, we have the observation:
\begin{equation}\label{equ:iso}
\mu_{z_{t-1}}\cdots\mu_{z_1}(F):\mu_{z_{t-1}}\cdots\mu_{z_1}(\Sigma)\rightarrow\mu_{F(z_{t-1})}\cdots\mu_{F(z_1)}(\Sigma')
\end{equation} as a seed isomorphism. Note that $F(z_i)=F^E(z_i)$ for any $i$.

Since $(z_1,\cdots,z_{t})$ is $\Sigma$-admissible,
  $(z_t)$ is $\mu_{z_{t-1}}\cdots\mu_{z_1}(\Sigma)$-admissible. Using the isomorphism in (\ref{equ:iso}), we know that $F^E(z_t)$ is $\mu_{F(z_{t-1})}\cdots\mu_{F(z_1)}(\Sigma')$-admissible; thus, $(z_t)$ is \\ \centerline{$(F^E,\mu_{z_{t-1}}\cdots\mu_{z_1}(\Sigma), \mu_{F(z_{t-1})}\cdots\mu_{F(z_1)}(\Sigma'))$-biadmissible.} Therefore,
$(z_1,\cdots,z_t)$ is also $(F^E,\Sigma,\Sigma')$-biadmissible. Moreover,
from (\ref{equ:iso}), it follows that
\begin{equation}\label{a}
F^E(\mu_{z_t}(z))=\mu_{F(z_t)}(F^E(z))
\end{equation}
for all cluster variables  $z$ in the seed $\mu_{z_{t-1}}\cdots\mu_{z_1}(\Sigma)$, where $\mu_{z_t}$ and $\mu_{F(z_t)}$ mean the mutations at $z_t$ and $F(z_t)$ in the seeds $\mu_{z_{t-1}}\cdots\mu_{z_1}(\Sigma)$ and $\mu_{F(z_{t-1})}\cdots\mu_{F(z_1)}(\Sigma')$, respectively. Hence,
$$F^E(\mu_{z_{t}}\cdots\mu_{z_1}(x))=F^E(\mu_{z_{t}}(\mu_{z_{t-1}}\cdots\mu_{z_1}(x)))=\mu_{F(z_{t})}(F^E(\mu_{z_{t-1}}\cdots\mu_{z_1}(x)))=\mu_{F(z_{t})}
\cdots\mu_{F(z_1)}(F^E(x))$$ for all $x\in\widetilde X$, where the second equality is by (\ref{a}) and the third one is by the induction assumption. Thus, CM3 follows.

Since $\mathcal{A}(\Sigma')$ is generated by all its cluster variables, we have $F^E(\mathcal{A}(\Sigma))=\bar f(\mathcal{A}(\Sigma))\subseteq \mathcal{A}(\Sigma')$ due to CM1, CM2 and CM3 shown above.

Second, the above discussion on $\bar f$ is also suitable for $\bar f^{-1}$; hence, similarly, we have $(\bar f)^{-1}(\mathcal{A}(\Sigma'))\subseteq \mathcal{A}(\Sigma)$. It follows that $\mathcal{A}(\Sigma')\subseteq \bar f( \mathcal{A}(\Sigma))$. Hence, $F^E(\mathcal{A}(\Sigma))=\bar f(\mathcal{A}(\Sigma))= \mathcal{A}(\Sigma')$.

Note that $F^E$ is injective. Therefore, $F^E:\mathcal{A}(\Sigma)\rightarrow \mathcal{A}(\Sigma')$ is a rooted cluster isomorphism.
\end{proof}

\begin{Rem}
Note that Theorem 3.9 in \cite{ADS} states that
$\mathcal{A}(\Sigma)\cong\mathcal{A}(\Sigma')$ in
\textbf{Clus} if and only if either $\Sigma\cong\Sigma'$ or $\Sigma\cong\Sigma'^{op}$. In fact, this result has dealt only the rooted cluster algebras with indecomposable seeds. For example, let $Q:\xymatrix{x_1\ar[r]^{}&*+[F]{x_2}\ar[r]^{}&x_3}$ and $Q':\xymatrix{x_1\ar[r]^{}&*+[F]{x_2}&x_3\ar[l]^{}}$, it is clear that $f:\mathcal A(\Sigma(Q))\rightarrow \mathcal A(\Sigma(Q')), x_i\mapsto x_i$ for $i=1,2,3$, is a rooted cluster isomorphism, while $\Sigma(Q)\not\cong \Sigma(Q')$ or $\Sigma(Q)\not\cong \Sigma(Q')^{op}$ in sense of \cite{ADS}.
\end{Rem}

For a rooted cluster morphism
$f:\mathcal{A}(\Sigma)\rightarrow\mathcal{A}(\Sigma')$, in \cite{ADS}, the authors defined the {\bf
image seed} of $f$ to be the image
seed of $f^S:\Sigma^{(f)}\rightarrow\Sigma'$ by Proposition \ref{induceseed}, that is, $f^S(\Sigma^{(f)})$ by Definition \ref{imageseed}.

 In \cite{ADS}, a rooted cluster morphism
$f:\;\mathcal{A}(\Sigma)\rightarrow\mathcal{A}(\Sigma')$ is called
 {\bf ideal} if $\mathcal{A}(f^S(\Sigma^{(f)}))=f(\mathcal{A}(\Sigma))$.

\begin{Lem}\label{idealmor.}(Proposition 2.36(2), \cite{CZ})
Let $f:\;\mathcal{A}(\Sigma)\rightarrow\mathcal{A}(\Sigma')$
be an ideal rooted cluster morphism. Then $f=\tau f_1$ with a surjective rooted cluster morphism $f_1$ and an injective rooted cluster morphism $\tau$, that is,
$f:\mathcal{A}(\Sigma)\overset{f_1}{\twoheadrightarrow}\mathcal{A}(f^S(\Sigma^{(f)}))\overset{\tau}{\hookrightarrow}\mathcal{A}(\Sigma').$
\end{Lem}

\begin{Lem}\label{ideal} (\cite{ADS})
Any injective rooted cluster morphism $f:\;\mathcal{A}(\Sigma)\rightarrow\mathcal{A}(\Sigma')$ is ideal.
\end{Lem}

\begin{proof}
Since $f$ is injective, we know clearly $\Sigma^{(f)}=\Sigma$. Then by Proposition \ref{induceseed}, $f^S: \Sigma\rightarrow \Sigma'$ is a seed homomorphism.  By definition,  $(f^{S})_1:\Sigma\rightarrow f^{S}(\Sigma)$ satisfies $(f^{S})_1(x)=f^{S}(x)$ for all $x\in \widetilde{X}$. We will prove that $(f^{S})_1$ is a seed isomorphism as follows.

Denote $f^{S}(\Sigma)=(Y,Y_{fr},\widetilde{C})$. Owing to the definition of $f^{S}(\Sigma)$, $(f^{S})_1|_X$ and $(f^{S})_1|_{\widetilde{X}}$ are bijections by injection of $f$. For any $x\in X$ and $y\in \widetilde{X}$, by CM3 for $f$, we have $f(\mu_x(x))=\mu_{f(x)}(f(x))$; thus, $$f(\prod\limits_{b_{xz}>0,z\in\widetilde{X}}z^{b_{xz}}+\prod\limits_{b_{xz}<0,z\in\widetilde{X}}z^{-b_{xz}})=\prod\limits_{b'_{f(x)w}>0,w\in\widetilde{X'}}w^{b'_{f(x)w}}+\prod\limits_{b'_{f(x)w}<0,w\in\widetilde{X'}}w^{-b'_{f(x)w}},$$ Comparing the exponent of $f(y)$ in the two sides of the above equation, we get either $b_{xy}=b'_{f(x)f(y)}$ or $b_{xy}=-b'_{f(x)f(y)}$. Thus, $|b_{xy}|=|b'_{f(x)f(y)}|=|c_{(f^{S})_1(x)(f^{S})_1(y)}|$. So, $(f^{S})_1$ is a seed isomorphism.

According to Proposition \ref{basiclem}, $\mathcal{A}(\Sigma)\overset{((f^{S})_1)^E}{\cong}\mathcal{A}(f^{S}(\Sigma))$ is a rooted cluster isomorphism. By definition, clearly $((f^{S})_1)^E=f$ on $\mathcal{A}(\Sigma)$. Moreover, since $f$ is injective, we have $\mathcal{A}(\Sigma)\overset{f}{\cong} f(\mathcal{A}(\Sigma))$.
 It follows that  $f(\mathcal{A}(\Sigma))=\mathcal{A}(f^S(\Sigma)).$
\end{proof}

 Using Lemma \ref{ideal} and (\cite{ADS}, Lemma 3.1), we have the following:
\begin{Prop}\label{Cor rooted}
(1)~If $\mathcal{A}(\Sigma)$ is a rooted cluster subalgebra of
$\mathcal{A}(\Sigma')$ with an injective rooted cluster morphism
$f$, then
 $\mathcal{A}(\Sigma)\cong \mathcal{A}(f^S(\Sigma))$ in {\bf Clus}. Moreover, $\Sigma\cong f^S(\Sigma)$ in {\bf Seed}.

(2)~If $\mathcal{A}(\Sigma')$ is a rooted cluster quotient algebra
of $\mathcal{A}(\Sigma)$ with a surjective rooted cluster morphism
$f$, then $f^S(\Sigma)=\Sigma'$. Moreover,
 $\mathcal{A}(f^S(\Sigma))=\mathcal{A}(\Sigma')$ as rooted cluster
algebras.

\begin{proof}
(1)~
  Since $f$ is injective, we have $\mathcal{A}(\Sigma)\cong f(\mathcal{A}(\Sigma))$ in {\bf Clus}. Then by Lemma \ref{ideal}, it follows that $\mathcal{A}(\Sigma)\cong \mathcal{A}(f^S(\Sigma))$ in {\bf Clus}. By Proposition \ref{basiclem}, $\Sigma\cong f^S(\Sigma)$ in {\bf Seed}.

(2)~  By (\cite{ADS}, Lemma 3.1), since $f$ is surjective, we have
$f(X)\supseteq X'$ and $f(\widetilde{X})\supseteq \widetilde{X'}$. Then
by the definition of image seed, we obtain
$f^S(\Sigma)=\Sigma'$. Therefore, $\mathcal
A(f^S(\Sigma))=\mathcal A(\Sigma')$.
\end{proof}

\end{Prop}

Owing to this proposition, the corresponding injective
(respectively, surjective) seed morphisms are deduced from injective
(respectively, surjective) rooted cluster morphisms as the below
observation:

\begin{Cor}\label{sur.}
 (1)~ The restricted seed morphism from $\Sigma$ to $\Sigma'$ of an injective rooted cluster morphism $f:\mathcal{A}(\Sigma)\rightarrow\mathcal{A}(\Sigma')$ is
injective.

(2)~ The restricted seed morphism from $\Sigma^{(f)}$ to $\Sigma'$
of a surjective rooted cluster morphism
$f:\mathcal{A}(\Sigma)\rightarrow\mathcal{A}(\Sigma')$ is
surjective.
\end{Cor}
\begin{proof}
(i)~ By Proposition \ref{Cor rooted}, $\Sigma\cong f^S(\Sigma)$,
and by (\ref{twoimagesubseed}),
$f^S(\Sigma)=\Sigma'_{I_0',I'_1}$. Then an injective seed
homomorphism is given.

(ii)~ By Proposition \ref{induceseed}(ii),
$\varphi:\Sigma^{(f)}\rightarrow\Sigma'$ is a seed homomorphism via
$\varphi(x)=f(x)$ for all $x\in \widetilde{X}^{(f)}$. By (\cite{ADS},Lemma 3.1),  $f(X)\supseteq X'$ and
$f(\widetilde{X})\supseteq \widetilde{X'}$. Then,
$\varphi(\widetilde{X}^{(f)})\cap \widetilde{X'}=f(\widetilde{X}^{(f)})\cap \widetilde{X'}=f(\widetilde{X})\cap
\widetilde{X'}=\widetilde{X'}$ and $\varphi(X^{(f)})\cap X'=f(X^{(f)})\cap X'=f(X)\cap X'=X'$. Thus, $\varphi(\widetilde{X}^{(f)})\supseteq \widetilde{X'}$ and $\varphi(X^{(f)})\supseteq X'$.  Therefore, we
have $\varphi(\Sigma^{(f)})=\Sigma'$. It follows that $\varphi$ is a
surjective seed homomorphism.
\end{proof}

\section{Sub-rooted cluster algebras and rooted cluster subalgebras  }

\subsection{Sub-rooted cluster algebras and two special cases}.

The following notion on the sub-structure of a rooted cluster algebra is a key in this research, which will be used to supply a unified view-point for the internal structure of a rooted cluster algebra.

\begin{Def}
 A rooted cluster algebra $\mathcal {A'}=\mathcal
{A}(\Sigma')$ is called a {\bf (mixing-type) sub-rooted cluster algebra
of type $(I_{0},I_{1})$} of the rooted cluster algebra $\mathcal
A=\mathcal A(\Sigma)$ if $\mathcal{A}(\Sigma')\cong\mathcal{A}(\Sigma_{I_0,I_1})$ in the category {\bf Clus}.
\end{Def}

From the definition of rooted cluster algebras, we can recognize the initial seed of a rooted cluster algebra as its ``{\em root}". So, it is natural for us to say the name of {\em sub-rooted cluster algebra} in the above definition since $\mathcal A$ is obtained from mixing-type sub-seed as the ``{\em  sub-root}".

By Proposition \ref{basiclem},  a rooted cluster algebra $\mathcal {A'}=\mathcal
{A}(\Sigma')$ is a  mixing-type sub-rooted cluster algebra
of type $(I_{0},I_{1})$ of  $\mathcal
A=\mathcal A(\Sigma)$ if and only if $\Sigma'\cong\Sigma_{I_0,I_1}$ in {\bf Seed}.

Now, we discuss two special cases of mixing-type sub-rooted cluster algebras of $\mathcal A(\Sigma)$.

{\bf Case $1$: $I_{1}=\emptyset$.}\; That is, the sub-seed is a pure sub-seed $\Sigma_{I_{0},\emptyset}=(X',\widetilde{B_0})$ of the seed $\Sigma$.

Since the extended clusters of $\Sigma$ and $\Sigma_{I_{0},\emptyset}$ are the same by $\widetilde{X'}=\widetilde{X}$, their fields of rational functions in the independent extended cluster variables are $\mathds F$ with coefficients in the rational field $\mathds Q$. But, the ground ring $\mathcal P$ of $\Sigma$ becomes a sub-ring of the ground ring $\mathcal P'$ of $\Sigma_{I_{0},\emptyset}$ generated by all frozen variables of $\Sigma_{I_{0},\emptyset}$ with unit, since the frozen variables of $\Sigma$ are only a part of the frozen variables of $\Sigma_{I_{0},\emptyset}$.

By Definition \ref{clusteralgebra}, from the seed $\Sigma_{I_{0},\emptyset}=(X',\widetilde{B_0})$, we get its associated cluster algebra $\mathcal A'=\mathcal A(\widetilde{B_0})$ over $\mathcal P'$ as the $\mathcal P'$-subalgebra of $\mathds F$ generated by all cluster variables in all seeds mutation equivalent to $\Sigma_{I_{0},\emptyset}$. An elementary fact is the following:

\begin{Prop}\label{clustersub}
{\rm $\mathcal A'=\mathcal A(\Sigma_{I_{0},\emptyset})$ is a subalgebra of the rooted cluster algebra $\mathcal A(\Sigma)$ over $\mathds{Q}$ as associative algebras.}
\end{Prop}

\begin{proof} Trivially, any frozen variables in $\mathcal{A'}$ are always in $\mathcal A$. For any exchangeable variable $x\in X'$ of $\mathcal A'$,
by Proposition \ref{submu}, there exists a positive isomorphism $\mu_{x}(\Sigma_{I_{0},\emptyset})\cong\mu_{x}(\Sigma)_{I_{0},\emptyset}$. But since $\widetilde{X'}=\widetilde{X}$, it is easy to see
that $\mu_{x}^{\Sigma_{I_{0},\emptyset}}(x)=\mu_{x}^{\Sigma}(x)$. Hence, $\mu_{x}(\Sigma_{I_{0},\emptyset})$ and
$\mu_{x}(\Sigma)_{I_{0},\emptyset}$ have the same cluster
variables. Therefore, the above positive isomorphism is in fact an identity, that is,
\begin{equation}\label{embedding}
\mu_{x}(\Sigma_{I_{0},\emptyset})=\mu_{x}(\Sigma)_{I_{0},\emptyset}
\end{equation}
 Then by induction,  any exchange cluster variable $y_{s}=\mu_{y_{s-1}}\cdots\mu_{y_{1}}(y_{1})$ is in $\mathcal{A}$ by (\ref{embedding}).
\end{proof}

For this reason, we say this sub-rooted cluster algebra $\mathcal A'=\mathcal A(\Sigma')$ to be a {\bf pure cluster sub-algebra} of $\mathcal A=\mathcal A(\Sigma)$ if $\mathcal A(\Sigma')\cong\mathcal A(\Sigma_{I_{0},\emptyset})$ in {\bf Clus}; equivalently, $\Sigma'\cong \Sigma_{I_{0},\emptyset}$ in {\bf Seed} for $I_0\subseteq X$.

Obviously, the rank of $\mathcal A'$ is $n_{1}$.

{\bf Case $2$: $I_{0}=\emptyset$.}\; That is, the sub-seed is a partial sub-seed $\Sigma_{\emptyset,I_1}=(X'',X''_{fr},\widetilde{B_1})$ of the seed $\Sigma$.

 A sub-rooted cluster algebra $\mathcal A'=\mathcal A(\Sigma')$ is called a {\bf pure sub-cluster algebra} of $\mathcal A=\mathcal A(\Sigma)$ if $\mathcal A(\Sigma')\cong\mathcal A(\Sigma_{\emptyset,I_1})$ in {\bf Clus}; equivalently, $\Sigma'\cong \Sigma_{\emptyset,I_1}$ in {\bf Seed} for some $I_1\subseteq \widetilde{X}$.

We give an example from \cite{ADS}. For two seeds
$\Sigma_{1}=(X_{1},(X_1)_{fr},\widetilde{B^{1}})$ and
$\Sigma_{2}=(X_{2},(X_2)_{fr},\widetilde{B^{2}})$ and their cluster algebras
$\mathcal{A}(\Sigma_{1})$ and $\mathcal{A}(\Sigma_{2})$, assume that there
exists $($possibly empty$)$ $\Delta_{1}\subseteq (X_{1})_{fr}$ and
$\Delta_{2}\subseteq (X_{2})_{fr}$ such that $\Sigma_{1}$ and
$\Sigma_{2}$ are glueable along $\Delta_{1}$ and $\Delta_{2}$.

Recall in \cite{ADS} that the {\bf amalgamated sum} along
$\Delta_{1}$ and $\Delta_{2}$ is defined as the rooted cluster algebra
$\mathcal{A}(\Sigma_{1})\amalg_{\Delta_{1},\Delta_{2}}
\mathcal{A}(\Sigma_{2})=\mathcal{A}(\Sigma)$ where $\Sigma=\Sigma_{1}\amalg_{\Delta_{1},\Delta_{2}} \Sigma_{2}$.

 $\mathcal
A(\Sigma_{i})$ ($i=1,2$) can be viewed easily  as pure sub-cluster algebras of
$\mathcal{A}(\Sigma_{1})\amalg_{\Delta_{1},\Delta_{2}}
\mathcal{A}(\Sigma_{2})$.

Denote $I_1=\{x_{s_1},\cdots,x_{s_l}\}$. Then, we can get a series of sub-cluster algebras as follows:
\begin{equation}\label{steppresent}
\mathcal A(\Sigma\backslash \{x_{s_1}\})=\mathcal A(\Sigma_{\emptyset,\{x_{s_1}\}}),\; \mathcal A(\Sigma\backslash \{x_{s_1}, x_{s_2}\})=\mathcal A(\Sigma_{\emptyset,\{x_{s_1}\}}\backslash \{x_{s_2}\})=\mathcal A(\Sigma_{\emptyset,\{x_{s_1}, x_{s_2}\}}),
\end{equation}
$$\cdots\cdots, \mathcal A(\Sigma\backslash \{x_{s_1}, x_{s_2}, \cdots, x_{s_l}\})=\mathcal A(\Sigma\backslash I_1)=\mathcal A(\Sigma_{\emptyset,\{x_{s_1},x_{s_2},\cdots,x_{s_{l-1}}\}}\backslash \{x_{s_l}\})=\mathcal A(\Sigma_{\emptyset,I_{1}})=\mathcal A''.$$

It is known that the exchange relation (\ref{exchangerelation}) for the adjacent cluster of $\mathcal A$ in direction $k\in [1,n]$ can be given equivalently using the following formula:
\begin{eqnarray}\label{ldser}
x_{k}x_{k}'=p_{k}^{+}\prod_{1\leq i\leq n;\;b_{ki}>0}x_{i}^{b_{ki}}+p_{k}^{-}\prod_{1\leq i\leq n;\;b_{ki}<0}x_{i}^{-b_{ki}},
\end{eqnarray}
where
\begin{eqnarray}\label{xs}
p_{k}^{+}=\prod_{1\leq i\leq m;\;b_{kn+i}>0}x_{n+i}^{b_{kn+i}},\;\;p_{k}^{-}=\prod_{1\leq i\leq m;\;b_{kn+i}<0}x_{n+i}^{-b_{kn+i}}
\end{eqnarray}
are, respectively, the products of frozen variables and their inverses.

On one hand, the field of rational functions in $\widetilde{X''}$, written as $\mathds F''$, is a sub-field of $\mathds F$ in $\widetilde{X}$ with coefficients in the rational field $\mathds Q$.

On the other hand, the ground ring $\mathcal P''$ of $\Sigma_{\emptyset,I_{1}}$, generated by the sub-set of the frozen variables with unit, is a sub-ring $\mathcal P$ of $\Sigma$.

By the definition of cluster algebras of geometric type, from the partial sub-seed $\Sigma_{\emptyset,I_{1}}=(X'',X''_{fr},\widetilde{B_1})$, its associated cluster algebra $\mathcal A''=\mathcal A(\widetilde{B_1})$ over $\mathcal P''$ is generated as the $\mathcal P''$-subalgebra of $\mathds F''$ generated by all cluster variables in all seeds mutation equivalent to $\Sigma_{\emptyset,I_{1}}$.

Referring to the equivalent form (\ref{ldser}) of the exchange relation (\ref{exchangerelation}), we can describe the exchange relation for the adjacent cluster of $\mathcal A''$ in direction $i_{k}$ for $k\in [1,n_{1}]$ using the following formula:
\begin{eqnarray}\label{ldser2}
x_{i_{k}}x_{i_{k}}'=p_{i_{k}}^{+}\prod_{1\leq t\leq n_{1};\;b_{i_{k}i_{t}}>0}x_{i_{t}}^{b_{i_{k}i_{t}}}+p_{i_{k}}^{-}\prod_{1\leq t\leq n_{1};\;b_{i_{k}i_{t}}<0}x_{i_{t}}^{-b_{i_{k}i_{t}}},
\end{eqnarray}
where
\begin{eqnarray*}
p_{i_{k}}^{+}=\prod_{1\leq l\leq n_{2};\;b_{i_{k}j_{l}}>0}x_{j_{l}}^{b_{i_{k}j_{l}}},\;\;p_{i_{k}}^{-}=\prod_{1\leq l\leq n_{2};\;b_{i_{k}j_{l}}<0}x_{j_{l}}^{-b_{i_{k}j_{l}}}.
\end{eqnarray*}

Note that the above $p_{i_{k}}^{+}$ and $p_{i_{k}}^{-}$ are the divisors of $p_{i_{k}}^{+}$ and $p_{i_{k}}^{-}$ in (\ref{xs}), respectively, and in (\ref{ldser2}), the products
\begin{eqnarray*}
\prod_{1\leq t\leq n_{1};\;b_{i_{k}i_{t}}>0}x_{i_{t}}^{b_{i_{k}i_{t}}},\;\;\prod_{1\leq t\leq n_{1};\;b_{i_{k}i_{t}}<0}x_{i_{t}}^{-b_{i_{k}i_{t}}}
\end{eqnarray*}
are, respectively, the divisors of the corresponding products in (\ref{ldser}). Hence, in general, the $x_{i_{k}}'$ in the adjacent cluster of $\mathcal A''$ is not the $x_{i_{k}}'$ in that of $\mathcal A$. Thus,

\emph{$\mathcal A''$ is not a subalgebra of the cluster algebra $\mathcal A(\Sigma)$ even over $\mathds Q$.}

Obviously, the rank of $\mathcal A''$ is also $n_{1}$, which is similar to that of $\mathcal A'$.

In general,  analogous to the pure sub-algebra $\mathcal A''$ above, a sub-rooted cluster algebra $\mathcal A(\Sigma_{I_0,I_1})$ is NOT
a subalgebra of the cluster algebra $\mathcal A(\Sigma)$ even over $\mathds Q$.

\subsection{Rooted cluster subalgebras as sub-class of sub-rooted cluster algebras  }.

We present a combinatorial characterization of rooted cluster subalgebras as a sub-class of mixing-type sub-rooted cluster algebras.

Given a cluster algebra $\mathcal{A}(\Sigma)$, we denote by\\
$\mathbb{A}_{1}(\Sigma)$:\; the set of all pure cluster subalgebras of
$\mathcal{A}(\Sigma)$,\\ $\mathbb{A}_{2}(\Sigma)$:\; the set of all
rooted cluster subalgebras of $\mathcal{A}(\Sigma)$ and,\\
$\mathbb{A}_{3}(\Sigma)$:\; the set of all mixing-type sub-rooted cluster algebras of $\mathcal{A}(\Sigma)$.

In the following discussion, we will have the inclusion relation:
\begin{equation}\label{propercluster}
\mathbb{A}_{1}(\Sigma)\subsetneqq\mathbb{A}_{2}(\Sigma)\subsetneqq\mathbb{A}_{3}(\Sigma).
\end{equation}

First, pure cluster subalgebras are special rooted cluster subalgebras in a rooted cluster algebra.
In fact, since $\mathds F(\Sigma_{I_0,\emptyset})=\mathds F(\Sigma)$, $\mathcal{A}(\Sigma_{I_0,\emptyset})$ is a subalgebra of $\mathcal{A}(\Sigma)$.
 Hence, we have the embedding $\bar f=id_{\mathcal{A}(\Sigma_{I_0,\emptyset})}:\; \mathcal{A}(\Sigma_{I_0,\emptyset})\hookrightarrow\mathcal{A}(\Sigma)$, and trivially,  the conditions CM1 and CM2 hold.
Using the condition $\mu_{x}(\Sigma_{I_{0},\emptyset})=\mu_{x}(\Sigma)_{I_{0},\emptyset}$
 for any $x\in X'$ and using induction,  the condition CM3 is satisfied. Hence,  $\bar f$ is an injective rooted cluster morphism. So,
we have:

\begin{Prop}\label{Lem rooted}
For a seed $\Sigma=(X,\widetilde{B})$ and its pure sub-seed $\Sigma_{I_0,\emptyset}=(X',\widetilde{B'})$ with $X'=X\backslash I_0$,
the pure cluster subalgebra $\mathcal{A}(\Sigma_{I_0,\emptyset})$ of $\mathcal{A}(\Sigma)$ is always a rooted cluster subalgebra of $\mathcal{A}(\Sigma)$.
\end{Prop}

Let $\Sigma=(X,B)$, where $X=(x_{1},x_{2})$, $B=\left(
                                                 \begin{array}{cc}
                                                   0 & 1 \\
                                                   -1 & 0 \\
                                                 \end{array}
                                               \right)
$ and $X_{fr}=\emptyset$ and $\Sigma'=(X',B')$, where $X'=\emptyset$ and
$X'_{fr}=\{x_{1}\}$. Then, $\Sigma'=\Sigma_{I_0,I_1}$ with $I_0=\{x_1\}$ and $I_1=\{x_2\}$, and $\mathcal{A}(\Sigma')$ is a rooted
cluster subalgebra but not a pure cluster subalgebra of $\mathcal{A}(\Sigma)$.

Then, the first strict inclusion relation in (\ref{propercluster}) follows.

A rooted cluster subalgebra $\mathcal
A(\Sigma')\in\mathbb{A}_{2}(\Sigma)$ is called {\bf proper} if it does not belong to $\mathbb{A}_{1}(\Sigma)$, that is, it is not a pure cluster sub-algebra.

 Note that we think $\mathcal A(\Sigma)$ as a special pure cluster sub-algebra of itself since $\Sigma=\Sigma_{\emptyset, \emptyset}$ for $I_{0}=\emptyset$ and then a proper rooted cluster subalgebra of $\mathcal A(\Sigma)$ never equals to $\mathcal A(\Sigma)$.

Now, we give a characterization of rooted cluster subalgebras as a sub-class of (mixing-type) sub-rooted cluster algebras in a
rooted cluster algebra $\mathcal{A}(\Sigma)$.

\begin{Thm}\label{rooted cluster subalgebra}
{\rm $\mathcal{A}(\Sigma')$ is a rooted cluster subalgebra of
$\mathcal{A}(\Sigma)$ if and only if there exists a mixing-type sub-seed  $\Sigma_{I_{0},I_{1}}$ of $\Sigma$ such that $\Sigma'\cong\Sigma_{I_{0},I_{1}}$ satisfies $b_{xy}=0$ for any $x\in X\setminus (I_0\cup I_1)$ and $y\in I_1$.}
\end{Thm}

\begin{proof}
``Only if'':\; Let $f: \mathcal{A}(\Sigma')\rightarrow
\mathcal{A}(\Sigma)$ be the injective rooted cluster morphism. By
Proposition \ref{Cor rooted}, we have $\mathcal{A}(\Sigma')\cong
\mathcal{A}(f^S(\Sigma'))$. By Proposition \ref{Cor rooted} (1),
$f^S(\Sigma')=\Sigma_{I_0,I_1}$ for $I_1=\widetilde{X}\backslash
(\widetilde{X}\cap f(\widetilde{X'}))$ and $I_0=X\backslash (f(X')\cup
I_1)$.
  Hence, $\Sigma'\cong\Sigma_{I_{0},I_{1}}$ by Proposition \ref{basiclem}.

 Now we show that the above sets $I_{0}$ and $I_{1}$ satisfy the condition in the theorem. Otherwise, there exists $x_{0}\in X\backslash (I_{0}\cup I_{1})$ and $y_{0}\in I_{1}$ such that $b_{x_{0}y_{0}}\neq 0$. Thus, in the rooted cluster algebra $\mathcal{A}(\Sigma)$,
$$\mu_{x_{0}, \Sigma}(x_{0})=\frac{\prod\limits_{y\in
\widetilde{X}, b_{x_{0}y}>0}y^{b_{x_{0}y}}+\prod\limits_{y\in \widetilde{X},
b_{x_{0}y}<0}y^{-b_{x_{0}y}}}{x_{0}},$$
 and in its sub-rooted cluster $(I_{0},I_{1})$-algebra $\mathcal{A}(\Sigma_{I_{0},I_{1}})$, we have
$$\mu_{x_{0},\Sigma_{I_{0},I_{1}}}(x_{0})=\frac{\prod\limits_{y\in
\widetilde{X}\backslash I_{1}, b_{x_{0}y}>0}y^{b_{x_{0}y}}+\prod\limits_{y\in
\widetilde{X}\backslash I_{1}, b_{x_{0}y}<0}y^{-b_{x_{0}y}}}{x_{0}}.$$
Owing to $b_{x_{0}y_{0}}\neq 0$,  the term $y_{0}^{b_{x_{0}y_{0}}}$ appears in the first equality but not in the second one. Since $\widetilde{X}$ is a transcendence basis of $\mathds F(\Sigma)$,  it follows that $\mu_{x_{0},\Sigma_{I_{0},I_{1}}}(x_{0})\neq \mu_{x_{0}, \Sigma}(x_{0})$, which contradicts the condition (CM3) for the injective rooted cluster morphism $f$.

``If'':\;  We know that $b_{xy}=0$ for any $x\in X\setminus (I_0\cup I_1)$ and $y\in I_1$ for the given $I_0$ and $I_1$. Let $\Delta=I_0 \cup X_{fr}$. Furthermore, let $\Sigma(I_{1}\cup\Delta)$
be the sub-seed of $\Sigma$ generated by the cluster variables of
$I_{1}$ and $\Delta$.  By definition of amalgamated sum, it is easy to see that
$\Sigma_{I_{0},\emptyset}=\Sigma(I_{1}\cup\Delta)_{I_0,\emptyset}\coprod_{\Delta_{1},\Delta_{2}}\Sigma_{I_{0},I_{1}}$ where $\Delta_{1}=\Delta_{2}=\Delta$. Thus, we get
$\mathcal{A}(\Sigma_{I_{0},\emptyset})=\mathcal{A}(\Sigma(I_{1}\cup\Delta)_{I_0,\emptyset})\coprod_{\Delta_{1},\Delta_{2}}\mathcal{A}(\Sigma_{I_{0},I_{1}}).$

By Lemma 4.13 in \cite{ADS}, $\mathcal{A}(\Sigma_{I_{0},I_{1}})$ is a rooted cluster
subalgebra of $\mathcal{A}(\Sigma_{I_{0},\Phi})$. By Proposition \ref{Lem rooted},  $\mathcal{A}(\Sigma_{I_{0},\Phi})$ is a rooted cluster subalgebra of
$\mathcal{A}(\Sigma)$. As the composition of two injective rooted cluster
morphisms, it follows that $\mathcal{A}(\Sigma_{I_{0},I_{1}})$ is a rooted cluster subalgebra
of $\mathcal{A}(\Sigma)$.
\end{proof}

This theorem tells us that in a cluster algebra, all rooted cluster
subalgebras form a {\em proper} sub-set of the set of rooted
sub-cluster algebras, that is, the second strict inclusion relation in (\ref{propercluster}) follows.

\begin{Rem}
According to Theorem \ref{rooted cluster subalgebra}, the existence
of rooted cluster subalgebras is dependent on the initial seed of
the rooted cluster algebra via, more precisely, mixing-type subseeds
of the initial seed. The following is an example to illustrate that
a rooted cluster subalgebra $\mathcal{A}(\Sigma_{I_0,I_1})$ of a
rooted cluster $\mathcal{A}(\Sigma)$ may not be isomorphic to any
rooted cluster subalgebra of $\mathcal{A}(\Sigma')$ anymore, for a seed $\Sigma'$, which is mutation equivalent to $\Sigma$.

Let
$Q$ be the quiver $1 \rightarrow 2 \rightarrow 3$ and the seed $\Sigma=\Sigma(Q)$.
Then $\mathcal{A}(\Sigma_{\{2\},\emptyset})$, a rooted cluster
subalgebra of $\mathcal{A}(\Sigma)$,  has $4$ cluster variables
and $1$ frozen variable. It is easy to see that
$\mathcal{A}(\mu_{2}(\Sigma))$ has no rooted cluster subalgebra
that possesses $4$ cluster variables
and $1$ frozen variable. Hence $\mathcal{A}(\Sigma_{\{2\},\emptyset})$
 is not isomorphic to any
rooted cluster subalgebra $\mathcal{A}((\mu_{2}(\Sigma))_{I_0,I_1})$ of $\mathcal{A}(\mu_{2}(\Sigma))$.
 \end{Rem}

\begin{Rem}
For any integer $m\geq 3$, denote by $\Pi_m$ the $m$-gon whose
points are labeled cyclically from $1$ to $m$. For $m\geq 4$, let
$\mathcal A(\Pi_m)$ be the cluster algebra from the fan
triangulation $T_m$ of $\Pi_m$ in Fig. 1, which is of type
$A_{m-3}$ with coefficients associated with boundary arcs. We
construct by induction a family $\{T_m\}_{m\geq 3}$ and then obtain a
family of cluster algebras $\{\mathcal A(\Pi_m)\}_{m\geq 3}$.
\begin{figure}[h] \centering
  \includegraphics*[120,488][322,642]{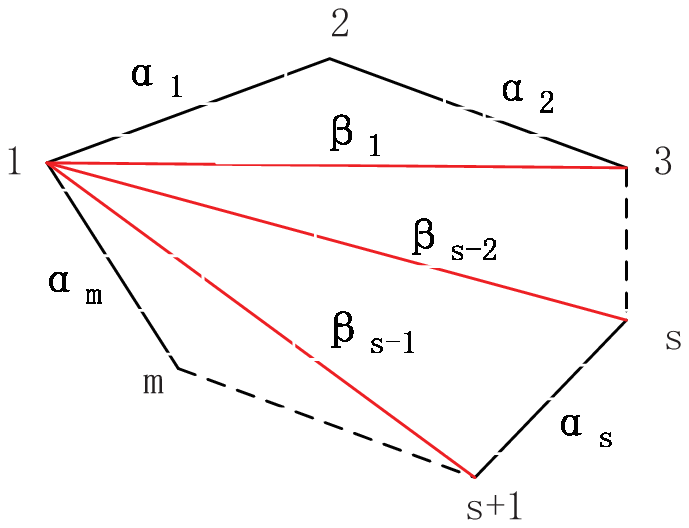}

 Figure 1
\end{figure}

For $m,m'$ satisfying $0<m<m'$, the inclusion of $T_m$ in $T_{m'}$
defines naturally the injective rooted cluster morphism $j_{m,m'}:
\mathcal{A}(\Pi_{m})\rightarrow \mathcal{A}(\Pi_{m'})$ given in
\cite{ADS}. Now we can interpret $j_{m,m'}$ by the language of a
mixing-type sub-rooted cluster algebra about some $(I_{0},I_{1})$.
Actually, as in Fig. 1, we denote $I_{0}$ to be the set
consisting of single vertex corresponding to the diagonal
$\beta_{m-2}$ and $I_{1}$ the set  consisting of frozen vertices
corresponding to the edges $\alpha_{m},\cdots, \alpha_{m'}$ and
exchangeable vertices corresponding to the diagonals
$\beta_{m-1},\cdots,\beta_{m'-3}$. Then, the rooted cluster subalgebra
$\mathcal{A}(\Pi_{m})$ of $\mathcal{A}(\Pi_{m'})$ is a mixing-type sub-rooted cluster subalgebra of $(I_{0},I_{1})$-type.

Obviously, in this example, $(I_0,I_1)$ satisfies the condition of Theorem \ref{rooted cluster
subalgebra}.
\end{Rem}

In order to use in the sequel, we introduce the so-called  {\em diagonal-unitization matrix} of an
extended exchange matrix from a cluster algebra.

\begin{Def} For an extended exchange
matrix $\widetilde{B}_{n\times(n+m)}$ of a cluster algebra,  we define its related {\bf diagonal-unitization matrix} $U\widetilde{B}$ to be a matrix $U\widetilde{B}=(c_{ij})_{n\times(n+m)}$ such that for any $i,j$,
\begin{eqnarray*}c_{ij}=\left\{\begin{array}{lll} b_{ij},& \text{if}
~~i\neq j;\\
1,& \text{if}~~i=j.\end{array}\right.\end{eqnarray*}
\end{Def}

Since all diagonal entries of $\widetilde{B}$ are zero due to its skew-symmetrizability, we have indeed  $$U\widetilde{B}=\widetilde{B}+(E_n\;O_{n\times m}),$$
where $E_n$ is an $n\times n$ identity matrix and $O_{n\times m}$ is a zero matrix. Note that all diagonal entries of $U\widetilde{B}$ are 1; in the sequel, one will see that $U\widetilde{B}$ is just a tool to judge when the row-index set and the column-index set of certain submartices are disjoint.

We need to understand a special case of cluster algebras, that is, a cluster algebra that is called {\bf trivial} if it has no exchangeable cluster variables except frozen cluster variables. All other cluster algebras are called {\bf non-trivial}.

 We will say a sub-matrix of a matrix to be an {\bf empty sub-matrix} if its either row-index set or column-index set is empty. To facilitate the statement of the conclusion, we think any empty submatrices are zero matrices.

\begin{Cor}\label{0 sub matrix}
\rm{ Using the above notations,  $\mathcal{A}(\Sigma')$ is a proper rooted cluster subalgebra  of $\mathcal{A}(\Sigma)$ if and only if there exist  $I'\subseteq X$ and $\emptyset\not=I_{1}\subseteq \widetilde{X}$ such that the $I'\times I_{1}$ sub-matrix of the diagonal-unitization matrix $U\widetilde{B}$ is a zero matrix and $\Sigma'\cong \Sigma_{I_{0},I_{1}}$ with $I_{0}=X\setminus (I'\cup I_1)$.}

In particular,
(i)~ a proper rooted cluster subalgebra $\mathcal{A}(\Sigma')$ of
$\mathcal{A}(\Sigma)$ is trivial if and only if it can be written as $\mathcal A(\Sigma_{I_0,I_1})$ with $I_{0}=X\setminus I_1$ and $I_{1}\neq \emptyset$ and

(ii)~ all proper rooted cluster subalgebras of
$\mathcal{A}(\Sigma)$ are trivial if and only if all entries of
$U\widetilde{B}$ are nonzero. In this case, there does not
exist non-trivial proper rooted cluster subalgebras.
\end{Cor}

\begin{proof}

``If'':\; In case $I'=\emptyset$, equivalently $I_0\cup I_1
\supseteq X$, it is easy to see $\mathcal{A}(\Sigma_{I_0,I_1})$
as a trivial rooted cluster subalgebra of $\mathcal{A}(\Sigma)$.
Furthermore, since $I_1\neq \emptyset$, $\mathcal{A}(\Sigma_{I_0,I_1})$ is a proper trivial rooted cluster subalgebra of $\mathcal{A}(\Sigma)$.

 In case $I'\neq \emptyset$, since $I'\subseteq X$,
$\emptyset\not=I_{1}\subseteq \widetilde{X}$
 and $I_{0}=X\setminus (I'\cup I_1)$, we have $X=I_{0}\cup I' \cup (X\cap I_{1})$.
 Moreover, the $I'\times I_{1}$ sub-matrix of the diagonal-unitization matrix
 $U\widetilde{B}$ is a zero matrix, which implies that $I'\cap I_{1}=\emptyset$.
 Thus, $I'=X\backslash (I_{0}\cup I_{1})$. As $b_{xy}=0$ for all $x\in I'=X\setminus (I_0\cup I_1)$ and $y\in
I_1$, according to Theorem \ref{rooted cluster subalgebra},
$\mathcal{A}(\Sigma_{I_0,I_1})$ is a rooted cluster subalgebra of
$\mathcal{A}(\Sigma)$. Moreover, as $I'\neq \emptyset$,
$\mathcal{A}(\Sigma_{I_0,I_1})$ is a proper non-trivial rooted
cluster subalgebra of $\mathcal{A}(\Sigma)$.

``Only If'':\;If $\mathcal{A}(\Sigma')$ is a
proper rooted cluster subalgebra  of $\mathcal{A}(\Sigma)$, then by
Theorem \ref{rooted cluster subalgebra}, there exist $I_{0}\subseteq
X$ and $I_{1}\subseteq \widetilde{X}$ satisfying $b_{xy}=0$ for any
$x\in X\backslash (I_{0}\cup I_{1})$ and $y\in I_{1}$ such that
$\mathcal{A}(\Sigma')\cong \mathcal{A}(\Sigma_{I_{0},I_{1}})$.

Now let $I'=X\backslash (I_{0}\cup I_{1})$, then $I'\cap I_1=\emptyset$ and the $I'\times
I_{1}$ sub-matrix of the diagonal-unitization matrix
$U\widetilde{B}$ is a zero matrix. Since $I_{0}$, $I_{1}$ and $I'\cap X$ are
pairwise disjoint, we have $I_{0}=X\setminus (I'\cup (I_1\cap
X))=X\setminus (I'\cup I_1)$. Finally, $I_{1}\neq \emptyset$ follows
 due to the fact that $\mathcal{A}(\Sigma')$ is proper.

(i)~ For a proper rooted cluster subalgebra
$\mathcal{A}(\Sigma')\cong \mathcal{A}(\Sigma_{I_{0},I_{1}})$, it is
trivial if and only if $I_{1}\neq \emptyset$ and $X\subseteq (I_0\cup I_1)$ and, equivalently, if and
only if $I_{1}\neq \emptyset$ and $I_{0}=X\setminus I_1$ since
$I_{0}\cap I_{1}=\emptyset$ and $I_{0}\subseteq X$.

(ii) ``If'': For any proper rooted cluster subalgebra
$\mathcal{A}(\Sigma')=\mathcal{A}(\Sigma_{I_{0},I_{1}})$, since all
entries of $U\widetilde{B}$ are nonzero,  the $I'\times I_1$
zero submatrix in the above  must be empty. But, $I_{1}\neq \emptyset$,
we have to get $I'=\emptyset$ and $I_{0}=X\setminus I_1$, and then
$X\subseteq (I_0\cup I_1)$ due to $I_0\cap
I_1=\emptyset$, which means $\Sigma_{I_0,I_1}$ has no exchangeable
cluster variable, i.e.
$\mathcal{A}(\Sigma')=\mathcal{A}(\Sigma_{I_{0},I_{1}})$ is
trivial.

``Only if": \;Otherwise, there exist $i$ and $j$ such that
$(i,j)$-entry $(U\widetilde{B})_{i,j}=0$. By choosing $I'=\{i\}$,
$I_{1}=\{j\}$ and $I_{0}=X\backslash \{i\}$ and then by Theorem \ref{rooted cluster subalgebra},
$\mathcal{A}(\Sigma_{I_{0},I_{1}})$ is a rooted cluster subalgebra
of $\mathcal{A}(\Sigma)$. Since $I_0\cap I_1=\emptyset$, we have $j\not=i$ and thus $I_{0}\not=X\backslash I_1$, which means
that $\mathcal{A}(\Sigma_{I_{0},I_{1}})$ is non-trivial by
(i) and is proper due to $I_1\not=\emptyset$. It contradicts with the given condition.
\end{proof}

In summary, in a rooted cluster algebra $\mathcal{A}(\Sigma)$, we have:
\begin{eqnarray*}
\mathbb{A}_2(\Sigma)
& = & \{  \text{rooted cluster subalgebras}\}\\
& = & \{ \text{pure cluster subalgebras}\} \sqcup \{\text{proper rooted cluster subalgebras}\}  \\
& = & \{  \text{pure cluster subalgebras}\} \sqcup \{  \text{trivial proper rooted cluster subalgebras}\} \\
& & \sqcup  \{\text{non-trivial  proper rooted cluster subalgebras}\}
\end{eqnarray*}
with $\{\text{ pure cluster subalgebras}\}=\{\mathcal{A}(\Sigma_{I_{0},I_{1}}):\;I_{1}=\emptyset, I_0\subseteq X\}$, \\
$\{\text{ trivial  proper rooted cluster subalgebras}\}=\{\mathcal{A}(\Sigma_{I_{0},I_{1}}):\;I_{1}\neq \emptyset, X\subseteq I_{0}\cup I_{1}\}$ and\\
$\{\text{  non-trivial  proper rooted cluster subalgebras}\}=\{\mathcal{A}(\Sigma_{I_{0},I_{1}}):\;I_{1}\neq \emptyset, X\nsubseteq I_{0}\cup I_{1}\}$,\\
where $\sqcup$ means the disjoint union of sets.

\begin{Rem}
(i)~ When $I'=\emptyset$, the zero sub-matrix $O_{I'\times I_{1}}$ is indeed an empty matrix and then $\mathcal A(\Sigma')$ is trivial.

(ii)~ If we assume $I_1=\emptyset$ in this corollary, then we get indeed
a pure cluster subalgebra $\mathcal A(\Sigma_{I_0,\emptyset})$ of
$\mathcal{A}(\Sigma)$ for $I_0=X\backslash I'$. Moreover, a pure cluster subalgebra $\mathcal A(\Sigma_{I_0,\emptyset})$ is trivial if and only if $I_{0}=X$ and $I_1=\emptyset$. This type of algebras is a unique trivial pure cluster subalgebra and is a special case of general (proper) trivial rooted cluster subalgebras.

\end{Rem}

The conclusion of Corollary \ref{0 sub matrix} can be stated equivalently: {\em a rooted cluster algebra $\mathcal A(\Sigma)$ has no proper non-trivial rooted cluster
subalgebras if and only if all entries of $U\widetilde{B}$ are nonzero.} Or say,
in case all entries of $U\widetilde{B}$ are nonzero, in $\mathbb{A}_{3}(\Sigma)$, the rooted cluster algebra $\mathcal A(\Sigma)$ has no non-trivial rooted cluster subalgebras except pure cluster subalgebras.

Meantime, the purpose of Corollary \ref{0 sub matrix} is to provide a detailed program to
construct proper rooted cluster subalgebras from a given rooted cluster
subalgebra $\mathcal{A}(\Sigma)$.

\begin{Ex} \rm{  Let $Q$ be the quiver of type $A_2$, that is,
$Q:\xymatrix{1\ar[r]^{}&2}$.
 For the corresponding cluster algebra $\mathcal A(\Sigma(Q))$ of $Q$, its diagonal-unitization matrix
 $U\widetilde{B}(Q)=\left(\begin{array}{cc}
1 & -1 \\
1 & 1
\end{array}\right).$
Obviously,   all the entries of $U\widetilde{B}(Q)$ are non-zero. The only proper rooted cluster subalgebras of $\mathcal{A}(\Sigma(Q))$ are:
 $$\mathcal{A}(\Sigma_{\{1\},\{2\}})\;\;\; \text{and}\;\;\; \mathcal{A}(\Sigma_{\{2\},\{1\}})$$
 which are both  trivial.}
\end{Ex}

\section{On enumeration and monoidal categorification}
\subsection{The number of rooted cluster subalgebras of the form $\mathcal{A}(\Sigma_{I_0,I_1})$}.

Using the result in Section 4, e.g. Theorem \ref{rooted cluster subalgebra}, and the above conclusions, we numberize some sub-classes in the internal structure of a rooted cluster algebra $\mathcal{A}(\Sigma)$ in
\textbf{Clus}.

\begin{Prop}\label{fornumber}
Let $\mathcal{A}(\Sigma)$ be a rooted cluster algebra and
$\widetilde{B}_{n\times (n+m)}$ be its extended exchange matrix.
Then the numbers of some sub-structures of $\mathcal{A}(\Sigma)$ in
\textbf{Clus} are given as follows:

~(1)~ The number of pure cluster sub-algebras $\mathcal{A}(\Sigma_{I_0,\emptyset})$ of
$\mathcal{A}(\Sigma)$ for $I_0\subseteq X$ is $2^{n}$.

~(2)~ The number of rooted cluster subalgebras $\mathcal{A}(\Sigma_{I_0,I_1})$ of
$\mathcal{A}(\Sigma)$ for $I_0\subseteq X$ and $I_1\subseteq \widetilde{X}$ satisfying Theorem \ref{rooted
cluster subalgebra} is
equal to the number $W$ of zero submatrices of $U\widetilde{B}$.
These zero submatrices include the empty sub-matrices in the form
$U\widetilde{B}_{J_0\times\emptyset}$ with $J_0\subseteq X$ or
$U\widetilde{B}_{\emptyset\times J_1}$ with $J_1\subseteq
\widetilde{X}$, where the numbers in these two forms are
respectively, $2^{n}$ and $2^{n+m}$.

~(3)~ The number of proper rooted cluster subalgebras $\mathcal{A}(\Sigma_{I_0,I_1})$ of
$\mathcal{A}(\Sigma)$ satisfying the condition in Corollary \ref{0 sub matrix} is
equal to $W-2^n$.
\end{Prop}

\begin{proof}
~(1)~ All pure cluster sub-algebras $\mathcal{A}(\Sigma_{I_0,\emptyset})$ of $\mathcal{A}(\Sigma)$ are
determined by $I_0\subseteq X$; hence, the required number is the chosen number of $I_0$,
that is, $C_{n}^{0}+C_{n}^{1}+\cdots+C_{n}^{n}=2^{n}$.

~(2)~ Denoting by $S_1$ the set of rooted cluster subalgebras of
$\mathcal{A}(\Sigma)$ in the form $\mathcal{A}(\Sigma_{I_0,I_1})$ and
by $S_2$ the set of all zero submatrices of
$U(\widetilde{B}_{n\times (n+m)})$, it suffices to set up a
bijection from $S_1$ to $S_2$.

Assume that $\mathcal{A}(\Sigma_{I_0,I_1})$ is  a rooted cluster
subalgebra of $\mathcal A(\Sigma)$. Then by Theorem \ref{rooted
cluster subalgebra}, $b_{xy}=0$ for all $y\in I_1$ and $x\in X\setminus
(I_0\cup I_1)$. Now we define a map $\varphi:S_1\rightarrow S_2$
with
$\varphi(\mathcal{A}(\Sigma_{I_0,I_1}))=U\widetilde{B}_{(X\setminus
(I_0\cup I_1))\times I_1}$, where $U\widetilde{B}_{(X\setminus
(I_0\cup I_1))\times I_1}$ is the  zero submatrix of
$U(\widetilde{B}_{n\times (n+m)})$ since $I_1\cap(X\setminus
(I_0\cup I_1))=\emptyset$.

On the one hand, for any $0$-submatrix $U\widetilde{B}_{J_0\times
J_1}$,  let $I_1= J_1$ and $I_0=X\setminus (J_0\cup J_1)$, so we
have $J_0=X\setminus (I_0\cup I_1)$ as $J_0\subseteq X$ and $J_0\cap
I_1=\emptyset$. By Theorem \ref{rooted cluster subalgebra},
$\mathcal{A}(\Sigma_{I_0,I_1})$ is a rooted cluster subalgebra of
$\mathcal{A}(\Sigma)$. Hence, we can define the map
$\phi:S_2\rightarrow S_1$ with $\phi(U\widetilde{B}_{J_0\times
J_1})=\mathcal{A}(\Sigma_{X\setminus(J_0\cup J_1),J_1})$.

Since $I_0\cap I_1=\emptyset$ and $I_0\subseteq X$, so $I_0\subseteq
X\setminus I_1$ and $X\setminus(X\setminus I_0)=I_0$. Thus,
$$X\setminus((X\setminus (I_0\cup I_1))\cup I_1)=X\setminus((X\setminus I_0)\cup I_1)=(X\setminus(X\setminus I_0))\cap (X\setminus I_1)=I_0.$$
Therefore, we have $\phi\varphi=id_{S_1}$.

On the other hand, since $J_0\cap J_1=\emptyset$, $J_0\subseteq X$ and
$X\setminus((X\setminus (J_0\cup J_1))\cup J_1)=J_0 ,$  we
have $\varphi\phi=id_{S_2}$. It follows that $\varphi$ is bijective.
Thus, the number of rooted cluster sub-algebras of
$\mathcal{A}(\Sigma)$ in the form $\mathcal{A}(\Sigma_{I_0,I_1})$ is
equal to the number of zero sub-matrices of $U\widetilde{B}$.

All the empty sub-matrices in the form
$U\widetilde{B}_{J_0\times\emptyset}$ with $J_0\subseteq X$ are
corresponding to pure sub-rooted cluster algebras of
$\mathcal{A}(\Sigma)$ in the form $\mathcal{A}(\Sigma_{I_0,I_1})$.
Thus, by (1), the number of such special zero submatrices is equal to
$2^n$.

All empty submatrices of $U\widetilde{B}$ in the form
$U\widetilde{B}_{\emptyset\times J_1}$ are determined uniquely by
the choice of $J_1$; hence, the number of such special zero
submatrices is equal to $C^0_{n+m} + C^1_{n+m} + \cdots +
C^{n+m}_{n+m} = 2^{n+m}$.

(3)~ The set of sub-rooted cluster algebras of $\mathcal{A}(\Sigma)$
of the form $\mathcal{A}(\Sigma_{I_0,I_1})$ is the disjoint union of the
subset of the proper ones and that of pure sub-rooted cluster
algebras of $\mathcal{A}(\Sigma)$. Hence, this result follows
directly from (1) and (2). \end{proof}

\begin{Rem}
 In Proposition \ref{fornumber} (2), the corresponding rooted
cluster algebras of the empty sub-matrices of the form
$U\widetilde{B}_{\emptyset\times J_1}$ with $J_1\subseteq
\widetilde{X}$ are just the trivial rooted cluster subalgebras of
$\mathcal{A}(\Sigma)$ that compose a part of sub-rooted cluster
algebras of $\mathcal{A}(\Sigma)$ of the form
$\mathcal{A}(\Sigma_{I_0,I_1})$.
\end{Rem}

\subsection{Monoidal sub-categorification of a rooted cluster algebra}.

Let us recall the definition of the {\bf monoidal categorification}
of a cluster algebra (refer to \cite{KKKO}). For a field $K$, let $\mathcal{M}$ be a
$K$-linear abelian monoidal category, where {\em $K$-linearity} means that the
tensor functor $\otimes$ is $K$-linear and exact. Moreover, we assume
 $\mathcal{M}$ to satisfy that (i) any object of $\mathcal{M}$ is of
 finite length and (ii) $K \cong
\rm{Hom}_\mathcal{M}(S,S)$ for any simple object $S$ of
$\mathcal{M}$.

\begin{Def}(\cite{KKKO})
Let $\mathscr{I}=(\{M_i\}_{i=1}^{n+m},\widetilde{B})$ be a pair of a
family $\{M_i\}_{i=1}^{n+m}$ of simple objects in $\mathcal{M}$ and
an integer-valued $n\times(n+m)$-matrix
$\widetilde{B}=(b_{M_i,M_j})_{i=1,\cdots,n;j=1,\cdots,n+m}$ whose
principal part is skew-symmetric.  We call $\mathscr{I}$ a {\bf
monoidal
seed} in $\mathcal{M}$ if
(i) $M_i\otimes M_j\cong M_j\otimes M_i$ for any $1\leq i,j\leq
n+m$ and
(ii) $\bigotimes\limits_{i=1}^{n+m}M_i^{\otimes a_i}$ is simple for
any $(a_i)\in \mathbb{Z}_{\geq 0}^{n+m}$.
\end{Def}

\begin{Def}\label{monomuta}(\cite{KKKO})
For $1\leq k\leq n$, we say that a monoidal seed
$\mathscr{I}=(\{M_i\}_{i=1}^{n+m},\widetilde{B})$ admits a {\bf
mutation $\mu_{M_k}$ in direction $M_k$} if there exists a simple
object $M'_{k}\in
\mathcal{M}$ such that\\
(i)~ there exist exact sequences in $\mathcal{M}$:\\
\centerline{$0\rightarrow
\bigotimes\limits_{b_{M_kM_i}>0}M_i^{\otimes b_{M_kM_i}}\rightarrow
M_k\otimes M'_k\rightarrow
\bigotimes\limits_{b_{M_kM_i}<0}M_i^{\otimes (-b_{M_kM_i})}\rightarrow 0$,}\\
\centerline{$0\rightarrow
\bigotimes\limits_{b_{M_kM_i}<0}M_i^{\otimes
(-b_{M_kM_i})}\rightarrow M'_k\otimes M_k\rightarrow
\bigotimes\limits_{b_{M_kM_i}>0}M_i^{\otimes b_{M_kM_i}}\rightarrow 0$;}\\
(ii)~ the pair $\mu_{M_k}(\mathscr{I}):=(\{M_i\}_{i\neq
k}\cup\{M'_k\},\mu_{M_k}(\widetilde{B}))$ is a monoidal seed in
$\mathcal{M}$. \\
In this case, we denote $\mu_{M_k}(M_k)\overset{\Delta}{=}M'_k$ and
$\mu_{M_k}(M_i)\overset{\Delta}{=}M_i$ if $i\neq k$.
\end{Def}

Similarly, as in the case of cluster algebras, for any $I_0\subseteq
\{M_1,\cdots, M_n\}$ and
 $I_1 \subseteq \{M_1,\cdots, M_{n+m}\}$ with $I_0\cap I_1=\emptyset$, we
 can define $\mathscr{I}_{I_0,I_1}$ for any monoidal seed
 $\mathscr{I}$. More precisely, $\mathscr{I}_{I_0,I_1}=(\{M_i|{M_i\not\in I_1}\},\widetilde{B'})$, where $\widetilde{B'}$ is the matrix obtained from $\widetilde{B}$ by deleting the
 $I_0\cup I_1$ rows and $I_1$ columns.

\begin{Def}\label{monoidalcat}(\cite{KKKO})
Using the notations above, $\mathcal{M}$ is called a
{\bf monoidal categorification} of a cluster algebra $\mathcal{A}=\mathcal{A}(\Sigma)$ if \\
(a) there is a ring isomorphism $\varphi: K_0(\mathcal{M}) \cong\mathcal{A}$, where $K_0(\mathcal{M})$ is the Grothendieck ring of $\mathcal M$, \\
(b) there exists a monoidal seed
$\mathscr{I}=(\{M_i\}_{i=1}^{n+m},\widetilde{B})$ in $\mathcal{M}$
such that $[\mathscr{I}]:=(\{[M_i]\}_{i=1}^{n+m},\widetilde{B})$ is
the initial seed $\Sigma$ of $\mathcal{A}$ and $\mathscr{I}$ admits
successive mutations in all directions.
\end{Def}

\begin{Rem}\label{rem5.6}
If a rooted cluster algebra $\mathcal{A}(\Sigma)$ admits a monoidal
categorification $\mathcal{M}$ such that the monoidal seed
$\mathscr{I}=(\{M_i\},\widetilde{B})$ corresponds to
$\Sigma=(\{x_i\},\widetilde{B})$, we always assume that $M_i$
correspond to $x_i$ for all $i$. According to the definition of
mutation of monoidal seeds, it is easy to see that
$$\varphi([\mu_{M_i}(M)])=\mu_{x_i}(\varphi([M]))$$
for all $x_i\in X$.
\end{Rem}

Following the above preparations,  we find the relation between
the categorification of a rooted cluster algebra
$\mathcal{A}(\Sigma)$ and the categorification of a rooted cluster
subalgebra of $\mathcal{A}(\Sigma)$.

\begin{Thm}\label{subcat}
\rm{Let the abelian monoidal category $\mathcal{M}$ be a monoidal
categorification of a cluster algebra $\mathcal{A}(\Sigma)$ and
$\mathcal{A}(\Sigma_{I_0,I_1})$ be a rooted cluster subalgebra of
$\mathcal{A}(\Sigma)$. Then,

(i)~ $\mathcal{A}(\Sigma_{I_0,I_1})$ has a monoidal
categorification $\mathcal{M'}$ that is also an abelian monoidal
sub-category  of $\mathcal{M}$;

(ii)~  For the Grothendieck rings $K_{0}(\mathcal{M})$ and $K_{0}(\mathcal{M'})$,  the following
diagram commutes via ring homomorphisms:}
 \\
 \centerline{$
 \begin{array}[c]{ccc}
 K_{0}(\mathcal{M'})&\stackrel{i_{1}}{\hookrightarrow}&K_{0}(\mathcal{M})\\
 \downarrow\scriptstyle{\cong\varphi'}&&\downarrow\scriptstyle{\cong\varphi}\\
 \mathcal{A}(\Sigma_{I_0,I_1})&\stackrel{i_{2}}{\hookrightarrow}&\mathcal{A}(\Sigma)
 \end{array}
 $,}
\rm{where $i_{1}$ and $i_{2}$ mean the injective ring homomorphisms.}
\end{Thm}

\begin{proof} (i)~
 Since $\mathcal{A}(\Sigma_{I_0,I_1})$ is a rooted cluster subalgebra of $\mathcal{A}(\Sigma)$, by (\cite{ADS}, Corollary 4.6(2)),
  the cluster variables of $\mathcal{A}(\Sigma_{I_0,I_1})$ are cluster variables of $\mathcal{A}(\Sigma)$. Let $\varphi: K_0(\mathcal{M}) \cong\mathcal{A}(\Sigma)$ be the ring isomorphism by Definition \ref{monoidalcat} and let
 $$\mathfrak{S}=\{S\in
\mathcal{M}|\; \varphi([S])\; \text{is a cluster variable of}\;
\mathcal{A}(\Sigma_{I_0,I_1})\}.$$ Denote by $\mathcal{M'}$ the
fully faithful abelian monoidal subcategory of $\mathcal{M}$
generated by $\mathfrak{S}$. Now we prove that $\mathcal{M'}$ is the
monoidal categorification of $\mathcal{A}(\Sigma_{I_0,I_1})$.

Let $W'$ be the set of cluster variables in
$\mathcal{A}(\Sigma_{I_0,I_1})$. We denote $[\mathfrak{S}]=\{[S]|S\in
\mathfrak{S}\}$. Then, $\varphi([\mathfrak{S}])=W'$.

(a)~ Since $\mathcal{M'}$ is the fully faithful abelian monoidal
subcategory of $\mathcal{M}$ generated by $\mathfrak{S}$, we have
$K_{0}(\mathcal{M'})$ as the sub-ring
 of $K_{0}(\mathcal{M})$ generated by $[\mathfrak{S}]$. Let $\varphi'=\varphi|_{K_{0}(\mathcal{M'})}$. Since $\mathcal{A}(\Sigma_{I_0,I_1})$ is generated by $W'=\varphi([\mathfrak{S}])$ as a ring,  we have $\varphi'(K_{0}(\mathcal{M'}))=
 \mathcal{A}(\Sigma_{I_0,I_1})$. Since $\varphi'$ is the restriction of the isomorphism $\varphi$, we get $\varphi': K_{0}(\mathcal{M'})\cong
 \mathcal{A}(\Sigma_{I_0,I_1})$, that is, Definition \ref{monoidalcat} (a) holds for $\mathcal M'$ and $\mathcal A(\Sigma_{I_0,I_1})$, and then the commute diagram in the proposition follows.

(b)~ Since $\mathcal{M}$ is the monoidal categorification of
$\mathcal{A}(\Sigma)$, by definition, there exists a monoidal seed
$\mathscr{I}=(\{M_i\},\widetilde{B})$ such that
$[\mathscr{I}]=\Sigma$. Set $\overline{I}_0=\{M\in
\{M_i\}|\;\varphi([M])\in I_0 \}$ and $\overline{I}_1=\{M\in
\{M_i\}|\;\varphi([M])\in I_1 \}$, so that all objects $M_i\in
\mathscr{I}_{\overline{I}_0, \overline{I}_1}$ are in one-one
correspondence to cluster variables in $\Sigma_{I_0,I_1}$.
$\mathscr{I}_{\overline I_0,\overline I_1}$ is a monoidal seed in $\mathcal{M'}$ by
definition, and furthermore,
$[\mathscr{I}_{\overline{I}_0,\overline{I}_1}]=\Sigma_{I_0,I_1}$.

For any $M_k\in \{M_1,M_2,\cdots,M_n\}\setminus
(\overline{I}_0\cup \overline{I}_1)$, since $\mathcal{M}$ is the
monoidal categorification of $\mathcal{A}(\Sigma)$, there exist $\mu_{M_k}(M_k)\in \mathcal{M}$ and two exact sequences in $\mathcal{M}$:\\
\begin{equation}\label{eq:ex1}
0\rightarrow \bigotimes\limits_{b_{M_kM_i}>0}M_i^{\otimes
b_{M_kM_i}}\rightarrow M_k\otimes \mu_{M_k}(M_k)\rightarrow
\bigotimes\limits_{b_{M_kM_i}<0}M_i^{\otimes
(-b_{M_kM_i})}\rightarrow 0,
\end{equation}
\begin{equation}\label{eq:ex2}
0\rightarrow \bigotimes\limits_{b_{M_kM_i}<0}M_i^{\otimes
(-b_{M_kM_i})}\rightarrow \mu_{M_k}(M_k)\otimes M_k\rightarrow
\bigotimes\limits_{b_{M_kM_i}>0}M_i^{\otimes b_{M_kM_i}}\rightarrow
0.
\end{equation}
Since $\mathcal{A}(\Sigma_{I_0,I_1})$ is a rooted cluster subalgebra of
$\mathcal{A}(\Sigma)$, according to Theorem \ref{rooted cluster
subalgebra}, for all $M_i\in \{M_j\}_{j=1}^{n+m}$, if $b_{M_kM_i}\neq 0$,
 we have $M_i\not\in \overline{I}_1$; thus, $M_i\in
\mathcal{M}'$ in (\ref{eq:ex1}) and
(\ref{eq:ex2}).

By Remark \ref{rem5.6},
$\varphi([\mu_{k}(M_k)])=\mu_{k}(\varphi([M_k]))$, which is a
cluster variable in $\mathcal{A}(\Sigma_{I_0,I_1})$. Hence,
$\mu_{k}(M_k)\in \mathfrak{S}\subseteq Ob(\mathcal{M'})$.

Therefore, the exact sequences (\ref{eq:ex1}) and (\ref{eq:ex2}) are
also in $\mathcal{M'}$. Thus, for $\mathcal M'$, there is a mutation
$\mu'_{M_k}$ in direct $k$ satisfying
$\mu'_{M_k}(\mathscr{I}_{\overline{I}_0,\overline{I}_1})=((\{M_i|i\not\in
\overline{I}_1\}\setminus \{M_k\})\cup
\{\mu_{M_k}(M_k)\},\mu_{M_k}(\widetilde{B'}))$ with
$\mu'_{M_k}(M_k)=\mu_{M_k}(M_k)$ and $\mu'_{M_k}(M_i)=M_i$ for $i\not =k$ and
$M_i\not\in \overline{I}_1$. Thus,
$\mathscr{I}_{\overline{I}_0,\overline{I}_1}$ admits successive
mutations in all directions by induction.

From (a), (b), by Definition \ref{monoidalcat}, $\mathcal{M'}$ is the
monoidal categorification of $\mathcal{A}(\Sigma_{I_0,I_1})$. So, (i) holds.

(ii)~ As mentioned in the proof of (i), the cluster variables of $\mathcal{A}(\Sigma_{I_0,I_1})$ are also cluster variables of $\mathcal{A}(\Sigma)$. By definition, $\varphi'=\varphi|_{K_0(\mathcal M')}$. Hence, the commutativity of the diagram follows naturally.
\end{proof}

Analog to $\Sigma_{I_0,I_1}$ in Proposition \ref{submu}, we can prove that
$\mu_{M_k}(\mathscr{I}_{\overline{I}_0,\overline{I}_1})=(\mu_{M_k}(\mathscr{I}))_{\overline{I}_0,\overline{I}_1}$.

The monoidal
category $\mathcal M'$ is called a {\bf monoidal
sub-categorification} of the cluster algebra $\mathcal A(\Sigma)$.

 By Theorem \ref{rooted cluster subalgebra}, for the rooted cluster
subalgebra $\mathcal{A}(\Sigma')$, there is a mixing-type sub-seed
$\Sigma_{I_{0},I_{1}}$ for $I_1=\widetilde{X}\backslash
(\widetilde{X}\cap f(\widetilde{X'})), I_0=X\backslash (f(X')\cup
I_1)$ satisfying $b_{xy}=0$ for any $x\in X\setminus (I_0\cup I_1)$ and
$y\in I_1$ such that $\Sigma'\cong\Sigma_{I_{0},I_{1}}$. Hence,
 the monoidal categorification $\mathcal M'$ of $\mathcal
A(\Sigma')$ can be obtained by using Theorem \ref{subcat}.

The remaining interesting question is:

 (a)~ {\em What is the relationship between the monoidal
categorification $\mathcal M'$ of an arbitrary sub-rooted cluster
algebra isomorphic to $\mathcal A(\Sigma_{I_0,I_1})$ and
$\mathcal M$ of the (rooted) cluster algebra $\mathcal A(\Sigma)$? }

 In particular, dually, when the abelian monoidal category $\mathcal M'$ is
a quotient category of $\mathcal M$ with the natural quotient functor $\pi:\mathcal{M}\rightarrow\mathcal{M'}$ preserving the
tensor product, we call $\mathcal M'$ a {\bf monoidal quotient
categorification} of the (rooted) cluster algebra $\mathcal
A(\Sigma)$. Therefore, we have the other question as follows:

 (b)~ {\em For a sub-rooted cluster algebra $\mathcal A(\Sigma')$
isomorphic to $\mathcal A(\Sigma_{I_0,I_1})$, what is the condition
of $\mathcal A(\Sigma')$ that makes its monoidal categorification
$\mathcal M'$ to be a monoidal quotient categorification of the
(rooted) cluster algebra $\mathcal A(\Sigma)$? }

As a partial answer of (b), a necessary condition is that $\mathcal{A}(\Sigma')$ is a
quotient algebra of $\mathcal{A}(\Sigma)$. Indeed, if $\mathcal M'$ is a monoidal quotient categorification  of $\mathcal A(\Sigma)$, by definition, we have  the algebra isomorphisms $K_0(\mathcal{M})\cong \mathcal{A}(\Sigma)$ and
$K_0(\mathcal{M'})\cong \mathcal{A}(\Sigma')$, and since the quotient functor $\pi:
\mathcal{M}\rightarrow\mathcal{M'}$ preserves tensor product, we obtain a surjective algebra homomorphism
$[\pi]:K_0(\mathcal{M})\rightarrow K_0(\mathcal{M'})$ naturally, which means that $\mathcal{A}(\Sigma')$ is a
quotient algebra of $\mathcal{A}(\Sigma)$.

These two questions will be discussed more in our further work.

\section{Rooted cluster quotient algebras   }

\subsection{Rooted cluster quotient algebras via pure sub-cluster algebras }.

We use the same notations in \cite{ADS}. Let $\Sigma\setminus
\{x\}=(X\setminus \{x\}, \widetilde{B}\setminus\{x\})$ denote the
seed obtained from $\Sigma$ by deleting $x\in \widetilde{X}$, where
$\widetilde{B}\setminus\{x\}$ means the matrix obtained by deleting the row
and column labeled by $x$ from $\widetilde{B}$ (when
$x\in\widetilde{X}\setminus X$, $\widetilde{B}$ has no row labeled
by $x$, so $\widetilde{B}\setminus\{x\}$ is obtained by deleting
only the column labelled by $x$ from $\widetilde{B}$). Denote
$\sigma_{x,1}$ as the unique algebra homomorphism from
$\mathcal{A}(\Sigma)$ to $\mathds F(\Sigma\setminus \{x\})$, which
sends $y$ to $y$ for $y\in \widetilde{X}\setminus \{x\}$ and $x$ to
1. Hence,
$\sigma_{x,1}:\mathcal{A}(\Sigma)\rightarrow\mathcal{A}(\Sigma\setminus\{x\})$
is an algebra homomorphism if and only if
$\sigma_{x,1}(\mathcal{A}(\Sigma))\subseteq
\mathcal{A}(\Sigma\setminus\{x\})$. We call $\sigma_{x,1}$ the {\bf simple specialisation} of $\mathcal A(\Sigma)$ at $x$. In the sequel, we will need to consider the composition of some simple specialisations for $x$ in a certain subset $I\subseteq\widetilde{X}$, i.e. $\sigma_{I,1}=\prod\limits_{x\in I}\sigma_{x,1}$, which is called the {\bf specialisation} of $\mathcal A(\Sigma)$ at $I$.

Denote by $\sigma_{x,1}(\mu_y(\Sigma)\setminus \{x\})$
the seed
$(\{\sigma_{x,1}(\mu_{y}(x_1)),\cdots,\sigma_{x,1}(\mu_{y}(x_n))\}\setminus
\{1\},\; \mu_{y}(\widetilde{B}\setminus\{x\}) )$ in case $x\neq
y\in X$.

In (\cite{ADS}, Problem 6.10), the authors propose a problem
on whether $\sigma_{x,1}$ induces a surjective ideal rooted cluster
morphism from $\mathcal{A}(\Sigma)$
 to $\mathcal{A}(\Sigma\setminus\{x\})$.

 (\cite{ADS}, Proposition 6.9) says that $\sigma_{x,1}$ is
a
 surjective ideal rooted cluster morphism from $\mathcal{A}(\Sigma)$
 to $\mathcal{A}(\Sigma\setminus\{x\})$ if and only if
 $\sigma_{x,1}$ is an algebra homomorphism and then if and only if $\sigma_{x,1}(\mathcal{A}(\Sigma))\subseteq\mathcal{A}(\Sigma\setminus\{x\})$.

 (\cite{ADS}, Corollary 6.14 and Theorem 6.15) tells us that if the seed $\Sigma$ is either acyclic or arising from a surface with certain condition on $x$,
  then the homomorphism $\sigma_{x,1}$ is surjective.

More generally, (\cite{ADS}, Theorem 6.17) shows that if
$\mathcal{A}(\Sigma)$ admits a 2-CY categorification, in particular,
if the skew-symmetric initial seed $\Sigma$ is mutation equivalent to an acyclic seed or say that $\mathcal{A}(\Sigma)$ is acyclic,
then $\sigma_{x,1}$ induces an algebra homomorphism between the
cluster character algebras.

Associated with these, in the sequel, we obtain the fact in Lemma \ref{acyclicsub1} that for a totally sign-symmetric seed $\Sigma$ if $\mathcal A(\Sigma\setminus \{x\})$ is acyclic, then $\sigma_{x,1}$ induces an algebra homomorphism between
cluster algebras. This fact is used to get Theorem \ref{cluster quo.},  the main result of this section.

According to (\cite{ADS}, Corollary 6.14)  and (\ref{steppresent}),
we have the following.

\begin{Prop}\label{acyclicsub}
For an acyclic seed $\Sigma$ and $I_1\subset \widetilde{X}$, the pure sub-cluster algebra $\mathcal A(\Sigma_{\emptyset, I_1})$ is a rooted cluster quotient algebra of the rooted cluster algebra $\mathcal A(\Sigma)$.
\end{Prop}

It is known well that a (rooted) cluster algebra $\mathcal A(\Sigma)$ is called {\bf acyclic} if the initial seed $\Sigma$ is mutation equivalent to an acyclic seed.
Hence, the above proposition means that for an acyclic rooted cluster algebra $\mathcal A(\Sigma)$ with an acyclic initial seed, its pure sub-cluster algebras are
always rooted cluster quotient algebras. Now, we can discuss the general case of an acyclic rooted cluster algebra $\mathcal A(\Sigma)$, that is, its acyclic seed does not need to be the initial seed $\Sigma$.

\begin{Lem}\label{mutation2}
{\rm Using the foregoing notations, let $(y_1,\cdots,y_s)$ be
a $\Sigma$-admissible sequence with $y_i\neq x\in \widetilde{X}$ for all $1\leq i\leq
s$. Denote by $\mu$ and $\mu'$, respectively, the mutations in the
cluster algebras $\mathcal{A}(\mu_{y_s}\cdots\mu_{y_1}(\Sigma))$ and
$\mathcal{A}(\sigma_{x,1}(\mu_{y_s}\cdots\mu_{y_1}(\Sigma)\setminus
\{x\}))$. Then for any $y\in \mu_{y_s}\cdots\mu_{y_1}(X)$ and $z\in
\mu_{y_s}\cdots\mu_{y_1}(\widetilde{X})$ with $y,z\neq x$, it holds
that}

\rm (i)~
$\sigma_{x,1}(\mu_{y}(z))=\mu'_{\sigma_{x,1}(y)}(\sigma_{x,1}(z))$;

\rm (ii)~
$\sigma_{x,1}(\mu_{y}\mu_{y_s}\cdots\mu_{y_1}(\Sigma)\setminus\{x\})=\mu'_{\sigma_{x,1}(y)}(\sigma_{x,1}(\mu_{y_s}\cdots\mu_{y_1}(\Sigma)\setminus\{x\}))$ and

\rm (iii)~ Any $\Sigma\setminus\{x\}$-admissible sequence
$(z_1,\cdots,z_t)$ can be lifted to a
$(\sigma_{x,1},\Sigma,\Sigma\setminus\{x\})$-biadmissible sequence
$(w_1,\cdots,w_t)$ satisfying that $\sigma_{x,1}(w_i)=z_i$ for $i=1,\cdots,t$ and
$$\sigma_{x,1}(\mu_{w_{t}}\cdots\mu_{w_1}(\Sigma) \backslash
\{x\})=\mu'_{z_t}\cdots\mu'_{z_1}(\Sigma\setminus\{x\}).$$

\end{Lem}

\begin{proof}
We denote $\mu_{y_s}\cdots\mu_{y_1}(\Sigma)$ by
$\Sigma'=(X',\widetilde{B'})$.

(i) In the cluster algebra $\mathcal A(\Sigma')$, since
\begin{eqnarray*}\mu_{y}(z)=\left\{\begin{array}{lll} \frac{\prod\limits_{t\in \widetilde{X'}, b'_{yt}>0}t^{b'_{yt}}+\prod\limits_{t\in \widetilde{X'}, b'_{yt}<0}t^{-b'_{yt}}}{y},& \text{if}
~~z=y;\\z,& \text{if}~~ z\neq y, \end{array}\right.\end{eqnarray*}
we have
\[\begin{array}{ccl} \sigma_{x,1}(\mu_{y}(z))
 & = & \left\{\begin{array}{lll}\frac{\prod\limits_{t\in \widetilde{X'}, b'_{yt}>0}\sigma_{x,1}(t)^{b'_{yt}}+\prod\limits_{t\in \widetilde{X'}, b'_{yt}<0}\sigma_{x,1}(t)^{-b'_{yt}}}{\sigma_{x,1}(y)},&\text{if} ~~z=y;\\ \sigma_{x,1}(z),& \text{if}~~ z\neq
  y.\end{array}\right.
\end{array}\]

On the other hand, in the seed
$\sigma_{x,1}(\Sigma'\backslash\{x\})$,

\begin{eqnarray*}\mu'_{\sigma_{x,1}(y)}(\sigma_{x,1}(z))=
\left\{\begin{array}{lll} \frac{\prod\limits_{t\in \widetilde{X'},
b'_{yt}>0, t\neq x}\sigma_{x,1}(t)^{b'_{yt}}+\prod\limits_{t\in
\widetilde{X'}, b'_{yt}<0, t\neq
x}\sigma_{x,1}(t)^{-b'_{yt}}}{\sigma_{x,1}(y)},&
\text{if} ~~z=y;\\
\sigma_{x,1}(z),& \text{if}~~ z\neq y.\end{array}\right.
\end{eqnarray*}
Therefore, the result follows.

(ii)   Denote
\begin{equation}\label{twosubseeds}
\sigma_{x,1}(\mu_{y}(\Sigma')\setminus\{x\})=(X'',
\widetilde{B''})\;\;\;\;\text{and} \;\;\;\;
\mu'_{\sigma_{x,1}(y)}(\sigma_{x,1}(\Sigma\setminus\{x\}))=(X''',
\widetilde{B'''}).
\end{equation}
 Then by definition, we have
$$X''=  \{\sigma_{x,1}(\mu_{y}(x'_i)):\; i=1,\cdots,n, \mu_{y}(x'_i)\not=x \}  =(\sigma_{x,1}(X')\setminus\{x,\sigma_{x,1}(y)\})\cup \{\sigma_{x,1}(\mu_{y}(y))\}$$
and by
the definitions of $\sigma_{x,1}$  and the mutation $\mu'$, we have
$$X'''=((\sigma_{x,1}(X')\backslash \{x\})\backslash \{\sigma_{x,1}(y)\})\cup
\{\mu'_{\sigma_{x,1}(y)}(\sigma_{x,1}(y))\}=(\sigma_{x,1}(X')\backslash
\{x,\sigma_{x,1}(y)\})\cup
\{\mu'_{\sigma_{x,1}(y)}(\sigma_{x,1}(y))\}.$$ Thus by (i), it
follows that $X''=X'''$.

 Owing to the first formula of (\ref{twosubseeds}),
$\widetilde{B''}$ is the matrix obtained by applying the mutation in
the direction $y$ on $\widetilde{B'}$ and then deleting the row and
column labeled $x$; on the other hand, due to the second formula of
(\ref{twosubseeds}), since $y\neq x$, $\widetilde{B'''}$ is
identified with the matrix by first deleting the row and column
labeled $x$ in $\widetilde{B'}$ and then applying the mutation in
the direction $y$.

By calculation, since $\widetilde{X'}\ni x\neq y\in X'$, for
$s\in X'\setminus \{x\}$ and $t\in \widetilde{X'}\setminus\{x\}$, we have
\begin{eqnarray*} b''_{st}=b'''_{st}=
\left\{\begin{array}{lll} b'_{st}+\frac{\mid b'_{sy}\mid
b'_{yt}+b'_{sy}\mid b'_{yt}\mid}{2},&
\text{if} ~~s\neq y \neq t;\\
-b'_{st},& \text{if}~~ s=y \; \text{or}~~ t=y.\end{array}\right.
\end{eqnarray*} Hence, $\widetilde{B''}=\widetilde{B'''}$, which completes the
proof.

(iii)   Let $w_1=z_1$. As $z_1\neq x$, we have $\sigma_{x,1}(w_1)=z_1$. By (ii),
$\sigma_{x,1}(\mu_{w_1}(\Sigma) \backslash \{x\}
)=\mu'_{z_1}(\Sigma\setminus\{x\})$. Assume that for $t-1$,
$(z_1,\cdots,z_{t-1})$ can be lifted to a
$(\sigma_{x,1},\Sigma,\Sigma\setminus\{x\})$-biadmissible sequence
$(w_1,\cdots,w_{t-1})$ and
$\sigma_{x,1}(\mu_{w_{t-1}}\cdots\mu_{w_1}(\Sigma) \backslash
\{x\})=\mu'_{z_{t-1}}\cdots\mu'_{z_1}(\Sigma\setminus\{x\})$. Then we consider the case for
$t$. Since $z_t$ is exchangeable in
$\mu'_{z_{t-1}}\cdots\mu'_{z_1}(\Sigma\setminus\{x\})$ and
$\sigma_{x,1}(\mu_{w_{t-1}}\cdots\mu_{w_1}(\Sigma) \backslash \{x\}
)=\mu'_{z_{t-1}}\cdots\mu'_{z_1}(\Sigma\setminus\{x\})$,
there exists $w_t$ that is exchangeable in
$\mu_{w_{t-1}}\cdots\mu_{w_1}(\Sigma)$ such that
$\sigma_{x,1}(w_t)=z_t$. Using (ii) and
$\Sigma'=\mu_{w_{t-1}}\cdots\mu_{w_1}(\Sigma)$, by the induction assumption, we have
$$\sigma_{x,1}(\mu_{w_{t}}\mu_{w_{t-1}}\cdots\mu_{w_1}(\Sigma) \backslash
\{x\})=\mu'_{z_t}(\sigma_{x,1}(\mu_{w_{t-1}}\cdots\mu_{w_1}(\Sigma)\setminus\{x\}))=\mu'_{z_t}\mu'_{z_{t-1}}\cdots\mu'_{z_1}(\Sigma\setminus\{x\}).$$
\end{proof}

\begin{Lem}\label{acyclic}
If $\mathcal A(\Sigma)$ is an acyclic rooted cluster algebra, then
$\mathcal A(\Sigma\setminus \{x\})$ is acyclic for any $x\in \widetilde{X}$.
\end{Lem}

\begin{proof}
Since $\mathcal A(\Sigma)$ is acyclic, there exists a sequence $(x_{i_1},x_{i_2},\cdots,x_{i_s})$ with $1\leq
i_l\leq n$ for $1\leq l\leq s$ such that
$\mu_{x_{i_s}}\cdots\mu_{x_{i_1}}(\Sigma)$ is acyclic. By calculation, it is obvious that
$$\mu_{x_{i_s}}\cdots\mu_{x_{i_1}}(\Sigma)\setminus\{\mu_{x_{i_s}}\cdots\mu_{x_{i_1}}(x)\}=\mu_{x_{j_t}}\cdots\mu_{x_{j_1}}(\Sigma\setminus\{x\}),$$
where the sequence $(x_{j_1},\cdots,x_{j_t})$ is obtained by deleting $x$ from
the sequence $(x_{i_1},\cdots,x_{i_s})$. If we denote $\Sigma=(X, \widetilde{B})$, then $$\mu_{x_{i_s}}\cdots\mu_{x_{i_1}}(\Sigma)\setminus\{\mu_{x_{i_s}}\cdots\mu_{x_{i_1}}(x)\}=(\mu_{x_{i_s}}\cdots\mu_{x_{i_1}}(X)\setminus\{\mu_{x_{i_s}}\cdots\mu_{x_{i_1}}(x)\}, \mu_{x_{i_s}}\cdots\mu_{x_{i_1}}(\widetilde{B})\setminus\{\mu_{x_{i_s}}\cdots\mu_{x_{i_1}}(x)\}),$$
$$\mu_{x_{j_t}}\cdots\mu_{x_{j_1}}(\Sigma\setminus\{x\})=(\mu_{x_{j_t}}\cdots\mu_{x_{j_1}}(X\setminus\{x\}),\mu_{x_{j_t}}\cdots\mu_{x_{j_1}}(\widetilde{B}\setminus\{x\})).$$
It follows that
\begin{equation}\label{acyclicmatrix}
\mu_{x_{i_s}}\cdots\mu_{x_{i_1}}(\widetilde{B})\setminus\{\mu_{x_{i_s}}\cdots\mu_{x_{i_1}}(x)\}=\mu_{x_{j_t}}\cdots\mu_{x_{j_1}}(\widetilde{B}\setminus\{x\}).
\end{equation}

Since $\mu_{x_{i_s}}\cdots\mu_{x_{i_1}}(\widetilde{B})$ is acyclic, $\mu_{x_{i_s}}\cdots\mu_{x_{i_1}}(\widetilde{B})\setminus\{\mu_{x_{i_s}}\cdots\mu_{x_{i_1}}(x)\}$ is acyclic. Then by (\ref{acyclicmatrix}),
$\mu_{x_{j_t}}\cdots\mu_{x_{j_1}}(\widetilde{B}\setminus\{x\})$ is acyclic, which means $\mu_{x_{j_t}}\cdots\mu_{x_{j_1}}(\Sigma\setminus\{x\})$
 is acyclic.
\end{proof}

For a cluster algebra $\mathcal{A}=\mathcal{A}(\Sigma)$, recall that the {\bf
upper cluster algebra} of $\mathcal{A}$ is defined in \cite{[1]} as
$$\mathcal{U}:=\bigcap\limits_{\text{clusters} \; X\;\text{of}\;
\mathcal{A}}\mathbb{Z}\mathcal P[X^{\pm}],$$ where
$\mathbb{Z}\mathcal P$  is the coefficient ring of $\mathcal{A}$. By this definition, $\mathcal U$ is determined by all seeds of $\mathcal A$.
If a seed $\Sigma$ of $\mathcal{A}$ is given as initial, we denote the upper cluster algebra as $\mathcal{U}=\mathcal{U}(\Sigma)$. So, for any other seed $\Sigma'$, we have $\mathcal{U}(\Sigma)=\mathcal{U}(\Sigma')$.

Given any (initial) seed $\Sigma=(X, \widetilde{B})$ of $\mathcal A$, the {\bf upper
bound} of $\mathcal{A}(\Sigma)$ on $X$ is defined in \cite{[1]} as
$$\mathcal{U}_{X}(\Sigma):=\mathbb{Z}\mathcal P[X^{\pm 1}]\bigcap(\bigcap\limits_{x\in X}\mathbb{Z}\mathcal P[(\mu_{x}X)^{\pm 1}]).$$

By Laurent phenomenon and the definitions of upper cluster algebra and upper bound, clearly
$$\mathcal{A}\subseteq \mathcal{U}=\bigcap\limits_{\text{cluster}~ X~ \text{of}~\mathcal{A}} \mathcal{U}_{X}(\Sigma).$$

\begin{Lem}\label{include}
Let $\Sigma=(X,\widetilde{B})$ be an initial seed, where $x\in X$. Then
$\sigma_{x,1}(\mathcal{A}(\Sigma))\subseteq
\mathcal{U}(\Sigma\setminus\{x\})$ is the upper cluster algebra
of $\mathcal{A}(\Sigma\setminus\{x\})$ .
\end{Lem}
\begin{proof}
By (\cite{ADS}, Proposition 6.13),
$\sigma_{x,1}(\mathcal{U}_{X}(\Sigma))\subseteq
\mathcal{U}_{X\setminus\{x\}}(\Sigma\setminus\{x\})$. Since
$\mathcal{A}\subseteq \mathcal{U}_{X}(\Sigma)$ and
$\mathcal{U}_{X\setminus\{x\}}(\Sigma\setminus\{x\})\subseteq
\mathbb{Z}\mathcal P[X^{\pm 1}\setminus \{x^{\pm 1}\}]$, we have
\begin{equation}\label{huangeq}
\sigma_{x,1}(\mathcal{A}(\Sigma))\subseteq \mathbb{Z}\mathcal
P[X^{\pm 1}\setminus\{x^{\pm 1}\}].
\end{equation}

For any cluster $X'$ of $\mathcal{A}(\Sigma\setminus\{x\})$, there
exists a $\Sigma\setminus\{x\}$-admissible sequence
$(z_1,\cdots,z_s)$ such that
$\mu'_{z_s}\cdots\mu'_{z_1}(X\setminus\{x\})=X'$. By Lemma
\ref{mutation2}(iii), $(z_1,\cdots,z_s)$ can be lifted to a
$(\sigma_{x,1},\Sigma,\Sigma\setminus\{x\})$-biadmissible sequence
$(y_1,\cdots,y_s)$ and
$\sigma_{x,1}(\mu_{y_s}\cdots\mu_{y_1}(\Sigma)\setminus\{x\})=\mu_{z_s}\cdots\mu_{z_1}(\Sigma\setminus\{x\})$.
Hence,
$\sigma_{x,1}(\mu_{y_s}\cdots\mu_{y_1}(X)\setminus\{x\})=\sigma_{x,1}(\mu_{y_s}\cdots\mu_{y_1}(X))\setminus\{1\}=X'$,
and then,
\begin{equation}\label{Laurent1}
\sigma_{x,1}(\mathbb{Z}\mathcal
P[(\mu_{y_s}\cdots\mu_{y_1}(X))^{\pm1}])\subseteq \mathbb{Z}\mathcal
P[X'^{\pm 1}].
\end{equation}
By Laurent phenomenon,
 \begin{equation}\label{Laurent2}
 \mathcal{A}(\Sigma)=\mathcal{A}(\mu_{y_s}\cdots\mu_{y_1}(\Sigma))\subseteq
\mathbb{Z}\mathcal P [(\mu_{y_s}\cdots\mu_{y_1}(X))^{\pm1}].
\end{equation}
From (\ref{Laurent1}) and (\ref{Laurent2}), we have
$\sigma_{x,1}(\mathcal{A}(\Sigma))\subseteq
\mathbb{Z}\mathcal P[X'^{\pm 1}].$
Hence by the definition of upper
cluster algebra and the arbitrariness of $X'$ as cluster in
$\mathcal{A}(\Sigma\setminus\{x\})$, it follows that
$\sigma_{x,1}(\mathcal{A}(\Sigma))\subseteq
\mathcal{U}(\Sigma\setminus\{x\})$.
\end{proof}

\begin{Lem}\label{acyclicsub1}
 Given an initial seed $\Sigma=(X,\widetilde{B})$, in case either $x\in X_{fr}$ or $\mathcal A(\Sigma\setminus \{x\})$ is acyclic, then (i)~ $\sigma_{x,1}(\mathcal{A}(\Sigma))=\mathcal{A}(\Sigma\setminus\{x\})$ and (ii)~ $\sigma_{x,1}$ is surjective. \end{Lem}

\begin{proof} (ii) is trivial from (i). So, we only need to prove (i).

For any cluster variable
$z\in\mathcal{A}(\Sigma\setminus\{x\})$,
$z=\mu_{z_t}\cdots\mu_{z_1}(z_0)$ for some
$\Sigma\setminus\{x\}$-admissible sequence $(z_1,\cdots,z_t)$ and
$z_0\in \widetilde{X}\setminus\{x\}$. According to Lemma
\ref{mutation2} (iii), $(z_1,\cdots,z_t)$ can lift to a
$(\sigma_{x,1},\Sigma,\Sigma\setminus\{x\})$-biadmissible sequence
$(y_1,\cdots, y_t)$. Thus,
$\sigma_{x,1}(\mu_{y_t}\cdots\mu_{y_1}(y_0))=\mu_{z_t}\cdots\mu_{z_1}(z_0)=z$,
where $y_0=z_0\in\widetilde{X}\setminus\{x\}$. Therefore,
$\sigma_{x,1}(\mathcal{A}(\Sigma))\supseteq
\mathcal{A}(\Sigma\setminus \{x\})$ since $\mathcal{A}(\Sigma\setminus
\{x\})$ is generated by all cluster variables.

In order to prove $\sigma_{x,1}(\mathcal{A}(\Sigma))\subseteq
\mathcal{A}(\Sigma\setminus \{x\})$, it is enough to claim
$\sigma_{x,1}(\mu_{y_{n}}\cdots\mu_{y_{2}}\mu_{y_{1}}(y))\in
\mathcal{A}(\Sigma\setminus \{x\})$ for any composition of mutations
$\mu_{y_{n}}\cdots\mu_{y_{2}}\mu_{y_{1}}$ and $y\in \widetilde{X}$
since $\mathcal{A}(\Sigma)$ is generated by all cluster variables.

Suppose $x\in X_{fr}$. Since $x$ is frozen, $y_{i}\neq x$ for any $i=1,\cdots,n$.

When $n=1$, it is clear by Lemma
\ref{mutation2} (i). Using induction on $n$ and by Lemma \ref{mutation2} (ii), in the case for $n\geq 2$, we have
$$\sigma_{x,1}(\mu_{y_{n}}\cdots\mu_{y_{2}}\mu_{y_{1}}(y))=\left\{\begin{array}{lll} \mu'_{y_{n}}\cdots\mu'_{y_{2}}\mu'_{y_{1}}(y),&
\text{if}\; y \neq x;\\
1,& \text{if}~ y=x.\end{array}\right.$$
which is certainly in $\mathcal{A}(\Sigma\setminus \{x\})$. Thus, the result
follows.

Suppose $\mathcal A(\Sigma\setminus \{x\})$ is acyclic. Then by (\cite{M. G}, Theorem 2),
$\mathcal{A}(\Sigma\setminus\{x\})=\mathcal{U}(\Sigma\setminus\{x\})$.
Thus, the result follows from Lemma \ref{include}.
\end{proof}

This lemma improves the results in
\cite{ADS} in some cases mentioned before Proposition \ref{acyclicsub} as the partial answer of the
problem in \cite{ADS} that whether $\sigma_{x,1}$ is surjective.

\begin{Rem}
(1) The case $x\in X_{fr}$ has been mentioned in (\cite{CZ},
2.39).

(2) Following this lemma and Lemma \ref{acyclic}, $\sigma_{x,1}$ is
surjective if $\mathcal{A}(\Sigma\setminus\{x\})$ is acyclic.
\end{Rem}

Since $I_1=I'_1\cup I''_1$ for $I'_1\subseteq X$ and
$I''_1\subseteq X_{fr}$, using Lemma \ref{acyclicsub1} step-by-step
at $x \in I_1$ and according to (\cite{ADS},Proposition
6.9), $\sigma_{x,1}$ is an ideal surjective rooted cluster morphism.
Since $\sigma_{x,1}$ is an algebra homomorphism, we have the following.

\begin{Thm}\label{cluster quo.}
For a seed $\Sigma$, if the rooted cluster algebra $\mathcal A(\Sigma)$ is acyclic, then for any $I_1\subseteq \widetilde{X}$, the pure sub-cluster algebra
$\mathcal{A}(\Sigma_{\emptyset, I_1})$ is a rooted cluster quotient algebra of $\mathcal{A}(\Sigma)$ via $\sigma_{I_1,1}=\prod\limits_{x\in I_1}\sigma_{x,1}:\mathcal{A}(\Sigma)\rightarrow \mathcal{A}(\Sigma_{\emptyset,I_1})$, i.e. the specialisation of $\mathcal A(\Sigma)$ at $I_1$.
\end{Thm}

Trivially, Proposition \ref{acyclicsub} is a special case of this
theorem.

By Theorem \ref{cluster quo.}, we know that a part of the set of rooted
cluster quotient algebras consists of
some pure sub-cluster algebras of rooted cluster algebras.

 Then the following remaining question is still open: \\
{\bf Question}:\; In the case that the
cluster algebra $\mathcal A(\Sigma\setminus\{x\})$ is cyclic, is $\sigma_{x,1}$ surjective or, equivalently,
Im$\sigma_{x,1}\subseteq \mathcal A(\Sigma\setminus\{x\})$?

As a corollary of Theorem \ref{cluster quo.}, we show the finite type and finite mutation type of mixing-type sub-rooted cluster algebras of a rooted cluster algebra as follows.

A (rooted) cluster algebra $\mathcal{A}=\mathcal{A}(\Sigma)$ is said to be of
{\bf finite type} if it has finite
cluster variables, and to be of {\bf finite mutation type} if it has finite number of exchange matrices up to similar permutation of
rows and columns, under all mutations (see \cite{fosth}\cite{fz1}\cite{fz2}). It is easy to see that a cluster algebra of finite type is always of finite mutation type.

\begin{Cor}\label{cor6.8}
\rm{ All mixing-type sub-rooted cluster algebras of a rooted cluster algebra of finite type (respectively, finite mutation type) are of finite type (respectively, finite mutation type).}
\end{Cor}

\begin{proof}
 Let $\mathcal A(\Sigma)$ be a rooted cluster algebra of finite type.
According to (\cite{fz2}, Theorem 1.4), since it is of finite
type, the corresponding Cartan matrix of $\mathcal A(\Sigma)$ is of Dynkin type, which means that $\mathcal A(\Sigma)$ is acyclic. Then, $\mathcal A(\Sigma_{I_0,I_1})$ is also acyclic. Note that $\Sigma_{I_0,I_1}=(\Sigma_{I_0, \emptyset})_{\emptyset,I_1}$.
 By Theorem  \ref{cluster quo.}, there is a surjective morphism from $\mathcal{A}(\Sigma_{I_0, \emptyset})$ to
$\mathcal{A}(\Sigma_{I_0, I_1})$ in {\bf Clus}. Hence, in order to show the finite type of
$\mathcal{A}(\Sigma_{I_0, I_1})$, it is enough to prove by (\cite{ADS}, Corollary 6.3)
 that $\mathcal{A}(\Sigma_{I_0, \emptyset})$ is of
finite type. By Theorem \ref{rooted cluster
subalgebra}, $\mathcal{A}(\Sigma_{I_0,\emptyset})$ is a rooted
cluster subalgebra of $\mathcal{A}(\Sigma)$. So, by (\cite{ADS},
Corollary 4.6(2)),  all cluster variables of
$\mathcal{A}(\Sigma_{I_0,\emptyset})$ are also in $\mathcal{A}(\Sigma)$.  Then, $\mathcal{A}(\Sigma_{I_0,\emptyset})$
is of finite type since $\mathcal{A}(\Sigma)$  is so.

It is easy to see that
$(\mu_{x}(\Sigma_{I_0,\emptyset}))_{\emptyset,I_1}=\mu_{x}(\Sigma_{I_0,I_1})$
for any $x\in X\setminus (I_0\cup I_1)$. Therefore, to show that
$\mathcal{A}(\Sigma_{I_0,I_1})$ is of finite mutation type, it suffices to
prove that $\mathcal{A}(\Sigma_{I_0,\emptyset})$ is of finite mutation type.
  By (\ref{embedding}),   all mutations of $\Sigma_{I_{0},\emptyset}$ are sub-seeds of some mutations of $\Sigma$. Hence,  $\mathcal{A}(\Sigma_{I_0,\emptyset})$
is of finite mutation type since $\mathcal{A}(\Sigma)$ is of finite mutation type.
\end{proof}

Furthermore, in the next sub-section, we give the
characterization of rooted cluster quotient algebras of rooted
cluster algebras, which are not sub-cluster algebras, through the method
of gluing frozen variables.

\subsection{Rooted cluster quotient algebras via gluing method}.

In this section, we give another class of rooted cluster quotient algebras
via gluing frozen variables.

For a surjective rooted cluster morphism $g:\mathcal{A}(\Sigma)\rightarrow \mathcal{A}(\Sigma')$ and
$I_1=\{x\in \widetilde{X}|g(x)\in\mathbb{Z}\}$, by Definition \ref{reducedseed}, we can define $\Sigma^{(g)}=(X^{(g)},\widetilde{B^{(g)}})$, the contraction of $\Sigma$ under $g$.
We will show that there is a surjective rooted cluster
morphism $f:\mathcal{A}(\Sigma^{(g)})\rightarrow
\mathcal{A}(\Sigma')$
 satisfying  $f(X^{(g)})= X'$ and
$f(\widetilde{X^{(g)}})=\widetilde{X'}$.

First, we have a unique algebra
homomorphism
\begin{equation}\label{morphismf}
f:\mathbb{Q}[X^{(g)}_{fr}][X^{(g)\pm 1}]\rightarrow
\mathbb{Q}[X'_{fr}][X'^{\pm 1}]
\end{equation}
 such that $f(x)=g(x)$ for $x\in
\widetilde{X^{(g)}}$ and $f(x^{-1})=g(x)^{-1}$ for $x\in X^{(g)}$.
By Laurent phenomenon, $\mathcal{A}(\Sigma^{(g)})\subseteq
\mathbb{Q}[X^{(g)}_{fr}][X^{(g)\pm 1}]$.

\begin{Prop}\label{inducedhom}  Let $g:\mathcal{A}(\Sigma)\rightarrow \mathcal{A}(\Sigma')$
be a surjective rooted cluster morphism. Define $f$ as above in (\ref{morphismf}) and restrict $f$ to $\mathcal{A}(\Sigma^{(g)})$. Then $f:\mathcal{A}(\Sigma^{(g)})\rightarrow
\mathcal{A}(\Sigma')$ is a surjective
rooted cluster morphism satisfying  $f(X^{(g)})= X'$ and
$f(\widetilde{X^{(g)}})=\widetilde{X'}$.

This surjective rooted cluster morphism $f$ is called the {\bf
contraction} of $g$.
\end{Prop}
\begin{proof}
First, we show that $f$ is a rooted cluster morphism. Obviously, $f$ satisfies CM1 and CM2. Now we prove that CM3 holds for $f$. For this, we need the following two claims.

\textbf{Claim 1:}\; If $(y_1,\cdots,y_t)$ is
$\Sigma^{(g)}$-admissible, then it is $(f, \Sigma^{(g)},
\Sigma')$-biadmissible.

 For the case $t=1$, since $y_1\in
X^{(g)}=X\setminus I_1$, we have $f(y_1)=g(y_1)\in X'$. Hence, the sequence $(y_1)$ of length 1 is
$(f,\Sigma^{(g)},\Sigma')$-biadmissible.

Assume the claim holds for the case $t-1$. Now using induction, we prove the claim for case $t$.

For any $y\in
\widetilde{X^{(g)}}\subseteq \widetilde{X}$, we have
$\mu_{y_1}^{\Sigma}(y)=\left\{\begin{array}{lll} y,& \text{if}
~~y\neq y_1;\\
\frac{\prod_1 A_1+\prod_2 A_2}{y_1},&
\text{if}~~y=y_1,\end{array}\right.$ where $A_i$ are
monomials over $I_1$ and $\prod_i$ are monomials over
$\widetilde{X}\setminus I_1$. It follows that
\begin{equation}\label{formp19-1}
g(\mu_{y_1}^{\Sigma}(y))=\left\{\begin{array}{lll} g(y),& \text{if}
~~y\neq y_1;\\
\frac{g(\prod_1)g(A_1)+g(\prod_2)g(A_2)}{g(y_1)},&
\text{if}~~y=y_1.\end{array}\right.\end{equation}

 By the definition of $\Sigma^{(g)}$,
we have, for any $y\in \widetilde{X^{(g)}}\subseteq \widetilde{X}$,
\begin{eqnarray*}
\mu_{y_1}^{\Sigma^{(g)}}(y)=\left\{\begin{array}{lll} y,& \text{if}
~~y\neq y_1;\\
\frac{\prod_1+\prod_2}{y_1},& \text{if}~~y=y_1, g(A_1)\neq 0
\neq g(A_2);\\
\frac{2}{y_1},& \text{if}~~y=y_1,
g(A_1)g(A_2)=0.\end{array}\right.\end{eqnarray*}
Therefore,
\begin{equation}\label{formp19-2}
f(\mu_{y_1}^{\Sigma^{(g)}}(y))=g(\mu_{y_1}^{\Sigma^{(g)}}(y))=\left\{\begin{array}{lll}
g(y),& \text{if}
~~y\neq y_1;\\
\frac{g(\prod_1)+g(\prod_2)}{g(y_1)},& \text{if}~~y=y_1, g(A_1)g(A_2)\neq 0;\\
\frac{2}{g(y_1)},& \text{if}~~y=y_1,
g(A_1)g(A_2)=0.\end{array}\right.\end{equation}

On the other
hand, $\mu_{g(y_1)}(g(y_1))=\frac{B_1+B_2}{g(y_1)}$, where $B_i$ are
monomials over $\widetilde{X'}$. By CM3, we have
$$\mu_{g(y_1)}(g(y_1))=g(\mu_{y_1}^{\Sigma}(y_1))=\frac{g(\prod_1)
g(A_1)+g(\prod_2) g(A_2)}{g(y_1)}$$
 so,  $\frac{g(\prod_1)
g(A_1)+g(\prod_2) g(A_2)}{g(y_1)}=\frac{B_1+B_2}{g(y_1)}$. Thus,
$g(\prod_1) g(A_1)+g(\prod_2) g(A_2)=B_1+B_2$. Comparing the two sides of this formula, since $g(\prod_i)$
are monomials over $\widetilde{X'}$ and $g(A_i)\in \mathbb{Z}$, by
the algebraic independence of $\widetilde{X'}$, we discuss the
following cases:

{\bf Case 1.}\; $g(A_1)g(A_2)=0$. Without loss of generality, assume
$g(A_1)=0$. Then $g(\prod_2)=1$ and $g(A_2)=B_1+B_2=2$.

{\bf Case 2.}\; $g(A_1)g(A_2)\neq 0$. Then, we have:

 (1) if $g(\prod_1)g(\prod_2)\neq 1$, then $g(A_1)=g(A_2)=1$ and

(2) if $g(\prod_1)=g(\prod_2)=1$, then $g(A_2)+g(A_2)=B_1+B_2=2$.\\
 Both (1)
and (2) mean $g(\prod_1) g(A_1)+g(\prod_2)
g(A_2)=g(\prod_1)+g(\prod_2)$.

Following this discussion, comparing (\ref{formp19-1}) and (\ref{formp19-2}), we get that
\begin{equation}\label{basic1}
f(\mu_{y_1}^{\Sigma^{(g)}}(y))=g(\mu_{y_1}^{\Sigma}(y))=\mu_{g(y_1)}(g(y))=\mu_{f(y_1)}(f(y)).
\end{equation}

Now we need to prove the important relation on the new seed $\Sigma^{(g)}$, which is as follows.
\begin{Lem}\label{basicfact}
There exists a positive isomorphism $(\mu_{y}^{\Sigma}(\Sigma))^{(g)}\cong
\mu_{y}^{\Sigma^{(g)}}(\Sigma^{(g)})$,
 where $y\in X\setminus I_1$.
\end{Lem}
\begin{proof}
Denote $(\mu_{y}^{\Sigma}(\Sigma))$ by $(Y,\widetilde{C})$,
$(\mu_{y}^{\Sigma}(\Sigma))^{(g)}$ by $(Y',\widetilde{C}')$ and
$\mu_{y}^{\Sigma^{(g)}}(\Sigma^{(g)})$ by $(Y'',\widetilde{C}'')$.

By definition, we have $$Y'=(X\setminus (I_1\cup
\{y\}))\cup\{\mu_{y}^{\Sigma}(y)\},\;\;\;
\widetilde{Y}'=(\widetilde{X}\setminus (I_1\cup
\{y\}))\cup\{\mu_{y}^{\Sigma}(y)\},$$ $$Y''=(X\setminus (I_1\cup
\{y\}))\cup\{\mu_{y}^{\Sigma^{(g)}}(y)\},\;\;\;
\widetilde{Y}''=(\widetilde{X}\setminus (I_1))\cup
\{y\})\cup\{\mu_{y}^{\Sigma^{(g)}}(y)\}.$$ Comparing the sets
$\widetilde{Y}'$ and $\widetilde{Y}''$, their elements can be
in one-to-one correspondence with the identity map but the
correspondence between $\mu_{y}^{\Sigma}(y)$ and
$\mu_{y}^{\Sigma^{(g)}}(y)$.

According to definition, we have
\begin{eqnarray*}
c_{xz}= \left\{\begin{array}{lll} b_{xz}+\frac{|b_{xy}|b_{yz}+b_{xy}|b_{yz}|}{2}, & \;\text{if}\; x\neq \mu_y^{\Sigma}(y)\neq z; \\
    -b_{xz},& \;\text{if}\; x=\mu_y^{\Sigma}(y) \;\text{or}\; z=\mu_y^{\Sigma}(y).\\
\end{array}\right.
\end{eqnarray*}
\begin{eqnarray*}
c'_{xz}= \left\{\begin{array}{lll}c_{xz}, & \;\text{if}\; g(w)\neq 0 \;\forall w\in I_1\;\text{adjacent to}\; x \;\text{or}\; z;\\
    0,& \text{otherwise}.\\
\end{array}\right.
\end{eqnarray*}
\begin{eqnarray*}
c''_{xz}= \left\{\begin{array}{lll} b^{(g)}_{xz}+\frac{|b^{(g)}_{xy}|b^{(g)}_{yz}+b^{(g)}_{xy}|b^{(g)}_{yz}|}{2}, & \;\text{if}\; x\neq \mu_y^{\Sigma^{(g)}}(y)\neq z; \\
    -b^{(g)}_{yz} \;\text{or}\; -b^{(g)}_{xy},& \;\text{if}\; x=\mu_y^{\Sigma^{(g)}}(y) \;\text{or}\; z=\mu_y^{\Sigma^{(g)}}(y).\\
\end{array}\right.
\end{eqnarray*}

We will show $\widetilde{C}'=\widetilde{C}''$ in three cases, which are as follows:
\\
(i)~ In case $x\in X\setminus (I_1\cup\{y\})$ and $z\in \widetilde{X}\setminus (I_1\cup\{y\})$.

(1) If there exists no $w\in \widetilde{X}$ adjacent to $y$ with $g(w)=0$,
 then
\begin{eqnarray*}
c'_{xz}=c''_{xz}= \left\{\begin{array}{lll}b_{xz}+\frac{|b_{xy}|b_{yz}+b_{xy}|b_{yz}|}{2}, & \;\text{if}\; g(w)\neq 0 \; \forall w\in I_1 \;\text{adjacent to}\; x \;\text{or}\; z;\;\\
    0,& \text{otherwise}.\\
\end{array}\right.
\end{eqnarray*}

(2) If there exists a $w\in \widetilde{X}$ adjacent to $y$ with
$g(w)=0$, then from the formula (\ref{basic1}),
$b_{yw}b_{xy}<0$, $b_{yz}b_{yw}>0$ and $b_{xy}b_{yz}<0$, we
have
\begin{eqnarray*}
c'_{xz}=c''_{xz}= \left\{\begin{array}{lll}b_{xz}, & \;\text{if}\; g(w)\neq 0 \;\forall\in I_1 \;\text{ adjacent to}\; x \;\text{or}\; z;\;\\
    0,& \text{otherwise}.\\
\end{array}\right.
\end{eqnarray*}
 \\
(ii)~ In case $z\in \widetilde{X}\setminus (I_1\cup\{y\})$.

(1) If there exists no $w\in \widetilde{X}$ adjacent to $y$ with $g(w)=0$, then
\begin{eqnarray*}
c'_{(\mu_{y}^{\Sigma}(y))z}=c''_{(\mu_{y}^{\Sigma^{(g)}}(y))z}= \left\{\begin{array}{lll}-b_{yz}, & \;\text{if}\; g(w)\neq 0\; \forall w\in I_1 \text{ adjacent to}\; z;\;\\
    0,& \text{otherwise}.\\
\end{array}\right.
\end{eqnarray*}

(2) If there exists a $w\in \widetilde{X}$ adjacent to $y$ with $g(w)=0$, then
$c'_{(\mu_{y}^{\Sigma}(y))z}=c''_{(\mu_{y}^{\Sigma^{(g)}}(y))z}=0$.
\\
(iii)~ In case $x\in X\setminus (I_1\cup\{y\})$, we have similarly
$c'_{x(\mu_{y}^{\Sigma}(y))}=c''_{x(\mu_{y}^{\Sigma^{(g)}}(y))}$.

According to the above discussion on $c'_{xy}$ and
$c''_{xy}$, under the correspondence between the sets
$\widetilde{Y}'$ and $\widetilde{Y}''$, we get that
$\widetilde{C}'=\widetilde{C}''$.

Hence,
$(\mu_{y}^{\Sigma}(\Sigma))^{(g)}\cong
\mu_{y}^{\Sigma^{(g)}}(\Sigma^{(g)})$.
\end{proof}

Let us return to prove the proposition.

By Lemma \ref{basicfact}, we have
$(\mu_{y_1}^{\Sigma}(\Sigma))^{(g)}\cong
\mu_{y_1}^{\Sigma^{(g)}}(\Sigma^{(g)})$. Since $(y_1,\cdots,y_{t})$
is $\Sigma^{(g)}$-admissible, it means $(y_2,\cdots,y_{t})$ is
$(\mu_{y_1}^{\Sigma}(\Sigma))^{(g)}$-admissible.

By Proposition \ref{generate},
$g:\mathcal{A}(\mu_{y_1}^{\Sigma}(\Sigma))\rightarrow
\mathcal{A}(\mu_{g(y_1)}(\Sigma'))$ is a surjective rooted cluster
morphism.  For this $g$, using the induction assumption,
$(y_2,\cdots,y_{t})$ is
$(f,\mu_{y_1}^{\Sigma^{(g)}}(\Sigma^{(g)}),\mu_{f(y_1)}(\Sigma'))$-biadmissible.
Therefore,  $(y_1,\cdots,y_t)$ is
$(f,\Sigma^{(g)},\Sigma')$-biadmissible. Hence, the claim holds.

\textbf{Claim 2:}\; For a $(f, \Sigma^{(g)}, \Sigma')$-biadmissible
$(y_1,\cdots,y_t)$, there exists uniquely a
$(g,\Sigma,\Sigma')$-biadmissible sequence $(x_1,\cdots,x_t)$ such
that
\begin{equation}\label{tworelations}
 f(\mu_{y_t}^{\Sigma^{(g)}}\cdots\mu_{y_1}^{\Sigma^{(g)}}(y))=g(\mu_{x_t}^{\Sigma}\cdots\mu_{x_1}^{\Sigma}(y))\;\;\;\;\text{and}\;\;\;\; g(x_i)=f(y_i)
\end{equation}
for all $1\leq i\leq t$ and any $y\in \widetilde{X^{(g)}}\subseteq\widetilde{X}$.
Conversely, for a $(g,\Sigma,\Sigma')$-biadmissible sequence
$(x_1,\cdots,x_t)$, there exists uniquely an $(f, \Sigma^{(g)},
\Sigma')$-biadmissible $(y_1,\cdots,y_t)$ such that the relations in (\ref{tworelations}) are satisfied.

In the case $t=1$, let $x_1=y_1$. Then we have $g(x_1)=f(y_1)$, and
by the formula (\ref{basic1}),
$f(\mu_{y_1}^{\Sigma^{(g)}}(z))=g(\mu_{x_1}^{\Sigma}(z))=\mu_{f(y_1)}(f(z))$
for any $z\in \widetilde{X^{(g)}}$. So, Claim 2 holds for $t=1$.

Assume that Claim 2 holds in the case for $t-1$. In the case for
$t$, by Proposition \ref{generate},
$g:\mathcal{A}(\mu_{x_1}^{\Sigma}(\Sigma))\rightarrow\mathcal{A}(\mu_{g(x_1)}(\Sigma'))$
is a surjective rooted cluster morphism. By (\ref{basicfact}),
$(\mu_{x_1}^{\Sigma}(\Sigma))^{(g)}\cong\mu_{y_1}^{\Sigma^{(g)}}(\Sigma^{(g)})$.
Hence, by induction assumption, there exits a
$(g,\mu_{x_1}(\Sigma),\mu_{g(x_1)(\Sigma')})$-biadmissible sequence
$(x_2,\cdots,x_t)$ such that $g(x_i)=f(x_i)$ for $2\leq i\leq t$ and
$g(\mu_{x_t}^{\Sigma}\cdots\mu_{x_2}^{\Sigma}(y))=f(\mu_{y_t}^{\Sigma^{(g)}}\cdots\mu_{y_2}^{\Sigma^{(g)}}(y))$
for all cluster variables $y$ in
$\mu_{y_1}^{\Sigma^{(g)}}(\Sigma^{(g)})$. Therefore,
$(x_1,\cdots,x_t)$ is the $(g,\Sigma,\Sigma')$-biadmissible sequence
such that $g(x_i)=f(y_i)$ for all $1\leq i \leq k$ and
$f(\mu_{y_t}^{\Sigma^{(g)}}\cdots\mu_{y_1}^{\Sigma^{(g)}}(z))=g(\mu_{x_t}^{\Sigma}\cdots\mu_{x_1}^{\Sigma}(z))$
for all $z\in \widetilde{X^{(g)}}$.

Conversely, we can prove similarly for a $(g,\Sigma,\Sigma')$-biadmissible sequence
$(x_1,\cdots,x_t)$ that there exists uniquely a $(f, \Sigma^{(g)},
\Sigma')$-biadmissible sequence $(y_1,\cdots,y_t)$ such that the relations in (\ref{tworelations}) are satisfied.

Hence, the claim follows.

Now, we prove the CM3 condition for $f$. For any
$(f,\Sigma^{(g)},\Sigma')$-biadmissible sequence $(y_1,\cdots,y_t)$,
by Claim 2, there exists uniquely $(g,\Sigma,\Sigma')$-biadmissible
sequence $(x_1,\cdots,x_t)$ such that $g(x_i)=f(y_i)$ for all $1\leq
i\leq t$ and
$f(\mu_{y_t}^{\Sigma^{(g)}}\cdots\mu_{y_1}^{\Sigma^{(g)}}(z))=g(\mu_{x_t}^{\Sigma}\cdots\mu_{x_1}^{\Sigma}(z))$
for all $z\in \widetilde{X^{(g)}}$. By CM3 for $g$,
$g(\mu_{x_t}^{\Sigma}\cdots\mu_{x_1}^{\Sigma}(z))=\mu_{g(x_t)}\cdots\mu_{g(x_1)}(g(z))$.
Then,
$f(\mu_{y_t}^{\Sigma^{(g)}}\cdots\mu_{y_1}^{\Sigma^{(g)}}(z))=\mu_{f(y_t)}\cdots\mu_{f(y_1)}(f(z))$
for all $z\in \widetilde{X^{g}}$.  Hence, CM3 for $f$ follows.

To show $f:
\mathcal{A}(\Sigma^{(g)})\rightarrow\mathcal{A}(\Sigma')$ is a ring
homomorphism,  it suffices to prove that
$f(\mathcal{A}(\Sigma^{(g)}))\subseteq \mathcal{A}(\Sigma')$ since
$g$ is a ring homomorphism. In fact, for any cluster variable $y\in
\mathcal{A}(\Sigma^{(g)})$, there exists a $\Sigma^{(g)}$-admissible
sequence $(y_1,\cdots,y_t)$ and $y_0\in \widetilde{X^{(g)}}$ such
that
$y=\mu_{y_t}^{\Sigma^{(g)}}\cdots\mu_{y_1}^{\Sigma^{(g)}}(y_0)$. By
Claim 1, any $\Sigma^{(g)}$-admissible sequence is
$(f,\Sigma^{(g)},\Sigma')$-biadmissible. Then by CM3,
$f(y)=\mu_{f(y_t)}\cdots\mu_{f(y_1)}(f(y_0))$, and by CM1 and the
definition of $f$, $f(y)$ is a cluster variable of
$\mathcal{A}(\Sigma')$. So, $f(\mathcal{A}(\Sigma^{(g)}))\subseteq
\mathcal{A}(\Sigma')$.

Now,  we verify that $f$ is surjective.  As $\mathcal{A}(\Sigma')$ is generated by all cluster
variables, it suffices to prove that any cluster variable $z\in \mathcal{A}(\Sigma')$ can be lifted to a cluster variable in
$\mathcal{A}(\Sigma^{(g)})$.

We have $z=\mu_{z_l}\cdots\mu_{z_1}(z_0)$ for some
$\Sigma'$-admissible sequence $(z_1,\cdots,z_l)$ and $z_0\in
\widetilde{X'}$. As $g$ is surjective, by (\cite{ADS}, Proposition
6.2), there exists a $(g,\Sigma,\Sigma')$-biadmissible sequence
$(x_1,\cdots,x_l)$ and $x_0\in \widetilde{X}$ such that
$g(\mu_{x_l}^{\Sigma}\cdots\mu_{x_1}^{\Sigma}(x_0))=\mu_{z_l}\cdots\mu_{z_1}(z_0)=z$.
 It is clear to see that $x_0\in \widetilde{X}\setminus I_1$.

According to Claim 2, there exists a
$(f,\Sigma,\Sigma')$-biadmissible sequence $(y_1,\cdots,y_l)$ such
that $$f(\mu_{y_l}^{\Sigma^{(g)}}\cdots\mu_{y_1}^{\Sigma^{(g)}}(x_0))=g(\mu_{x_l}^{\Sigma}\cdots\mu_{x_1}^{\Sigma}(x_0)).$$
So, it follows that
$f(\mu_{y_l}^{\Sigma^{(g)}}\cdots\mu_{y_1}^{\Sigma^{(g)}}(x_0))=z$. It means that $f$ is surjective.

In summary, we have shown that  $f:\mathcal{A}(\Sigma^{(g)})\rightarrow
\mathcal{A}(\Sigma')$ is a surjective
rooted cluster morphism.

Last, by the definition of $f$, we have $f(X^{(g)})= X'$ and
$f(\widetilde{X^{(g)}})=\widetilde{X'}$.
\end{proof}

 In Proposition \ref{inducedhom}, $f=g$ if $g$ is noncontractible, i.e. $I_1=\emptyset$ and then $\Sigma^{(g)}=\Sigma$. By definition and Proposition \ref{inducedhom}, it is easy to see the following.
\begin{Lem}  A
surjective rooted cluster morphism $g: \mathcal A(\Sigma)\rightarrow \mathcal A(\Sigma')$ is noncontractible if and only if $g(X)=X'$
and $g(\widetilde{X})=\widetilde{X'}$.
\end{Lem}

Owing to this and Proposition \ref{inducedhom}, in the sequel, in order to characterize the quotient from a surjective rooted
cluster morphism $f$, we always assume that $f$ is noncontractible, that is,  the conditions $f(X)=X'$
and $f(\widetilde{X})=\widetilde{X'}$ are satisfied.

 For two sets $U$ and $V$, $\varphi$ is called a {\bf partial map} from $U$ to $V$ on a subset $W$ of $U$ if $\varphi: W\rightarrow V$ is a map, denoted as $\varphi_W: U\rightarrow V$.

\begin{Def}\label{gluing}

Given a seed $\Sigma=(X, X_{fr},\widetilde{B})$, a subset $S\subseteq
\widetilde{X}$ and a partial injective map $\varphi_S:\;
\widetilde{X}\rightarrow \widetilde{X}$ satisfying $\varphi_S(S\cap
X_{fr})\subseteq X_{fr}$, we define a new seed
$\Sigma_{\varphi_S}=(\overline X,\widetilde{\overline B})$ (denoted
as $\Sigma_{\widehat{S\varphi_S(S)}}$) {\bf by gluing $S$ and
$\varphi_S(S)$ under $\varphi_S$} as follows:

(i) ~For any $s\in S$, define a new variable $\overline{s}$,  called the {\bf gluing variable} of $s$ and $\varphi_S(s)$  on
$\varphi_S$. For any subset  $T\subseteq S$, let
$\overline{T}=\{\overline{s}\;|\;s\in T\}$, which is called the {\bf
gluing set} of $T$ and $\varphi_S(T)$ on $\varphi_S$.

(ii) ~ Define $\widetilde{\overline X}=(\widetilde{X}\setminus (S\cup
\varphi_S(S)))\cup \overline{S}$ and  $\overline X=(X\setminus (S\cup
\varphi_S(S)))\cup \overline{S\cap X}$
 and its extended exchange matrix $\widetilde{\overline B}$ to be a $\#\overline X\times
\#\widetilde{\overline X}$ matrix satisfying:
\begin{eqnarray*}\overline b_{z_{1}z_{2}}= \left\{\begin{array}{lll} b_{z_{1}z_{2}}, & \text{if}\; z_1\in X\setminus(S\cup \varphi_S(S)), z_2\in \widetilde{X}\setminus(S\cup \varphi_S(S)) ;\\
b_{y_{1}y_{2}},& \text{if}\; z_1=\overline{y_{1}}\in \overline{S\cap X}, z_2=\overline{y_{2}}\in \overline{S};\\
b_{z_{1}y_{2}},& \text{if}\; z_1\in X\setminus(S\cup \varphi_S(S)), z_2=\overline{y_{2}}\in \overline{S}, y_{2}= \varphi_S(y_{2});\\
b_{z_{1}y_{2}}+b_{z_{1}\varphi_S(y_{2})},& \text{if}\; z_1\in X\setminus(S\cup \varphi_S(S)), z_2=\overline{y_{2}}\in \overline{S}, y_{2}\neq \varphi_S(y_{2});\\
b_{y_{1}z_{2}},& \text{if}\; z_1=\overline{y_{1}}\in \overline{S\cap X}, z_2\in \widetilde{X}\setminus(S\cup \varphi_S S).\\
\end{array}\right.
\end{eqnarray*}
\end{Def}

Following this definition, in the new seed
$\Sigma_{\varphi_S}=(\overline X, \widetilde{\overline B})$, we have
$\overline X_{fr}=(X_{fr}\setminus (S\cup \varphi_S(S)))\cup
\overline{S\cap X_{fr}}$.
 It is easy to see that $\widetilde{\overline B}$ is
skew-symmetrizable.

\begin{Ex}
Let $Q$ be the quiver:\; $\xymatrix{x_1\ar[r]^{}&*+[F]{x_2}&x_3\ar[l]^{}}$ and let $S=\{x_1\}$. Define $\varphi_S(x_1)=x_3$. Then $\Sigma(Q)_{\varphi_S}=\Sigma(Q')$, where $Q':\xymatrix{\overline{x}_1\ar[r]^{}&*+[F]{x_2}}$.
\end{Ex}

\begin{Rem}
~ For $y_{1},y_{2}\in X_{fr}$ in $\Sigma$,  define $S=\{y_1\}$ and
$\varphi_S: \{y_1\}\rightarrow \{y_2\}$ to get $\varphi_S(S)=\{y_2\}$.
Then, we get $\Sigma_{\widehat{y_{1}y_{2}}}=(\overline
X,\widetilde{\overline B})$, where $\overline X=X$,
$\widetilde{\overline X}=\widetilde{X}\setminus\{y_{1},y_{2}\}\cup
\{\overline y\}$ and its extended exchange matrix
$\widetilde{\overline B}$ is of $n\times (n+m-1)$ satisfying
$\overline b_{xy}=b_{xy}$ and $\overline b_{x\overline
y}=b_{xy_{1}}+b_{xy_{2}}$ for all $x\in X$ and $y\in
\widetilde{\overline X}\setminus \{\overline y\}$. In this
situation, we say $\overline y$ to be the {\bf gluing variable} of
$y_1$ and $y_2$.
\end{Rem}

Let $\Sigma=(X,\widetilde{B})$ be a seed with $y_{1}$ and $y_{2}$ as two frozen
variables. We can obtain a natural ring morphism $\pi: \;
\mathds Q[X_{fr}][X^{\pm 1}]\rightarrow\mathds F(\Sigma_{\widehat{y_{1}y_{2}}})$
satisfying $$\pi(x)=x \;\; \forall\; x\in\widetilde{X},  x\neq y_{1},
x\neq y_2,\;\; \text{and} \;\;\pi(y_{1})=\pi(y_{2})=\overline y.$$
Restricting $\pi$ on $\mathcal{A}(\Sigma)$,  such morphism $\pi_0=\pi|_{\mathcal{A}(\Sigma)}$ is called the {\bf canonical morphism induced by gluing $y_{1}, y_{2}\in X_{fr}$}.

To illustrate our result, we need the following lemma.

\begin{Lem}\label{glue}
Let $y_{1}, y_{2}\in X_{fr}$ in seed $\Sigma$ such that
$b_{xy_1}b_{xy_2}\geq 0$ for all $x\in X$. Then for any exchangeable
variable $z\in X$, there is a positive isomorphism
$\mu_{z}(\Sigma_{\widehat{y_{1}y_{2}}})\overset{h}{\cong}(\mu_{z}(\Sigma))_{\widehat{y_{1}y_{2}}}.$
\end{Lem}

\begin{proof}
 Denote
$\mu_{z}(\Sigma_{\widehat{y_{1}y_{2}}})=(Y,\widetilde{C})$ and
$(\mu_{z}(\Sigma))_{\widehat{y_{1}y_{2}}}=(\overline
Y,\widetilde{\overline C})$. Then by definition,
$$Y=\overline
X\setminus\{z\}\cup\{\mu_{z}^{\Sigma_{\widehat{y_1y_2}}}(z)\}=X\setminus\{z\}\cup\{\mu_{z}^{\Sigma_{\widehat{y_1y_2}}}(z)\},
\widetilde{Y}=\widetilde{ \overline
X}\setminus\{z\}\cup\{\mu_{z}^{\Sigma_{\widehat{y_1y_2}}}(z)\}=(\widetilde{X}\setminus\{z,y_1,y_2\})\cup\{\mu_{z}^{\Sigma_{\widehat{y_1y_2}}}(z),\overline
y\},$$ $$\overline
Y=(X\setminus\{z\})\cup\{\mu_{z}^{\Sigma}(z)\},\widetilde{\overline
Y}=(\widetilde{X}\setminus\{z,y_1,y_2\})\cup\{\mu_{z}^{\Sigma}(z),\overline
y\}.$$ Comparing the sets $\widetilde{Y}$ and $\widetilde{\overline
Y}$, their elements can be one-to-one correspondent by a bijection
$h$ with the identity map but the correspondence between
$\mu_{z}^{\Sigma_{\widehat{y_1y_2}}}(z)$ and  $\mu_{z}^{\Sigma}(z)$.

For all $x\in Y$ and $y\in \widetilde Y$, $\overline b_{xy}=b_{xy}$, and
$\overline b_{x\overline y}=b_{xy_{1}}+b_{xy_{2}}$ for all $x\in X$ and
$y\in \widetilde{\overline X}\setminus \{\overline y\}$. Furthermore, as
$b_{xy_1}b_{xy_2}\geq 0$ for all $x\in X$, so $|\overline
b_{x\overline y}|=|b_{xy_1}|+|b_{xy_2}|$. Hence, the entries
of $\widetilde{C}$ is given:
\[\begin{array}{ccl} c_{xy} & = & \left\{\begin{array}{lll} \overline {b}_{xy}+\frac{\overline
b_{xz}\cdot|\overline b_{zy}|+|\overline b_{xz}|\cdot \overline
b_{zy}}{2},&
\text{if} ~~x\neq \mu_{z}^{\Sigma_{\widehat{y_1y_2}}}(z)\neq y;\\
-\overline b_{xy},&\text{if}
~~x=\mu_{z}^{\Sigma_{\widehat{y_1y_2}}}(z)\;\text{or}\;y=\mu_{z}^{\Sigma}(z).\end{array}\right.\\
  & = & \left\{\begin{array}{lll}
b_{xy}+\frac{b_{xz}\cdot|b_{zy}|+|b_{xz}|\cdot b_{zy}}{2},&
\text{if} ~~y \neq \overline y, x\neq \mu_{z}^{\Sigma_{\widehat{y_1y_2}}}(z)\neq y;\\
-b_{xy},&\text{if} ~~y \neq \overline y, x=\mu_{z}^{\Sigma_{\widehat{y_1y_2}}}(z)\;\text{or}\;y=\mu_{z}^{\Sigma}(z);\\
\sum\limits_{i=1}^{2}(b_{xy_i}+\frac{b_{xz}\cdot|b_{zy_i}|+|b_{xz}|\cdot
b_{zy_i}}{2}),&
\text{if} ~~y =\overline y, x\neq \mu_{z}^{\Sigma_{\widehat{y_1y_2}}}(z)\neq y;\\
-b_{xy_1}-b_{xy_2},&\text{if} ~~y=\overline y,
x=\mu_{z}^{\Sigma_{\widehat{y_1y_2}}}(z)\;\text{or}\;y=\mu_{z}^{\Sigma}(z).\end{array}\right..
\end{array}\]

On the other hand, for all $x\in \overline Y$ and $y\in \widetilde
{\overline Y}$, we have the entries of $\widetilde{\overline C}$:
\begin{eqnarray*}\overline c_{xy}=
\left\{\begin{array}{lll}
b_{xy}+\frac{b_{xz}\cdot|b_{zy}|+|b_{xz}|\cdot b_{zy}}{2},&
\text{if} ~~y \neq \overline y, x\neq \mu_{z}^{\Sigma}(z)\neq y;\\
-b_{xy},&\text{if} ~~y \neq \overline y, x=\mu_{z}^{\Sigma}(z)\;\text{or}\;y=\mu_{z}^{\Sigma}(z);\\
\sum\limits_{i=1}^{2}(b_{xy_i}+\frac{b_{xz}\cdot|b_{zy_i}|+|b_{xz}|\cdot
b_{zy_i}}{2}),&
\text{if} ~~y =\overline y, x\neq \mu_{z}^{\Sigma}(z)\neq y;\\
-b_{xy_1}-b_{xy_2},&\text{if} ~~y=\overline y,
x=\mu_{z}^{\Sigma}(z)\;\text{or}\;y=\mu_{z}^{\Sigma}(z).\end{array}\right.
\end{eqnarray*}

According to the above expressions of $c_{xy}$ and $\overline
c_{xy}$, under the correspondence between the sets $\widetilde{Y}$
and $\widetilde{\overline Y}$, we get
$\widetilde{C}=\widetilde{\overline C}$. Hence,
$\mu_{z}(\Sigma_{\widehat{y_{1}y_{2}}})\cong(\mu_{z}(\Sigma))_{\widehat{y_{1}y_{2}}}$.
\end{proof}

With the above preparations,  another class of rooted
cluster quotient algebras is given as follows.

\begin{Prop}\label{gluefrozens}
Let $y_{1}, y_{2}\in X_{fr}$ in the seed $\Sigma$ with the gluing
variable $\overline{y}$. Then,
$\mathcal{A}(\Sigma_{\widehat{y_{1}y_{2}}})$ is a rooted cluster
quotient algebra of $\mathcal{A}(\Sigma)$ under the canonical morphism
$\pi_0=\pi\mid_{\mathcal{A}(\Sigma)}$ if and only if
$b'_{xy_{1}}b'_{xy_{2}}\geq 0$ for any exchangeable variable $x$ in
any seed $\Sigma'$ mutation equivalent to $\Sigma$ and its exchange
matrix $\widetilde{B'}$.
\end{Prop}

\begin{proof}
\textbf{``Only if":}\;
Let $\Sigma'=(X',\widetilde{B} ')$ be a seed mutation equivalent to
$\Sigma$. Hence, there exists a $\Sigma$-admissible sequence
$(z_1,\cdots,z_s)$ such that
$\Sigma'=\mu_{z_s}\cdots\mu_{z_1}(\Sigma)$. According to Proposition
\ref{generate}, $$\pi_0:\mathcal{A}(\Sigma')\rightarrow
\mathcal{A}(\mu_{\pi_0(z_s)}\cdots\mu_{\pi_0(z_1)}(\Sigma_{\widehat{y_1y_2}}))$$
is a rooted cluster morphism. Moreover, using Lemma \ref{glue}, it is easy to see that
$\mu_{\pi_0(z_s)}\cdots\mu_{\pi_0(z_1)}(\Sigma_{\widehat{y_1y_2}})\cong\Sigma'_{\widehat{y_1y_2}}$
by a seed isomorphism $h$; hence, it induces a rooted cluster
isomorphism
$$h:\mathcal{A}(\mu_{\pi_0(z_s)}\cdots\mu_{\pi_0(z_1)}(\Sigma_{\widehat{y_1y_2}}))\rightarrow\mathcal{A}(\Sigma'_{\widehat{y_1y_2}}).$$
Then the fact below follows.

\begin{Fac}\label{fact1}
$h\pi_0:\mathcal{A}(\Sigma')\rightarrow\mathcal{A}(\Sigma'_{\widehat{y_1y_2}})$
is the rooted cluster morphism given by gluing variables $y_1$
and $y_2$.
\end{Fac}

In fact, when $s=1$, from the seed isomorphism defined in Lemma
\ref{glue}, for $x\in
\mu_{z_1}(\widetilde{X})=(\widetilde{X}\setminus\{z_1\})\cup\{\mu_{z_1}^{\Sigma}(z_1)\}$,
we know
\begin{eqnarray*} h\pi_0(x)=
\left\{\begin{array}{lll} h(x)=x,&
\text{if} ~~x\in \widetilde{X}\setminus\{z_1, y_1, y_2\};\\
h(\overline y)=\overline y,&\text{if} ~~x=y_1\;\text{or}\;y_2;\\
h(\mu_{z_1}^{\Sigma_{\widehat{y_1y_2}}}(z_1))=\mu_{z_1}^{\Sigma}(z_1),&\text{if}
~~x=\mu_{z_1}^{\Sigma}(z_1).\end{array}\right.
\end{eqnarray*}

Hence, $h\pi_0$ is the rooted cluster morphism given
by gluing variables $y_1$ and $y_2$. Using the above mutation step
by step, we know that $h\pi_0$ is the rooted cluster
morphism given by gluing variables $y_1$ and $y_2$ for any $s$.

In particular, when $\pi_0$ is surjective, then $h\pi_0$ is
also surjective.

Now for any $x\in X'$, it is clear that the sequence $(x)$ of length
1 is a
$(h\pi_0,\Sigma',\Sigma'_{\widehat{y_{1}y_{2}}})$-biadmissible
sequence. By CM3, we have
$h\pi(\mu_{x}(x))=\mu_{h\pi_0(x)}(h\pi_0(x))$. Equivalently, we have
\begin{equation}\label{equ:11}
h\pi_0(\prod \limits_{b'_{xy}>0,y\in
\widetilde{X}'}y^{b'_{xy}}+\prod \limits_{b'_{xy}<0,y\in
\widetilde{X}'}y^{-b'_{xy}})=\prod \limits_{\overline
b'_{g\pi_0(x)y'}>0,y'\in \widetilde{\overline X}'}y^{\overline
b'_{g\pi_0(x)y'}}+\prod \limits_{\overline b'_{g\pi_0(x)y'}<0,y'\in
\widetilde{\overline X}'}y^{-\overline b'_{g\pi_0(x)y'}},
\end{equation}

Assuming $b'_{xy_{1}} b'_{xy_{2}}<0$, without loss of generality, we
may assume that $b'_{xy_1}>0$ and $b'_{xy_2}<0$. Therefore, by the
definition of $\pi_0$ and the construction of $h$, (\ref{equ:11}) becomes the equality:
\begin{equation} \label{con}
(\prod \limits_{b'_{xy}>0,y_1 \neq y\in
         \widetilde{X}'}y^{b'_{xy}})\overline y^{b'_{xy_1}}+(\prod \limits_{b'_{xy}<0,y_2\neq
         y\in \widetilde{X}'}y^{-b'_{xy}})\overline y^{-b'_{xy_2}}=\prod
         \limits_{\overline b'_{xy'}>0,y'\in \widetilde{\overline X}'}y^{\overline b'_{xy'}}+\prod
         \limits_{\overline b'_{xy'}<0,y'\in \widetilde{\overline X}'}y^{-\overline b'_{xy'}}.
 \end{equation}
By the  algebraical independence of $\widetilde{\overline X}'$ and the
skew symmetrizablity of $\widetilde{B}'$, the right hand side of
(\ref{con}) can not include a cluster monomial divisor, but the left
hand side of (\ref{con}) has $\overline
y^{min\{b'_{xy_1},-b'_{xy_2}\}}$ as its divisor. It is a contradiction. Hence, the
assumption $b'_{xy_{1}} b'_{xy_{2}}<0$ is not true and we have $b'_{xy_{1}} b'_{xy_{2}}\geq
0$.

\textbf{``If":}
By definition, trivially $\pi_0$ satisfies the conditions CM1 and
CM2 of root cluster morphism. Now we need to prove the CM3 for
$\pi_0$ with $\pi_0(\mathcal{A}(\Sigma))\subseteq
\mathcal{A}(\Sigma_{\widehat{y_{1}y_{2}}})$. First, we have

\textbf{Claim 1:} \\(a)~Any $\Sigma$-admissible sequence
$(z_{1},\cdots,z_{s})$ is
$(\pi_0,\Sigma,\Sigma_{\widehat{y_{1}y_{2}}})$-biadmissible.\\ (b)~
$\pi_0$ satisfies CM3, that is, for any
$(\pi_0,\Sigma,\Sigma_{\widehat{y_{1}y_{2}}})$-biadmissible sequence
$(z_{s},\cdots,z_{l})$ and $y\in \widetilde{X}$,
$$\pi_0(\mu_{z_s}\cdots\mu_{z_{1}}(y))=\mu_{\pi_0(z_s)\cdots\pi_0(z_{1})}(\pi_0(y)).$$

When $s=1$, since $z_{1}\in X$, $(z_{1})$ is a $\Sigma$-admissible
sequence and $\pi_0(z_1)=z_1$ by definition of $\pi_0$.
 Then, $(\pi_0(z_{1}))$ is
$\Sigma_{\widehat{y_{1}y_{2}}}$-admissible, and thus, $(z_{1})$ is
$(\pi_0,\Sigma,\Sigma_{\widehat{y_{1}y_{2}}})$-biadmissible.
Moreover, as $b_{xy_1}b_{xy_2}\geq 0$ for all $x\in X$, without loss
of generality assume that $b_{z_1y_1}\geq 0, b_{z_1y_2}\geq 0$. For
any $y\in \widetilde{X}$, we have

\[\begin{array}{ccl}\pi_0(\mu_{z_{1}}(y)) & = &
\left\{\begin{array}{lll} \frac{\prod\limits_{t\in \widetilde{X},
b_{z_1t}>0}\pi_0 (t^{b_{z_1t}})+\prod\limits_{t\in \widetilde{X},
b_{z_1t}<0}\pi_0(t^{-b_{z_1t}})}{\pi_0(z_1)},& \text{if}
~~y=z_1;\\\pi_0(y),& \text{if}~~ y\neq z_1;
\end{array}\right.\\
 & = & \left\{\begin{array}{lll} \frac{\prod\limits_{\overline y\neq t\in
\widetilde{\overline X}, b_{z_1t}>0} t^{b_{z_1t}}\overline
{y}^{b_{z_1y_1}+b_{z_1y_2}}+\prod\limits_{t\in \widetilde{X},
b_{z_1t}<0}t^{-b_{z_1t}}}{z_1},& \text{if} ~~y=z_1;\\\pi_0(y),&
\text{if}~~ y\neq z_1.
\end{array}\right.
\end{array}\]
On the other hand, we have
\[\begin{array}{ccl}\mu_{\pi_0(z_{1})}(\pi_0(y)) & = & \left\{\begin{array}{lll} \frac{\prod\limits_{t\in \widetilde{\overline X},\overline b_{\pi_0(z_1)t}>0}t^{\overline b_{\pi_0(z_1)t}}+\prod\limits_{t\in \widetilde{\overline X},\overline b_{\pi_0(z_1)t}<0}t^{-\overline b_{\pi_0(z_1)t}}}{\pi_0(z_1)},& \text{if}
~~\pi_0 (y)=\pi_0(z_1);\\\pi_0(y),& \text{if}~~ \pi_0(y)\neq
\pi_0(z_1);
\end{array}\right.\\
 & = & \left\{\begin{array}{lll} \frac{\prod\limits_{\overline t\in
\widetilde{\overline X}, \overline b_{z_1t}>0} t^{b_{z_1t}}\overline
{y}^{\overline b_{z_1\overline y}}+\prod\limits_{t\in \widetilde{X},
b_{z_1t}<0}t^{-b_{z_1t}}}{z_1},& \text{if} ~~y=z_1;\\\pi_0(y),&
\text{if}~~ y\neq z_1.
\end{array}\right.
\end{array}\]

As $\overline b_{z_1y}=b_{z_1y}$ for $y\neq \overline y$ and
$\overline b_{z_1y}=b_{z_1y_1}+b_{z_1z_2}$ for $y=\overline y$,
therefore, we get
$\pi_0(\mu_{z_{1}}(y))=\mu_{\pi_0(z_{1})}(\pi_0(y))$ for all $y\in
\widetilde{X}$.

Assume that Claim 1 holds for $s-1$. Similar to Fact
\ref{fact1}, it can be proved straightly that
\begin{equation}\label{ringfield}
\mathcal{A}(\mu_{z_1}\Sigma)\overset{\pi}{\rightarrow}\mathds F(\mu_{\pi
(z_1)}(\Sigma_{\widehat{y_1y_2}}))\overset{h}{\rightarrow}\mathds F((\mu_{z_1}\Sigma)_{\widehat{y_1y_2}})
\end{equation}
is a ring homomorphism obtained by gluing variables $y_1$ and $y_2$
with $h$. The field isomorphism induced by the seed isomorphism
$\mu_{\pi
(z_1)}(\Sigma_{\widehat{y_1y_2}})\overset{h}{\rightarrow}(\mu_{z_1}\Sigma)_{\widehat{y_1y_2}}$
in Lemma \ref{glue}, in particular, $\mathcal{A}(\mu_{\pi
(z_1)}(\Sigma_{\widehat{y_1y_2}}))\overset{h}{\rightarrow}\mathcal{A}((\mu_{z_1}\Sigma)_{\widehat{y_1y_2}})$
is a rooted cluster isomorphism. Therefore, by the induction
assumption, $(z_2,\cdots,z_s)$ is
$(h\pi_0,\mu_{z_1}(\Sigma),(\mu_{\pi_0(z_1)}(\Sigma))_{\widehat{y_{1}y_{2}}})$-biadmissible
and
$$h\pi_0(\mu_{z_s}\cdots\mu_{z_{2}}(y))=\mu_{h\pi_0(z_s)}\cdots \mu_{h\pi_0(z_{2})}(h\pi_0(y))$$
for all $y\in \mu_{z_1}(\widetilde{X})$, and since $h$ is a
seed isomorphism,  $(z_2,\cdots,z_s)$ is
$(\pi_0,\mu_{z_1}(\Sigma),\mu_{z_1}(\Sigma_{\widehat{y_{1}y_{2}}}))$-biadmissible
and then
$\pi_0(\mu_{z_s}\cdots\mu_{z_{2}}(y))=\mu_{\pi_0(z_s)}\cdots
\mu_{\pi_0(z_{2})}(\pi_0(y))$ for all $y\in
\mu_{z_1}(\widetilde{X})$. Therefore, $(z_1,\cdots,z_s)$ is
$(\pi_0,\Sigma,\Sigma_{\widehat{y_{1}y_{2}}})$-biadmissible and
$\pi_0(\mu_{z_s}\cdots\mu_{z_{1}}(y))=\mu_{\pi_0(z_s)}\cdots
\mu_{\pi_0(z_{1})}(\pi_0(y))$ for all $y\in \widetilde{X}$, that is,
$\pi_0$ satisfies CM3. The claim is proved.

Now, we prove that $\pi_0(\mathcal{A}(\Sigma))\subseteq
\mathcal{A}(\Sigma_{\widehat{y_{1}y_{2}}})$. For any cluster
variable $z$ in $\mathcal{A}(\Sigma)$, there exists a
$\Sigma$-admissible sequence $(z_1,\cdots,z_s)$ and $z_0\in
\widetilde{X}$ such that $\mu_{z_s}\cdots\mu_{z_1}(z_0)=z$.
According to Claim 1, $(z_1,\cdots,z_s)$ is
$(\pi_0,\Sigma,\Sigma_{\widehat{y_1y_2}})$-biadmissible, and by CM3,
we have
$\pi_0(\mu_{z_l}\cdots\mu_{z_{2}}(y))=\mu_{\pi_0(z_l)\cdots\pi_0(z_{2})}(\pi_0(y))\in
\mathcal{A}(\Sigma_{\widehat{y_1y_2}}).$ It follows that
$\pi_0(z)\in \mathcal{A}(\Sigma_{\widehat{y_1y_2}})$.

In summary, we have shown $\pi_0=\pi|_{\mathcal A(\Sigma)}$ to be a rooted cluster morphism.

Last, we need to prove that $\pi_0$ is surjective on
$\mathcal{A}(\Sigma_{\widehat{y_{1}y_{2}}})$. For this, we only need
to show that $\pi_0(\mathcal{A}(\Sigma))\supseteq
\mathcal{A}(\Sigma_{\widehat{y_{1}y_{2}}})$.

\textbf{Claim 2:}\; Any $\Sigma_{\widehat{y_1y_2}}$-admissible
sequence $(w_1,\cdots,w_t)$ can be lifted to a
$(\pi_0,\Sigma,\Sigma_{\widehat{y_{1}y_{2}}})$-biadmissible sequence
$(z_1,\cdots,z_s)$.

For $t=1$, since $w_1\in \overline{X}=X$, let $z_1=w_1$. Then,
$\pi_0(z_1)=w_1$ by definition of $\pi$ and $(z_1)$ is
$(\pi_0,\Sigma,\Sigma_{\widehat{y_{1}y_{2}}})$-biadmissible.

Assume that Claim 2 holds for $t-1$.

 By (\ref{ringfield}), $h\pi_0:
\mathcal{A}(\mu_{z_1}(\Sigma))\rightarrow\mathcal{A}(\mu_{z_1}\Sigma)_{\widehat{y_{1}y_{2}}}$
is the rooted cluster morphism given by gluing variables $y_1$ and
$y_2$, where $h:
\mathcal{A}(\mu_{w_1}(\Sigma_{\widehat{y_{1}y_{2}}}))\rightarrow\mathcal{A}((\mu_{z_1}\Sigma)_{\widehat{y_1y_2}})$
is the rooted cluster isomorphism induced by the seed isomorphism
$h$.

Therefore, by induction assumption, $(h(w_2),\cdots,h(w_s))$ can be
lifted to a
$(h\pi_0,\mu_{z_1}\Sigma,(\mu_{z_1}\Sigma)_{\widehat{y_{1}y_{2}}})$-biadmissible
sequence $(z_2,\cdots,z_s)$, as $h$ is a rooted cluster isomorphism.
Hence, we have $\pi_0(z_i)=w_i$ for $2\leq i\leq s$. Therefore,
$(w_1,\cdots,w_s)$ can be lifted to a
$(\pi_0,\Sigma,\Sigma_{\widehat{y_{1}y_{2}}})$-biadmissible sequence
$(z_1,\cdots,z_s)$. Hence, the claim follows.

For any cluster variable $w\in
\mathcal{A}(\Sigma_{\widehat{y_1y_2}})$, there exists a
$\Sigma_{\widehat{y_1y_2}}$-admissible sequence $(w_1,\cdots,w_s)$
and $w_0\in \widetilde{\overline{X}}$ such that
$w=\mu_{w_s}\cdots\mu_{w_1}(w_0)$. According to Claim 2,
$(w_1,\cdots,w_s)$ can be lifted to a
$(\pi_0,\Sigma,(\Sigma)_{\widehat{y_{1}y_{2}}})$-biadmissible
$(z_1,\cdots,z_s)$. It is clear that there exits $z_0\in
\widetilde{X}$ such that $\pi_0(z_0)=w_0$. Thus, by CM3, we
have
$w=\mu_{w_s}\cdots \mu_{w_{2}}(w_0)=\pi_0(\mu_{z_s}\cdots\mu_{z_{2}}(z_0)).$
Hence, $\pi_0$ is surjective follows.
\end{proof}

This proposition tells us the condition for two frozen variables $y_{1}$ and $y_{2}$ to be glued so as to make the canonical morphism $\pi_0$ to be surjective. So, we define two frozen variables $y_{1}$ and $y_{2}$ of a rooted cluster algebra $\mathcal A(\Sigma)$ to be {\bf glueable} if the condition of Proposition \ref{gluefrozens} is satisfied.  The following Lemma \ref{ind} (i) can be thought as another characterization for two frozen variables $y_{1}$ and $y_{2}$ in $\Sigma$  to be glueable via noncontractible rooted cluster morphisms.

\begin{Lem}\label{gluingfact}
 For $y_1\not=y_2\in \widetilde{X}$, if $f(y_1)=f(y_2)\in
\widetilde{X'}$ for a rooted cluster morphism
$f:\mathcal{A}(\Sigma)\rightarrow\mathcal{A}(\Sigma')$, then $y_1,
y_2\in X_{fr}$.
\end{Lem}
\begin{proof}
 If $y_1\not\in X_{fr}$, then $y_1\in X$. According to CM3, we
have $\mu_{f(y_1)}(f(y_2))=f(\mu_{y_1}(y_2))=f(y_2)$. Then,
$f(y_2)=\mu_{f(y_2)}(f(y_2))$. However, it is impossible due to the
definition of mutation.
\end{proof}

\begin{Lem}\label{twof}
For a surjective noncontractible rooted cluster morphism $g:\mathcal{A}(\Sigma)\rightarrow\mathcal{A}(\Sigma')$,

 (i)~ $g(X_{fr})\subseteq X'_{fr}$;

 (ii)~ $\#\widetilde{X}\geq\#\widetilde{X'}$;

  (iii)~
 if $\#\widetilde{X}\gneqq\#\widetilde{X'}$, then there exist $y_1, y_2\in X_{fr}$ with $y_1\not=y_2$ such that $g(y_1)=g(y_2)$ and

(iv)~ $\#\widetilde{X}=\#\widetilde{X'}$ if and only if $\#X=\#X'$ and $\#X_{fr}=\#X'_{fr}$.
\end{Lem}
\begin{proof}
 (i)~ Otherwise, there exists a $y\in X_{fr}$ such that $g(y)\in X'$. Then by (Lemma 3.1, \cite{ADS}), we have an $x\in X$ such that $g(y)=g(x)\in X'$. By CM3, we have $g(\mu_{x}(y))=\mu_{g(x)}(g(y))$, however, which is impossible since $g(\mu_{x}(y))=g(y)\neq \mu_{g(y)}(g(y))=\mu_{g(x)}(g(y))$.

 (ii)~  Since $g$ is noncontractible, we have $g(\widetilde X)\subseteq\widetilde{X'}$. Because $g$ is surjective, $g(\widetilde X)=\widetilde{X'}$.  So,  we have $\#\widetilde{X}\geq\#g(\widetilde X)=\#\widetilde{X'}$.

 (iii)~  By Lemma \ref{gluingfact}, the images of any two various exchangeable cluster variables under $g$ are always various. Thus, $\# X=\#g(X)\leq \# X'$. Moreover, since
 $\# X+\# X_{fr}=\# \widetilde{X}\gvertneqq \# \widetilde{X'}=\# X'+\# X'_{fr}$, so we have $\# X_{fr} \gvertneqq \#X'_{fr}$. Thus, there always exist $y_1, y_2\in X_{fr}$ such that $g(y_1)=g(y_2)$.

 (iv)~ By (Lemma 3.1, \cite{ADS}), $X'\subseteq g(X)$. Moreover, since $g$ is noncontractible, we have $g(X)\subseteq X'$. Thus, $g(X)=X'$. By Lemma \ref{gluingfact}, $g|_X$ is injective. Thus, $\# X=\#X'$.
\end{proof}

\begin{Lem}\label{ind} (i)~ For a seed $\Sigma$, two frozen variables $y_1$ and $y_2$ in $\Sigma$ are glueable if and only if there is another seed $\Sigma'$ and  a   noncontractible surjective rooted cluster morphism $f:\mathcal{A}(\Sigma)\rightarrow\mathcal{A}(\Sigma')$ such that $f(y_1)=f(y_2)$.

 (ii)~ In the situation of (i), the noncontractible
surjective rooted cluster morphism $f:\mathcal{A}(\Sigma)\rightarrow\mathcal{A}(\Sigma')$ can be decomposed into $f=h_1f_1$ for $f_1=\pi_0$, a surjective canonical morphism, and another surjective rooted cluster morphism $h_1:\mathcal{A}(\Sigma_{\widehat{y_1y_2}})\twoheadrightarrow\mathcal{A}(\Sigma')$.
\end{Lem}

\begin{proof}

(i)~ \;``Only If'':\; By the definition of ``glueable'' and Proposition \ref{gluefrozens}, it follows immediately that $\Sigma'=\Sigma_{\widehat{y_1y_2}}$ and $f=\pi_0$.

\;``If'':\; To show $\pi_0$ is a surjective rooted cluster morphism,
by Proposition \ref{gluefrozens}, it suffices to prove
that $c_{xy_1}c_{xy_2}\geq 0$ for any exchangeable variable $x$ in
any seed $\mu_{z_s}\cdots\mu_{z_1}(\Sigma)=(Y,\widetilde{C})$ obtained through mutations.

According to Claim 1 in the proof of Proposition \ref{inducedhom},
$(z_1,\cdots,z_s)$ is $(f,\Sigma,\Sigma')$-biadmissible. By Proposition
\ref{generate}, $f:\mu_{z_s}\cdots\mu_{z_1}(\Sigma)\rightarrow
\mu_{f(z_s)}\cdots\mu_{f(z_1)}(\Sigma')$ is a rooted cluster
morphism. It is clear that $f(Y)=\mu_{f(z_s)}\cdots\mu_{f(z_1)}(X')$
and
$f(\widetilde{Y})=\mu_{f(z_s)}\cdots\mu_{f(z_1)}(\widetilde{X}')$.
For any $x\in Y$, by CM3, we have $f(\mu_{x}(x))=\mu_{f(x)}(f(x))$.
Equivalently,
\begin{equation}\label{surj}
f(\frac{\prod\limits_{y\in \widetilde{Y},
c_{xy}>0}y^{c_{xy}}+\prod\limits_{y\in \widetilde{Y},
c_{xy}<0}y^{-c_{xy}}}{x})=\frac{\prod_1+\prod_2}{f(x)},
\end{equation}
 where $\prod_1$ and $\prod_2$ are coprime cluster monomials in
$\mu_{f(z_s)}\cdots\mu_{f(z_1)}(\Sigma')$; hence, there is no
non-trivial divisor in the right hand side of (\ref{surj}). Thus, it
can be seen that for the given $y_1, y_2$ in (1), we have
$c_{xy_1}c_{xy_2}\geq 0$, since otherwise, there is a non-trivial divisor
$f(y_1)^{min\{|c_{xy_1}|,|c_{xy_2}|\}}$. Hence, $\pi_0$ is a surjective.

(ii)~ From (i), $f_1=\pi_0$ is a surjective rooted cluster morphism.

Owing to $f(\widetilde{X})=\widetilde{X'}$, let
$h_1:\mathcal{A}(\Sigma_{\widehat{y_1y_2}})\rightarrow\mathds F(\Sigma')$
be the unique ring homomorphism by defining $h_1(x)=f(x)$
for any $x\in \widetilde{\overline X}\setminus\{\overline y\}$ and
$h_1(\overline y)=f(y_1)=f(y_2)$, where $\overline y$ is the gluing
variable in $\Sigma_{\widehat{y_1y_2}}$. We will see below that
$h_1:\mathcal{A}(\Sigma_{\widehat{y_1y_2}})\rightarrow\mathcal{A}(\Sigma')$
is a surjective rooted cluster morphism satisfying
$h_1(\overline X)=X'$ and $h_1(\widetilde{\overline
X})=\widetilde{X}'$.

It is clear that $h_1(\overline X)=X'$ and $h_1(\widetilde{\overline
X})=\widetilde{X}'$, so CM1 and CM2 hold for $h_1$. For any
$\Sigma_{\widehat{y_1y_2}}$-admissible sequence $(w_1,\cdots,w_s)$,
by Proposition \ref{gluefrozens}, there exists uniquely a
$(f_1,\Sigma,\Sigma_{\widehat{y_1y_2}})$-biadmissible sequence
$(z_1,\cdots,z_s)$ such that $f_1(z_i)=w_i$ for $i=1,\cdots s$.

\textbf{Claim:} $f(z_i)=h_1(w_i)$ for $i=1,\cdots s$ and
$h_1(\mu_{w_s}\cdots\mu_{w_1}(w))=f(\mu_{z_s}\cdots\mu_{z_1}(z))$
for all $w\in \widetilde{\overline X}$, where $z\in \widetilde{X}$
such that $f_1(z)=w$.

In case $s=1$, we have $z_1=w_1$; hence, $f(z_1)=h_1(z_1)$ by
the definition of $h_1$. Without loss of generality, we may assume
that $\overline b_{w_1\overline y}\geq 0$. Equivalently,
$b_{w_1y_1},b_{w_1y_2}\geq 0$. Hence, we have
\[\begin{array}{ccl}h_1(\mu_{w_{1}}(w))
& = & \left\{\begin{array}{lll} \frac{\prod\limits_{t\in
\widetilde{\overline X}, \overline b_{w_1t}>0}h_1 (t^{\overline
b_{w_1t}})+\prod\limits_{t\in \widetilde{\overline X}, \overline
b_{w_1t}<0}h_1(t^{-\overline b_{w_1t}})}{h_1(w_1)},& \text{if}
~~w=w_1;\\h_1(w),& \text{if}~~ w\neq w_1;
\end{array}\right.\\
& = & \left\{\begin{array}{lll} \frac{\prod\limits_{t\in
\widetilde{X}\setminus\{y_1,y_2\}, b_{w_1t}>0} f(t^{
b_{w_1t}})f(y_1)^{b_{w_1y_1}+b_{w_1y_2}}+\prod\limits_{t\in
\widetilde{X}, b_{w_1t}<0}h_1(t^{-b_{w_1t}})}{f(w_1)},& \text{if}
~~w=w_1;\\h_1(w)=f(w),& \text{if}~~ \overline y\neq w\neq w_1;\\
h_1(\overline y)=f(y_1)=f(y_2),& \text{if}~~ w=\overline
y.\end{array}\right.
\end{array}\]
On the other hand,
\[\begin{array}{ccl}f(\mu_{w_{1}}(z))
& = & \left\{\begin{array}{lll} \frac{\prod\limits_{t\in
\widetilde{X}, b_{w_1t}>0} f(t^{ b_{w_1t}})+\prod\limits_{t\in
\widetilde{X},  b_{w_1t}<0}f(t^{- b_{w_1t}})}{f(w_1)},& \text{if}
~~z=w_1;\\f(z),& \text{if}~~ z\neq w_1.
\end{array}\right.
\end{array}\]
Thus, for all $w\in \widetilde{ \overline X}$ and $z\in \widetilde{X}$
such that $h_1(z)=w$, we have $h_1(\mu_{w_1}(w))=f(\mu_{z_1}(z))$.

Assume that the Claim holds for $s-1$. By Fact \ref{fact1},
$hf_1:\mathcal{A}(\mu_{w_1}(\Sigma))\rightarrow
\mathcal{A}((\mu_{w_1}(\Sigma))_{\widehat{y_1y_2}})$ is the
surjective rooted cluster morphism obtained by gluing $y_1$ and
$y_2$, where
$h:\mathcal{A}(\mu_{z_1}(\Sigma_{\widehat{y_1y_2}}))\rightarrow
\mathcal{A}((\mu_{z_1}(\Sigma))_{\widehat{y_1y_2}})$ is the rooted
cluster isomorphism induced by the corresponding seed isomorphism
defined in Lemma \ref{glue}. Similarly, it can be seen that
$h_1h^{-1}:\mathcal{A}((\mu_{z_1}(\Sigma))_{\widehat{y_1y_2}})\rightarrow
\mathds F(\mu_{f(z_1)}(\Sigma'))$ is the ring homomorphism induced
by the surjective rooted cluster morphism
$f:\mathcal{A}(\mu_{z_1}(\Sigma))\rightarrow\mathcal{A}(\mu_{f(w_1)}(\Sigma'))$.

Since $(h(w_2),\cdots,h(w_s))$ is
$(\mu_{z_1}(\Sigma))_{\widehat{y_1y_2}}$-admissible, it can be
lifted to a $(hf_1,\mu_{z_1}(\Sigma),(\mu_{z_1}(\Sigma))_{\widehat{y_1y_2}})$-biadmissible
sequence $(z_2,\cdots,z_s)$. Therefore, by induction assumption, we
have $f(z_i)=(h_1h^{-1})(h(w_i))=h_1(w_i)$ for $2\leq i\leq s$ and
$$h_1(\mu_{w_s}\cdots\mu_{w_2}(w))=(h_1h^{-1})(\mu_{h(w_s)}\cdots\mu_{h(w_2)}(h(w)))=f(\mu_{z_s}\cdots\mu_{z_2}(z))$$
for all $w\in\mu_{w_1}(\widetilde{\overline X})$ and $z\in
\mu_{z_1}(\widetilde{X})$ such that $f_1(z)=w$. For $w\in
\widetilde{\overline X}$ and $z\in \widetilde{X}'$ such that $f_1(z)=w$,
we have $f_1(\mu_{w_1}(z))=\mu_{z_1}(f_1(z))$. Thus, for $1\leq i\leq
s$, $f(z_i)=h_1(w_i)$ and
$h_1(\mu_{w_s}\cdots\mu_{w_1}(w))=f(\mu_{z_s}\cdots\mu_{z_1}(z))$
for all $w\in \widetilde{\overline X}$, where $z\in \widetilde{X}$
such that $f_1(z)=w$. Hence, claim is proved.

Let $(w_1,\cdots,w_s)$ be $\Sigma_{\widehat{y_1y_2}}$-admissible, so
there exists $u_i\in \overline X=X$ such that
$w_i=\mu_{w_{i-1}}\cdots\mu_{w_1}(u_i)$ for each $1\leq i\leq s$.
Let $(z_1,\cdots,z_s)$ be the
$(f_1,\Sigma,\Sigma_{\widehat{y_1y_2}})$-biadmissible sequence such
that $f_1(z_i)=w_i$. Since $f(\widetilde{X})=\widetilde{X}'$, it is
easy to see that $(z_1,\cdots,z_s)$ is
$(f,\Sigma,\Sigma')$-biadmissible. According to the above claim, we have
$$h_1(w_i)=h_1(\mu_{w_{i-1}}\cdots\mu_{w_1}(u_i))=f(\mu_{z_{i-1}}\cdots\mu_{z_1}(u_i))=\mu_{f(z_{i-1})}\cdots\mu_{f(z_1)}(f(u_i)).$$
As $f(u_i)\in X'$, so $f_1(w_i)$ is exchangeable in
$\mu_{f(z_{i-1})}\cdots\mu_{f(z_1)}(\Sigma')$ for each $1\leq i\leq
s$. Hence, $(w_1,\cdots,w_s)$ is
$(h_1,\Sigma_{\widehat{y_1y_2}},\Sigma')$-biadmissible.

Now, we prove that $h_1$ satisfies CM3. From the discussion above,
any $\Sigma_{\widehat{y_1y_2}}$-admissible sequence
$(w_1,\cdots,w_s)$ is
$(h_1,\Sigma_{\widehat{y_1y_2}},\Sigma')$-biadmissible. For any
$w\in \widetilde{\overline X}$, let $z\in \widetilde{X}$ such that
$f_1(z)=w$. By the above claim, we have
$$h_1(\mu_{w_{s}}\cdots\mu_{w_1}(w))=f(\mu_{z_{s}}\cdots\mu_{z_1}(z))=\mu_{f(z_{s})}\cdots\mu_{f(z_1)}(f(z))=\mu_{h_1(w_{s})}\cdots\mu_{h_1(w_1)}(h_1(w)).$$
Thus, CM3 follows.

Second,
$h_1(\mathcal{A}(\Sigma_{\widehat{y_1y_2}}))\subseteq
\mathcal{A}(\Sigma')$ follows immediately from CM3 for
$h_1$ and the fact that any
$\Sigma_{\widehat{y_1y_2}}$-admissible sequence is
$(h_1,\Sigma_{\widehat{y_1y_2}},\Sigma')$-biadmissible.

Third, we show that $h_1$ is surjective. Any cluster variable
$v$ of $\mathcal{A}(\Sigma')$ can be written as $v=\mu_{v_s}\cdots\mu_{v_1}(v_0)$ for $v_0\in \Sigma'$. There
exists a $(f,\Sigma,\Sigma')$-biadmissible sequence
$(z_1,\cdots,z_s)$ and $z_0\in \widetilde{X}$ such that $f(z_i)=v_i$
for $i=0,\cdots,s$. According to Claim 1 in the proof of Proposition
\ref{gluefrozens}, $(z_1,\cdots,z_s)$ is
$(f_1,\Sigma,\Sigma_{\widehat{y_1y_2}})$-biadmissible. Therefore, by the above
Claim, we have
$$h_1(\mu_{f_1(z_{s})}\cdots\mu_{f_1(z_1)}(f_1(z_0)))=f(\mu_{z_{s}}\cdots\mu_{z_1}(z))=\mu_{v_s}\cdots\mu_{v_1}(v_0)=v.$$

Last, for each cluster variable of $\mathcal{A}(\Sigma)$,
$z=\mu_{z_s}\cdots\mu_{z_1}(z_0)$.
According to the above Claim,
$$h_1f_1(z)=h_1f_1(\mu_{z_s}\cdots\mu_{z_1}(z_0))=h_1(\mu_{f_1(z_{s})}\cdots\mu_{f_1(z_1)}(f_1(z_0)))=f(\mu_{z_{s}}\cdots\mu_{z_1}(z_0))=f(z).$$
Thus, $h_1f_1=f$ holds.
\end{proof}

Following these lemmas, we can get the
main conclusion.

\begin{Thm}\label{maindecomp}
 Let $f:\mathcal{A}(\Sigma)\rightarrow\mathcal{A}(\Sigma')$ be a noncontractible
surjective rooted cluster morphism and $s=\#\widetilde{X}-\#\widetilde{X'}$. Then, either $f=g_0$ or
 $f=g_sf_s\cdots f_2f_1, (s\geq 1)$ for a series of
surjective rooted cluster morphisms:
$$\mathcal A(\Sigma)\overset{f_1}{\rightarrow}\mathcal A(\Sigma_1)\overset{f_2}{\rightarrow}\cdots\overset{f_{s-1}}{\rightarrow}\mathcal A(\Sigma_{s-1})\overset{f_s}{\rightarrow}\mathcal A(\Sigma_s)\overset{g_s}{\rightarrow}\mathcal A(\Sigma'),$$
where $g_s$ is a rooted cluster isomorphism and each $f_i$ is just
the canonical morphism on $\mathcal A(\Sigma_{i-1})$ with $\Sigma_i$
the seed given from $\Sigma_{i-1}$ by gluing a pair of frozen cluster
variables with the same images under $f$ for $i=1,\cdots s$ and
$\Sigma=\Sigma_0$.
\end{Thm}

\begin{proof} By Lemma \ref{twof} (ii), $s=\#\widetilde{X}-\#\widetilde{X'}\geq 0$.

If $s=0$, then $\#\widetilde{X}=\#\widetilde{X'}$. Therefore, $f^{S}: \widetilde{X}\rightarrow
\widetilde{X'}$ is a bijective map. For any $x\in X$ and
$y\in \widetilde{X}$, by CM3, $f(\mu_{x}(x))=\mu_{f(x)}(f(x))$, that is
$$f(\frac{\prod\limits_{b_{xz>0,z\in\widetilde{X}}}z^{b_{xz}}+\prod\limits_{b_{xz<0},z\in\widetilde{X}}z^{-b_{xz}}}{x})=\frac{\prod\limits_{b'_{f(x)w>0,w\in\widetilde{X'}}}w^{b'_{f(x)w}}+\prod\limits_{b'_{f(x)w<0},w\in\widetilde{X'}}w^{-b'_{f(x)w}}}{f(x)}.$$ Comparing the exponent of $f(y)$ in this equality, we have $|b_{xy}|=|b'_{f(x)f(y)}|$. Hence, by Lemma \ref{twof} (iv) and Definition \ref{seediso}, $f^{S}$ is a seed isomorphism. According to Proposition \ref{basiclem}, $g_0=f$ is a rooted cluster isomorphism.

Fix a positive integer $t$ and assume that the result holds for any $s<t$. Consider the case $s=t$. By Lemma \ref{twof},
there exists $y_1$ and $y_2\in X_{fr}$ with $y_1\not=y_2$ such that $f(y_1)=f(y_2)$. According to Lemma \ref{ind}, $f=h_1f_1$ for the surjective canonical morphism
$f_1:\mathcal{A}(\Sigma)\rightarrow \mathcal{A}(\Sigma_{\widehat{y_1y_2}})$ and another surjective rooted cluster morphism $h_1:\mathcal{A}(\Sigma_{\widehat{y_1y_2}})\rightarrow \mathcal{A}(\Sigma')$.
Since $\#\widetilde{X}_{\widehat{y_1y_2}}=\#\widetilde{X}-1$, we have $\#\widetilde{X}_{\widehat{y_1y_2}}-\#\widetilde{X'}=t-1$. By the inductive assumption, $h_1=g_{t}f_t\cdots f_2$ for a rooted cluster isomorphism $g_t$, and for $2\leq i\leq t$, $f_i$ are
the surjective canonical morphisms on $\mathcal A(\Sigma_{i-1})$ with $\Sigma_i$
the seed given from $\Sigma_{i-1}$ by gluing a pair of cluster
variables with the same images under $f$ for $i=1,\cdots s$ and
$\Sigma=\Sigma_0$.
Thus, the result holds.
\end{proof}

\begin{Cor}
Let $f:\mathcal{A}(\Sigma)\rightarrow \mathcal{A}(\Sigma')$
be a surjective rooted cluster morphism with $0\not\in f(\widetilde{X})$
 and $\mathcal{A}(\Sigma)$ be acyclic. Denote $I_{1}=\{x\in \widetilde{X}|f(x)\in \mathbb{Z}\}$.
Then,

(i)~ A surjective rooted cluster morphism $f':\mathcal{A}(\Sigma)\rightarrow\mathcal{A}(\Sigma')$
can be uniquely constructed from $f$ satisfying $f'(x)=f(x)$ for $x\in \widetilde{X}\setminus I_1$ and $f'(x)=1$ for $x\in I_1$, which is called the {\bf unitary morphism} of $f$ on $I_1$;

(ii)~ $f'=f_0\sigma_{I_1,1}$ with two
surjective rooted cluster morphisms $\sigma_{I_1,1}:\mathcal{A}(\Sigma)\rightarrow
\mathcal{A}(\Sigma_{\emptyset, I_1})$ and
 $f_0:
\mathcal{A}(\Sigma_{\emptyset, I_1})\rightarrow\mathcal{A}(\Sigma')$,
where $\sigma_{I_1,1}$ is the specialisation of $\mathcal{A}(\Sigma)$ at $I_1$ and $f_0$ is the contraction of $f$, which can be decomposed into a composition of
 a rooted cluster isomorphism $g_s$ and a series of
some surjective canonical morphisms $f_1,\cdots,f_s$ obtained step-by-step by gluing a pair of frozen cluster
variables with the same images under $f$ for $i=1,\cdots s$ as given in Theorem \ref{maindecomp}.
\end{Cor}

\begin{proof} We first prove (ii). Since $\mathcal{A}(\Sigma)$ is acyclic,  by Theorem \ref{cluster quo.}, the specialisation $\sigma_{I_1,1}=\prod\limits_{x\in I_1}\sigma_{x,1}:\mathcal{A}(\Sigma)\rightarrow \mathcal{A}(\Sigma_{\emptyset,I_1})$ is a surjective rooted cluster morphism. Let $f_0$ be the contraction of $f$. Since $0\not\in f(\widetilde{X})$, we have $\Sigma^{(f)}=\Sigma_{\emptyset,I_1}$ by Remark \ref{c}. By Proposition \ref{inducedhom}, $f_0:\mathcal A(\Sigma_{\emptyset, I_1})\rightarrow\mathcal A(\Sigma')$ is a noncontractible surjective rooted cluster morphism.
By Theorem \ref{maindecomp}, $f_0$ can be decomposed as required.
Setting $f'=f_0\sigma_{I_1,1}$, then (ii) follows immediately. We will show that this $f'$ is just that required in (i).

In fact, $f'$ is a  surjective rooted cluster morphism, since $f_0$ and $\sigma_{I_1,1}$ are so.
We have $f'(x)=(f_0\sigma_{I_1,1})(x)=f_0(x)=f(x)$ for all $x\in \widetilde{X}\setminus I_1$ and $f'(x)=1$ for all $x\in I_1$ by the definitions of contraction and specialisation. The uniqueness of $f'$ follows from Lemma \ref{equel}.
\end{proof}

\begin{Ex} For two seeds $\Sigma_1$ and $\Sigma_2$, we have the rooted cluster algebras $\mathcal{A}(\Sigma_1\sqcup
\Sigma_2)$ and  $\mathcal{A}(\Sigma_1\coprod_{\Delta_1,\Delta_2}\Sigma_2)$ from the union seed and the amalgamated sum, respectively. Define a rooted cluster morphism $f:\mathcal{A}(\Sigma_1\sqcup
\Sigma_2)\rightarrow\mathcal{A}(\Sigma_1\coprod_{\Delta_1,\Delta_2}\Sigma_2)$ satisfying $f(x)=x$ for all $x\in (\widetilde{X}_1\cup \widetilde{X}_2)\setminus (\Delta_1\cup\Delta_2)$ and $f(y')=f(y'')=\bar y$ the image variable in $\Delta$ for a pair of corresponding variables $y'\in \Delta_1$ and $y''\in\Delta_2$. Trivially, $f$ is noncontractible and surjective. By Theorem \ref{maindecomp}, we can decompose $f$ into $f=g_sf_s\cdots f_2f_1$ for surjective canonical morphisms $f_i$ and a rooted cluster isomorphism $g_s$. In this case, $g_s=id_{\mathcal{A}(\Sigma_1\coprod_{\Delta_1,\Delta_2}\Sigma_2)}$. Assuming that all pairs of the corresponding variables from $\Delta_1$ and $\Delta_2$ are $(y'_1,y''_1),\cdots,(y'_s,y''_s)$, then $f_i$ can be obtained by gluing $y'_i$ and $y''_i$, i.e. $f(y_i')=f(y_i'')=\bar y_i$ for $i=1,\cdots,s$.
\end{Ex}

{\bf Acknowledgements:}\; This project is supported by the National
Natural Science Foundation of China (No.11271318 and  No.11571173) and the
Zhejiang Provincial Natural Science Foundation of China
(No.LZ13A010001).

\end{document}